\documentclass{amsart}

\usepackage{amsmath}
\usepackage{amsthm}
\usepackage{amssymb}
\usepackage{dsfont}
\usepackage{graphicx}
\usepackage{caption}
\usepackage{subcaption}
\usepackage{color}
\usepackage[english]{babel}
\usepackage{layout}
\usepackage{pdfsync}

\newtheorem{theorem}{Theorem}[section]

\newtheorem{lemma}[theorem]{Lemma}
\newtheorem{proposition}{Proposition}

\theoremstyle{definition}
\newtheorem{definition}[theorem]{Definition}
\newtheorem{remark}{Remark}

\newcommand{\LM}[1]{\hbox{\vrule width.2pt \vbox to#1pt{\vfill \hrule
			width#1pt
			height.2pt}}}
\def\LL{\!
	{\mathchoice{\>\LM7\>}{\>\LM7\>}{\,\LM5\,}{\,\LM{3.35}\,}}}

\newcommand{\dx}{\,\mathrm{d}x}

\newcommand{\e}{\varepsilon}

\newcommand{\lu}{{L}^1}
\newcommand{\lud}{\lu(D)}

\newcommand{\dist}{{\rm{dist}}}

\def\R{\mathbb{R}}
\def\Q{\mathbb{Q}}

\newcommand{\Hk}{{\mathcal H}^{k-1}}

\def\L{\mathcal{L}}
\def\w{\omega}
\def\Lw{\mathcal{L}(\w)}
\def\S{\mathcal{S}}
\def\NNw{\mathcal{N\!N}(\w)}
\def\Ard{\mathcal{A}^R(D)}
\def\a{\alpha}

\def\NNS{\mathcal{N\!N}(\mathcal{L})}
\def\rzd{r^{\prime}\mathbb{Z}^d}

\begin{document}

\author{Andrea Braides}
\address[Andrea Braides]{Dipartimento di Matematica, Universit\`{a} di Roma 'Tor Vergata', via della Ricerca Scientifica, 00133 Rome, Italy}
\email{braides@mat.uniroma2.it}

\author{Marco Cicalese}
\address[Marco Cicalese]{Zentrum Mathematik - M7, Technische Universit\"at M\"unchen, Boltzmannstrasse 3, 85747 Garching, Germany}
\email{cicalese@ma.tum.de}

\author{Matthias Ruf}
\address[Matthias Ruf]{Zentrum Mathematik - M7, Technische Universit\"at M\"unchen, Boltzmannstrasse 3, 85747 Garching, Germany}
\email{mruf@ma.tum.de}

\title{Continuum limit and stochastic homogenization of discrete ferromagnetic thin films}

\begin{abstract}
We study the discrete-to-continuum limit of ferromagnetic spin systems when the lattice spacing tends to zero. We assume that the atoms are part of a (maybe) non-periodic lattice close to a flat set in a lower dimensional space, typically a plate in three dimensions. Scaling the particle positions by a small parameter $\e>0$ we perform a $\Gamma$-convergence analysis of properly rescaled interfacial-type energies. We show that, up to subsequences, the energies converge to a surface integral defined on partitions of the flat space. In the second part of the paper we address the issue of stochastic homogenization in the case of random stationary lattices. A finer dependence of the homogenized energy on the average thickness of the random lattice is analyzed for an example of magnetic thin system obtained by a random deposition mechanism.  
\end{abstract}


\maketitle

\tableofcontents
\section{Introduction}
Polymeric magnets are known to be lighter and more flexible than conventional magnets. They can be easily manufactured to form thin films made of few layers and are currently considered as one of the main building blocks of the future generation of electronic devices. Under external magnetic fields they form Weiss domains whose wall energy is influenced by the thickness and the roughness of the film which in turn depends on the physical and chemical properties of the specific material at use. A fairly large amount of experimental results reconstruct the relation between film thickness and interfacial domain wall energy for different ferromagnetic materials (see \cite{KS} and references therein), but no rigorous explanation has appeared so far in this direction. Among the reasons for such an unsatisfactory analysis we single out one which has a geometric flavour: depositing magnetic particles on a substrate to obtain a thin films leads to disordered arrangements of particles and rough film surfaces which makes very difficult to formulate a right ansatz leading to study the correct (and simpler) continuum model. In this paper we look at this problem from a different perspective: we single out a simple Ising-type model for a thin film obtained by random deposition of magnetic particles on a flat substrate, for which the geometric part of the problem is still non trivial, and propose an ansatz-free variational analysis of such a film. Combining $\Gamma$-convergence and percolation theory we finally obtain a rigorous explanation of the relation between film thickness and domain-wall energy in some asymptotic regimes. 

\vspace{.5cm}

A simple way to model thin ferromagnetic polymeric materials at the micro scale first requires the definition of a polymeric matrix made of magnetic cells and then that of an interaction energy between those cells (see \cite{vollath2013} and reference therein for further details). The polymeric matrix of such a system can be seen as a random network whose nodes are the cross-linkers molecules of the 3-d polymeric magnet, which are supposed to entail the local magnetic properties of the system and to interact as magnetic elementary cells via a ferromagnetic Potts-type coupling. The system is supposed to be thin in the sense that the nodes of the matrix are within a small distance, of the order of the average distance between the nodes themselves, from a 2-d plane.  In presence of an external magnetic field or of proper boundary conditions, the ferromagnetic coupling induces the system to form mesoscopic Weiss domains, i.e. regions of constant magnetization.
 
\vspace{.5cm}

In this paper we aim at upscaling the system described above from its microscopic description to a mesoscopic one in a variational setting.
This consists in performing the limit of its energy as the average distance between the magnetic cells, which we denote by $\e$, goes to zero with respect to the macroscopic size of the system. Such a limit will have two main effects: it will allow us to describe the original discrete system as a continuum while at the same time it will reduce its dimension from three to two (or more in general from $d$ to $k$ with $2\leq k<d$). 

The discrete-to-continuum analysis in this paper is also part of a general study of the effects of discreteness in lattice systems on their macroscopic description. It is directly related to a series of papers describing the overall behaviour of spin energies \cite{CadlF,ABC,BP,BC,AlGe}. Moreover, discrete-to-continuum analyses for thin elastic objects in a deterministic setting have also been considered, e.g.~in \cite{ABC2,Sch,LPS}, and the behaviour of full-dimensional random lattices is dealt with in \cite{ACG2} (see also \cite{BLBL}). For dimension-reduction problems for continuum elastic objects we also refer to \cite{LDR,BFF}, the latter introducing a dimensionally reduced localization argument similar to the one use in the present paper.

\vspace{.5cm}
Using the same model as in \cite{ACR} we describe the polymeric matrix as a random network whose nodes $\L\subset \R^d$ form a thin {\sl admissible stochastic lattice}, meaning that the matrix is thin, i.e. there exist $k\in{\mathbb N}$ with $2\leq k<d$ and $M>0$ such that, identifying $\R^k$ with a linear subspace of $\R^d$,
\begin{equation*}
\dist(x,\R^k)\leq M\quad\hbox{ for all } x\in\L
\end{equation*} 
and that it is admissible according to the following standard definition (see \cite{ruelle} and also \cite{ACG2, BLBL} in the framework of rubber elasticity). 
We say that $\L$ is an admissible set of points if the following two requirements are satisfied:
\begin{itemize}
\item[(i)] there exists $r>0$ such that $|x-y|\geq r$ for all $x\neq y,\,x,y\in\L$,
\item[(ii)] there exists $R>0$ such that $\dist(x,\L)\leq R$ for all $x\in\R^k$.
\end{itemize}
Within this definition we may include `slices' of periodic lattices \cite{ABC2}, and also aperiodic geometries \cite{BrCaSo}.

Given a probability space $(\Omega,{\mathcal F},\mathbb P)$, a random variable 
${\mathcal L}:\Omega\to(\R^d)^{\mathbb{N}}$ is called an {\sl admissible stochastic lattice} if, uniformly with respect to $\omega\in\Omega$, ${\mathcal L}(\omega)$ is an admissible set of points. \\

We assume that the magnetization takes only finitely many values, that is to say we consider configurations $u:\e\L\to \S$ with a state-space $\S=\{s_1,\dots,s_q\}$ that we embed in the euclidean space $\R^{q}$. We have in mind the case of spin systems, where $u_i\in\{1,-1\}$. Note that even in that case it is sometime necessary to use a larger set of parameters $\S$ if frustration forces the formation of texture (see \cite{BC}). Note that if we have more than two parameters, we may have concentration phenomena of a third phase on the interfaces between two phases. A finer description of this phnomenon can be found in \cite{ACS}.

Associating a Voronoi tessellation $\mathcal{V}(\L)$ to the lattice $\L$, one introduces the set of nearest neighbours $\NNS$ as the set of those pairs of points in $\L$ whose Voronoi cells share a $(d-1)$-dimensional edge. 
This allows us to distinguish between long-range and short-range interactions introducing the following ($\L$-dependent) interactions
\begin{equation*}
f_{\e}(x,y,s_i,s_j)=\begin{cases}
f^{\e}_{nn}(x,y,s_i,s_j) &\mbox{if $(x,y)\in\mathcal{N\!N}(\L)$,}\\
f^{\e}_{lr}(x,y,s_i,s_j) &\mbox{otherwise,}
\end{cases}
\end{equation*}
that we assume to be non-negative and to satisfy the following coerciveness and growth assumptions.\\

\noindent \textbf{Hypothesis 1}\quad There exist $c>0$ and a decreasing function $J_{lr}:[0,+\infty)\rightarrow[0,+\infty)$ with 
\begin{equation*}
\int_{\R^k}J_{lr}(|x|)|x|\dx =J<+\infty
\end{equation*}
such that, for all $\e>0$, $x,y\in\R^d$ and $s_i,s_j\in\S$,
\begin{equation*}
c|s_i-s_j|\leq f_{nn}^{\e}(x,y,s_i,s_j)\leq J_{lr}(|x-y|)|s_i-s_j|,  \quad f_{lr}^{\e}(x,y,s_i,s_j)\leq J_{lr}(|x-y|)|s_i-s_j|. 
\end{equation*}
We remark that the decay of $J_{lr}$ is needed to control the effect of long-range interactions and we use the same bound for short-range interactions only to save notation.
\\
\hspace*{0,5cm}
Given $D\subset\R^k$ and denoted by $P_k:\R^d\to\R^k$ the projection onto $\R^k$, for a given configuration $u:\e\L\to\S$ we consider the energy per unit ($(k-1)$-dimensional) surface of $D$ to have the ferromagnetic Potts form (see also \cite{ABC, ACR, ACS, BC}) given by
\begin{equation*}
E_{\e}(u)=\sum_{\substack{\e x,\e y\in P_k^{-1}D}}\e^{k-1} f_{\e}(x,y,u(\e x),u(\e y)).
\end{equation*}

Since the sets $\e\L$ will eventually shrink to a $k$-dimensional set, we conveniently describe the system in terms of an {\sl average spin order parameter} $Pu:\e P_k\L\to{\rm co}(\S)$ defined on the $k$-dimensional set $\e P_{k}\L$ by 
\begin{equation*}
Pu(z):=\frac{1}{\#\left(P_k^{-1}(z)\cap\e\L\right)}\sum_{\e x\in P_k^{-1}(z)\cap\e\L}u(\e x).
\end{equation*}
We then embed the energies $E_{\e}$ in $L^{1}(D)$ by identifying $Pu$ with a function piecewise constant on the cells of the Voronoi tessellation of $P_k\L$, define the convergence $u_{\e}\to u$ in $D$ in the sense that the piecewise constant functions $Pu_{\e}$ converge to $u$ strongly in $L^1(D)$ and perform the $\Gamma$-convergence analysis with respect to this notion (see Section \ref{sec:model} for further details).

\smallskip

In Theorem \ref{mainthm1} we prove a compactness and integral representation result for the $\Gamma$-limit $E$ of $E_{\e}$, stating that, up to subsequences, this is finite only on $BV(D,\S)$, where it takes the integral form
\begin{equation*}
E(u)=\int_{S_u}\phi^{\w}(x,u^+,u^-,\nu_u)\,\mathrm{d}\mathcal{H}^{k-1}.
\end{equation*}
In this formula $S_u$ is the jump set of $u$, the functions $u^+$ and $u^-$ represent the traces on both sides of the jump set, $\nu_{u}\in S^{k-1}$ is the measure-theoretical normal to $S_u$ and $\mathcal{H}^{k-1}$ the $(k-1)$-dimensional Hausdorff measure. 
The function $\phi^{\w}$ is interpreted as the domain-wall interaction energy (per unit $(k-1)$-dimensional area) between Weiss domains. \\

The dependence of such an energy on the randomness of the lattice is studied in Section \ref{homogenization} in the context of stochastic homogenization assuming the thin random lattice to be stationary (or ergodic) in the directions of the flat subspace to which it is close to and the interaction coefficients to be invariant under translation in these directions. More precisely we assume that there exists a measure-preserving group action $(\tau_z)_{z\in\mathbb{Z}^k}$ on $\Omega$ such that, almost surely in $\Omega$, $\mathcal{L}(\tau_z\omega)=\Lw+z$ (if in addition $(\tau_z)_{z\in\mathbb{Z}^k}$ is ergodic, then also the lattice $\mathcal{L}$ is said to be ergodic) and the following structural assumption:\\

\textbf{Hypothesis 2}\quad There exist functions $f_{nn},f_{lr}:\R^k\times\R^{2(d-k)}\times\S^2\to [0,+\infty)$ such that, setting $\Delta_k(x,y)=(y_1-x_1,\dots,y_k-x_k,x_{k+1},y_{k+1},\dots,x_d,y_d)$, it holds
\begin{equation*}
f_{nn}^{\e}(x,y,s_i,s_j)=f_{nn}(\Delta_k(x,y),s_i,s_j),\quad f_{lr}^{\e}(x,y,s_i,s_j)=c_{lr}(\Delta_k(x,y),s_i,s_j).
\end{equation*}

In Theorem \ref{mainthm2} we prove that under Hypotheses $1$ and $2$ and assuming the stationarity (or ergodicity) in the sense specified above, the $\Gamma$-limit of $E_{\e}$ as $\e\to 0$ exists and is finite only on $BV(D,\S)$ where it takes the form
\begin{equation*}
E^{\omega}_{\text{hom}}(u)=
\int_{S_u} \phi^{\omega}_{\text{hom}}(u^+,u^-,\nu_u)\,\mathrm{d}\Hk. 
\end{equation*}
The energy density is given by an asymptotic homogenization formula which is averaged in the probability space under ergodicity assumptions on $\mathcal{L}$, thus turning the stochastic domain wall energy into a deterministic one.\\

The result is proved by the abstract methods of Gamma-convergence, first showing an abstract compactness result, and then giving an integral representation of the limit, as described in detail for deterministic bulk elastic thin films in \cite{BFF} (for other applications of this method in a discrete-to-continuum setting see e.g. \cite{AC,LDR2,BC}). The proof makes use of two main ingredients: the integral-representation theorem in \cite{BFLM} and the subadditive ergodic theorem by Ackoglu and Krengel in \cite{AkKr}. They are combined together following a scheme introduced in \cite{ACG2} in the context of random discrete systems with limit energy on Sobolev spaces (see also \cite{DMM}) and recently extended to sets of finite perimeter in \cite{ACR}. Section \ref{sect-vol} is devoted to extend the result above to the case of a volume constraint on the phases.  \\

An interesting issue in the theory of thin magnetic composite polymeric materials is the dependence of the domain wall energy on the random geometry of the polymer matrix. We devote the second part of the paper to this problem. We consider a specific model of a discrete system in which the state-space is $\S=\{\pm 1\}$ and the stochastic lattice is generated by the random deposition of magnetic particles on a two-dimensional flat substrate. For simplicity we limit ourselves to a simple deposition model with vertical order and suppose that the magnetic interactions have finite range. We are interested in the dependence of the domain wall energy on the average thickness of the thin film. Even though a complete picture would need a more extended treatment, thanks to percolation arguments we are able to attack the problem in the asymptotic cases when the thickness of the film is either small or very large. 

More specifically, we model the substrate (where the particles are deposited) by taking a two-dimensional deterministic lattice, which we choose for simplicity as $\mathcal{L}^0=\mathbb{Z}^2\times\{0\}$. We then consider an independent random field $\{X_i^p\}_{i\in\mathbb{Z}^3}$, where the $X_i^p$ are Bernoulli random variables with $\mathbb{P}(X_i^p=1)=p\in(0,1)$. For fixed $M\in\mathbb{N}$ we construct a random point set as follows:
\begin{equation*}
\mathcal{L}_p^M(\w):=\left\{(i_1,i_2,i_3)\in\mathbb{Z}^3:\;0\leq i_3\leq\sum_{k=1}^MX_{(i_1,i_2,k)}^p(\w)\right\},
\end{equation*}
which means that we successively deposit particles $M$ times independently onto the flat lattice $\mathcal{L}^0$ and stack them over each other (the point set constructed is stationary with respect to translations in $\mathbb{Z}^2$ and ergodic). Moreover, given $u:\e\mathcal{L}_p^M(\w)\to \{\pm 1\}$, we consider an energy of the form
\begin{equation*}
E^p_{\e,M}(\w)(u,A)=\sum_{\substack{x,y\in\mathcal{L}_p^M(\w)\\ \e P_2(x),\e P_2(y)\in A}}\e \,c(x-y)|u(\e x)-u(\e y)|,
\end{equation*}
where the interaction constant $c:\R^3\to [0,+\infty)$ is finite range, bounded above and coercive on nearest-neighbours, so that the Hypotheses 1 and 2 above are satisfied. As a result Theorem \ref{mainthm2} guarantees the existence of a surface tension, say $\phi^p_{\text{hom}}(M;\nu)$ given by an asymptotic cell 
formula. \\

The main issue now is the dependence of $\phi^p_{\text{hom}}(M;\nu)$ on $p$ and $M$. \\

A first result in this direction is proved in Proposition \ref{layerdependence} where we show that, for every direction $\nu\in S^{1}$, the wall energy density is linear in the average thickness $pM$ as $M\to +\infty$, that is 
\begin{equation}\label{intro-average}
\lim_{M\to +\infty}\frac{\phi^p_{\text{hom}}(M;\nu)}{pM}=\phi^{1}(\nu),
\end{equation}
with $\phi^{1}(\nu)$ given in Lemma \ref{auxlemma} being the wall energy per unit thickness of the deterministic problem obtained for $p=1$. \\ 
A second and more delicate result is contained in Theorem \ref{almostpercolation} and concerns a percolation type phenomenon which can be roughly stated as follows: When the deposition probability $p$ is sufficiently low (below a certain critical percolation threshold) the domain wall energy is zero for $M$ small enough. At this stage it is worth noticing that our energy accounts for the interactions between the deposited particles and the substrate. On one hand this assumption might be questionable from a physical point of view in the case one assumes to grow thin films on neutral media, thus expecting the properties of the film to be independent of the substrate. On the other hand removing such an interaction leads to a dilute model similar the one considered in \cite{Dilute}. An adaption of this analysis would require a lot of additional work like the extension of fine percolation results to the (range 1)-dependent case which goes far beyond the scopes of the present paper (see also Remark \ref{perc}). We prove the percolation result for nearest-neighbour positive interactions. Setting the interaction with the substrate to be $\eta>0$ we can prove that if $p<1-p_{site}$ (here $p_{site}$ is the critical site percolation threshold in $\mathbb{Z}^2$), the limit energy $\phi^{p,\eta}_{\text{hom}}(M;\nu)$ is bounded above (up to a constant) by $\eta$ for $M$ small enough. This result suggests the absence of a positive domain wall energy in the thin film on a neutral substrate ($\eta=0$ case). In the limit as $M$ diverges \eqref{intro-average} holds with $\phi^{p,\eta}_{\text{hom}}(M;\nu)$, which is independent of $\eta$, thus showing that the contribution of the first layer does not affect the asymptotic average domain wall energy as expected. The proof of these results needs the extension to the dimension reduction framework of a result by Caffarelli-de la Lave \cite{CadlF} about the existence of plane-like minimizers for discrete systems subject to periodic Ising type interactions at the surface scaling. This is contained in the appendix to the paper.

As a final remark, we mention that we prove all our results in the case when the flat object is at least two-dimensional. Most of the results can be extended to one-dimensional objects (with the proof being much simpler), except the ones contained in Section \ref{sect-vol} which fail in dimension one as can be seen by simple examples and the percolation-type phenomenon in Section \ref{sec-example} as no percolation can occur in (essentially) one-dimensional lattices.

\section{Modeling discrete disordered thin sets and spin systems}\label{sec:model}
This section is devoted to the precise description of the model we are going to study. We start with the notation we are going to use in the sequel.
\\
\hspace*{0,5cm}
As we are concerned with dimension-reduction issues, there will be two geometric dimensions $k$ and $d$ with $2\leq k<d$. Given a measurable set $A\subset\R^k$ we denote by $|A|$ its $k$-dimensional Lebesgue measure, while more generally $\mathcal{H}^{m}(A)$ stands for the $m$-dimensional Hausdorff measure. We denote by $\mathds{1}_A$ the {\it characteristic function} of $A$. Given $x\in\R^k$ and $r>0$, $B_r(x)$ is the open ball around $x$ with radius $r$. By $|x|$ we denote the usual euclidean norm of $x$. Moreover, we set $\text{d}_{\mathcal{H}}(A,B)$ the Hausdorff distance between the sets $A$ and $B$ and $\text{dim}_{\mathcal{H}}(A)$ the Hausdorff dimension of $A$. If it is clear from the context we will use the same notation as above also in $\R^d$ (otherwise we will indicate the dimension by sub/superscript indices). Given an open set $D\subset\R^k$ we denote by $\mathcal{A}(D)$ the family of all bounded open subsets of $D$ and by $\Ard$ the family of those sets in $\mathcal{A}(D)$ with Lipschitz boundary. Given a unit vector $\nu\in S^{k-1}$, let $\nu=\nu_1,\dots,\nu_k$ be a orthonormal basis. We define the open cube in $\R^k$
\begin{equation*}
Q_{\nu}=\Bigl\{x\in\R^k:\;|\langle x,\nu_i\rangle|<\frac{1}{2}\quad\hbox{ for all }i\Bigr\},
\end{equation*}
and, for $x\in\R^k,\rho>0$, we set $Q_{\nu}(x,\rho):=x+\rho\, Q_{\nu}$. We call $\nu\in S^{k-1}$ a rational direction if $\nu\in \Q^{k}$. We denote by $P_k:\R^d\to\R^k$ the projection onto $\R^k$. 

For $q\in\mathbb{N}$ we let $BV(D,\R^{q})$ be the space of $\R^{q}$-valued {\it functions of bounded variation}; that are, those functions $u\in L^1(D,\R^{q})$ such that their distributional derivate $Du$ is a matrix-valued Radon measure. Given a set $\S\subset\R^{q}$, we denote by $BV(A,\S)$ the space of those functions $u\in BV(A,\R^{q})$ such that $u(x)\in \S$ almost everywhere. If $S$ is a finite set, then the distributional derivative of $u$ can be represented on any Borel set $B\subset D$ as $Du(B)=\int_{B\cap S_{u}}(u^{+}(x)-u^{-}(x))\otimes\nu_{u}(x)\,\mathrm{d}\mathcal{H}^{k-1}(x)$, for a countably $\mathcal{H}^{k-1}$-rectifiable set $S_{u}$ in $D$ which coincides $\mathcal{H}^{k-1}$-almost everywhere with the complement in $D$ of the Lebesgue points of $u$. Moreover $\nu_{u}(x)$ is a unit normal to $S_{u}$, defined for $\mathcal{H}^{k-1}$-almost every $x$ and $u^{+}(x),\, u^{-}(x)$ are the traces of $u$ on both sides of $S_{u}$. Here the symbol $\otimes$ stands for the tensorial product of vectors, that is for any $a,b\in\R^{k}$ $(a\otimes b)_{ij}:=a_{i}b_{j}$. A measurable set $B$ is said to have finite perimeter in $D$ if its characteristic function belongs to $BV(D)$. We refer the reader to \cite{AFP} for an introduction to functions of bounded variation. The letter $C$ stands for a generic positive constant that may change every time it appears.
\\
\hspace*{0,5cm}
We want to describe (possibly non-periodic) particle systems, where the particles themselves are located very close to a lower-dimensional linear subspace. To this end we make the following assumptions: Let $\mathcal{L}\subset\R^d$ be a countable set. We assume that there exists $M>0$ such that, after identifying $\R^k\sim\R^k\times\{0\}^{d-k}$, we have
\begin{equation}\label{thickness}
\dist(x,\R^k)\leq M\quad\hbox{ for all } x\in\mathcal{L}.
\end{equation} 
Moreover, adapting ideas from \cite{ACG2,ACR,BLBL} we assume that the point set is regular in the following sense:

\begin{definition}\label{flatadmissible}
A countable set $\mathcal{L}\subset\R^d$ is a {\em thin admissible lattice} if (\ref{thickness}) holds and

\begin{itemize}
\item[(i)] there exists $r>0$ such that $|x-y|\geq r$ for all $x\neq y,\,x,y\in\mathcal{L}$,
\item[(ii)]  there exists $R>0$ such that $\dist(x,\mathcal{L})\leq R$ for all $x\in\R^d$.
\end{itemize}
\end{definition}

We associate to such a lattice a truncated Voronoi tessellation $\mathcal{V}(\mathcal{L})$, where the corresponding $d$-dimensional cells $\mathcal{C}\in\mathcal{V}(\mathcal{L})$ are defined by
\begin{equation*}
\mathcal{C}(x):=\{z\in\R^k\times[-2M,2M]^{d-k}:\;|z-x|\leq |z-x^{\prime}|\hbox{ for all } x^{\prime}\in\mathcal{L}\},
\end{equation*}
and we introduce the set of nearest neighbours accordingly by setting
\begin{equation*}
\NNS:=\{(x,y)\in\L^2:\;\text{dim}_{\mathcal{H}}(\mathcal{C}(x)\cap \mathcal{C}(y))=d-1\}.
\end{equation*}
As usual in the passage from atomistic to continuum theories we scale the point set $\L$ by a small parameter $\e>0$. We assume that the magnetization of the particles takes values in a finite set $\S=\{s_1,\dots,s_q\}\subset\R^{q}$. Fix a $k$-dimensional reference set $D\in\mathcal{A}^R(\R^k)$. Given $A\in\Ard$ and $u:\e\L\to\S$, we consider a localized (on $A$) pairwise interaction energy 
\begin{equation*}
E_{\e}(u,A)=\sum_{\substack{\e x,\e y\in P_k^{-1}A}}\e^{k-1} f_{\e}(x,y,u(\e x),u(\e y)),
\end{equation*}
where the ($\L$-dependent) interactions distinguish between long and short-range interactions and are of the form
\begin{equation*}
f_{\e}(x,y,s_i,s_j)=\begin{cases}
f^{\e}_{nn}(x,y,s_i,s_j) &\mbox{if $(x,y)\in\mathcal{N\!N}(\L)$,}\\
f^{\e}_{lr}(x,y,s_i,s_j) &\mbox{otherwise.}
\end{cases}
\end{equation*}
For our analysis we make the following assumptions on the measurable functions $f_{nn}^{\e},f_{lr}^{\e}:\R^d\times\R^d\times \S^2\to [0,+\infty)$:\\

\medskip
\noindent \textbf{Hypothesis 1 }There exist $c>0$ and a decreasing function $J_{lr}:[0,+\infty)\rightarrow[0,+\infty)$ with 
\begin{equation*}
\int_{\R^k}J_{lr}(|x|)|x|\dx =J<+\infty
\end{equation*}
such that, for all $\e>0$, $x,y\in\R^d$ and $s_i,s_j\in\S$,
\begin{equation*}
c\leq c_{nn}^{\e}(x,y)\leq J_{lr}(|x-y|),  \quad c_{lr}^{\e}(x,y)\leq J_{lr}(|x-y|). 
\end{equation*}
\hspace*{0,5cm}
Since the sets $\e\L$ shrink to a $k$-dimensional set as $\e$ vanishes, we want to define a convergence of discrete variables on shrinking domains.To that end, denoting by ${\rm co}(\S)$ the convex hull of $\S$, we define the averaged and projected spin variable $Pu:\e P_k\L\to {\rm co}(\S)$ via
\begin{equation*}
Pu(\e z):=\frac{1}{\#\left(P_k^{-1}(z)\cap\L\right)}\sum_{x\in P_k^{-1}(z)\cap\L}u(\e x).
\end{equation*}
The projected lattice $P_k\L\subset\R^k$ inherits property (ii) from Definition \ref{flatadmissible}, but (i) might fail after projection. Nevertheless, due to (\ref{thickness}) the projected lattice is still locally finite and the following uniform bound on the number of points holds true: there exists a constant $C=C_{\L}>0$ such that, given a set $A\in\mathcal{A}(D)$ with $|\partial A|=0$, we have
\begin{equation}\label{localratio}
\e^k\#\{\e z\in \e P_k\L\cap A\}\leq C|A|
\end{equation}
for $\e$ small enough. 
We now associate the corresponding $k$-dimensional Voronoi tessellation $\mathcal{V}(P_k\L)=\{\mathcal{C}_k(z)\}$ in $\R^k$ to the lattice $ P_k\L$ and we identify $Pu$ with a piecewise-constant function belonging to the class 
\begin{equation*}
\mathcal{PC}_{\e}(\L):=\{v:\R^k\to {\rm co}(\S):\;v_{|\e\mathcal{C}_k(z)}\text{ is constant for all } z\in P_k\L\}
\end{equation*}
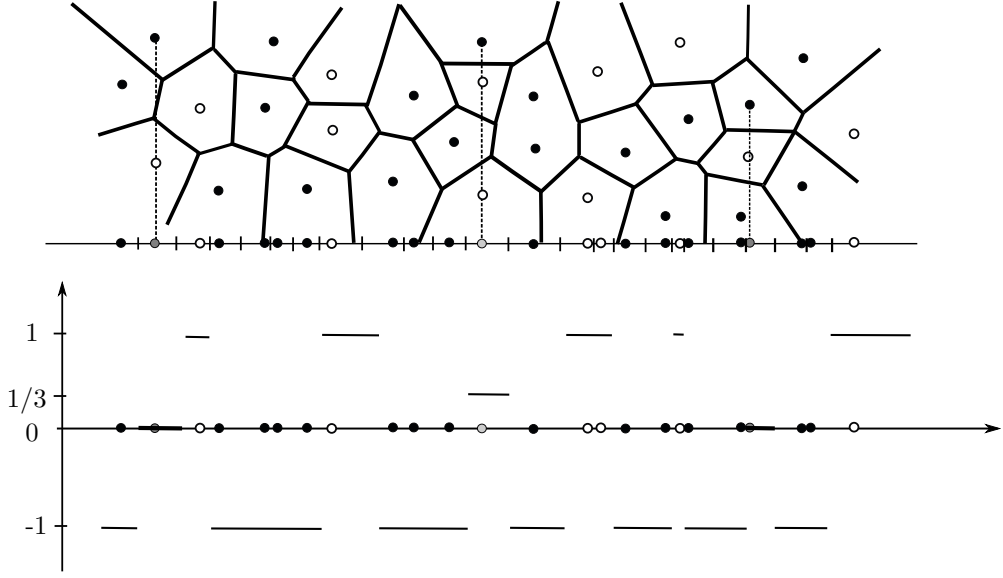
\begin{figure}
\centerline{\includegraphics [width=5in]{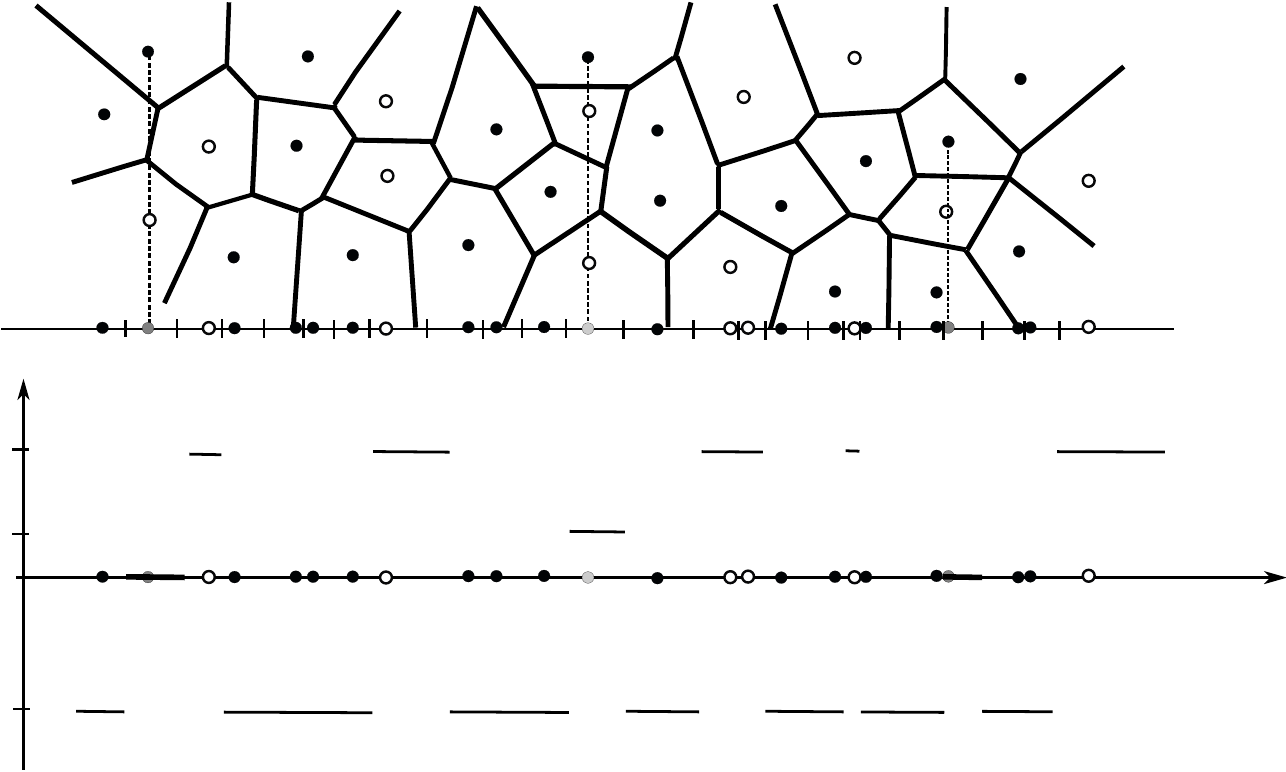}}
\caption{Construction of the piecewise-constant interpolation for $d=2$, $k=1$ and $\S=\{\pm 1\}$.} \label{voronoi}
\begin{picture}(0,0)
\put(-188,59){-1}\put(-188,94){0}\put(-195,107){1/3}\put(-188,132){1}
\end{picture}
\end{figure}
Note that we can embed $\mathcal{PC}_{\e}(\L)$ in $L^1(D)$ since the intersection of two Voronoi cells always has zero $k$-dimensional Lebesgue measure. 

For the sake of illustration, in Figure \ref{voronoi} we picture the construction in the simple case $d=2$, $k=1$ and $\S=\{\pm 1\}$. In the picture above, we draw a portion of the truncated Voronoi diagram of the lattice $\L$ represented by the dots, black for $u=-1$ and white for $u=+1$. At the bottom of the Voronoi diagram we include the projected points $P_1\L$ and the values of the variable $Pu\in[-1,1]$ (range reflected by the grey scale in the figure). The dashed lines indicate the exceptional set of projection points where $|Pu|\neq 1$. In the picture below, it is represented the piecewise-constant function on the Voronoi intervals subordinated to $P_1\L$. 

To deal with convergence of sequences $u_{\e}:\e\L\to\S$, we adopt the idea of \cite{BrCaSo}. We will see in Section \ref{sect-vol} that this notion of convergence is indeed meaningful for variational problems in a random environment.

\begin{definition}\label{defconv}
Let $A\in\mathcal{A}(D)$. We say that a sequence $u_{\e}:\e\L\to\S$ {\em converges} in $A$ to $u:A\to\R^q$ if the piecewise-constant functions $Pu_{\e}$ converge to $u$ strongly in $L^1(A)$.
\end{definition}

\hspace*{0,5cm}
For our variational analysis we also introduce the lower and upper $\Gamma$-limits $E^{\prime},E^{\prime\prime}:L^1(D,\R^q)\times\Ard\to [0,+\infty]$ setting
\begin{align*}
E^{\prime}(u,A)&:=\inf\Bigl\{\liminf_{\e\to 0}E_{\e}(u_{\e},A):\;u_{\e}\to u \text{ in }D\Bigr\},\\
E^{\prime\prime}(u,A)&:=\inf\Bigl\{\limsup_{\e\to 0}E_{\e}(u_{\e},A):\;u_{\e}\to u\text{ in }D\Bigr\}.
\end{align*}

\begin{remark}\label{almostgamma}
The functionals $E^{\prime},E^{\prime\prime}$ are not $\Gamma$-lower/upper limits in the usual sense since they are not defined on the same space as $E_{\e}$. However, if we define the functionals $\tilde{E}_{\e}:L^1(D,\R^q)\times\Ard\to [0,+\infty]$ as
\begin{equation*}
\tilde{E}_{\e}(u,A):=
\begin{cases}
\inf_{v}E_{\e}(v,A)&\mbox{if $u=Pv$ for some $v:\e\L\to\S$,}\\
+\infty&\mbox{otherwise,}
\end{cases}
\end{equation*} 
then $E^{\prime},E^{\prime\prime}$ agree with the $\Gamma$-lower/upper limit of $\tilde{E_{\e}}$ in the strong $L^1(D)$-topology. Therefore we will refer to the equality of $E^{\prime}$ and $E^{\prime\prime}$ as $\Gamma$-convergence. Moreover, one can show that
\begin{align*}
E^{\prime}(u,A)&=\inf\Bigl\{\liminf_{\e\to 0}E_{\e}(u_{\e},A):\;u_{\e}\to u \text{ in }A\Bigr\},\\
E^{\prime\prime}(u,A)&=\inf\Bigl\{\limsup_{\e\to 0}E_{\e}(u_{\e},A):\;u_{\e}\to u\text{ in }A\Bigr\}.
\end{align*}
By the properties of $\Gamma$-convergence this implies that both functionals $u\mapsto E^{\prime}(u,A)$ and $u\mapsto E^{\prime\prime}(u,A)$ are $L^1(A)$-lower semicontinuous and hence local in the sense of Theorem \ref{represent} (ii).
\end{remark}

\hspace*{0,5cm}
We now prove several properties of the convergence introduced in Definition \ref{defconv}. We start with an equi-coercivity property.

\begin{lemma}\label{compactness}
Assume Hypothesis 1 holds. Let $A\in\mathcal{A}(D)$ and let $u_{\e}:\e\L\to\S$ be such that
\begin{equation*}
\sup_{\e}E_{\e}(u_{\e},A)<+\infty.
\end{equation*}
Then, up to subsequences, the functions $Pu_{\e}$ converge strongly in $L^1(A)$ to some $u\in BV(A,\S)$.
\end{lemma}

\begin{proof}
Fix $A^{\prime}\subset\subset A$ such that $A^{\prime}\in\Ard$. We start by estimating the measure of the set $\{Pu_{\e}\notin\S\}\cap A^{\prime}$. Note that if $Pu_{\e}(\e z)\notin\S$ for some $z\in P_k\L$ such that $\e\,\mathcal{C}_k(z)\cap A^{\prime}\neq\emptyset$, then there exist $x_1,x_2\in P_k^{-1}(z)\cap\L$ such that $u_{\e}(\e x_1)\neq u_{\e}(\e x_2)$. As a preliminary step we show that we can find a path of nearest neighbours in $\L$ joining $x_1$ and $x_2$; that is, a finite collection of points $\{x^1,\dots,x^m\}\subset\L$ such that $x^1=x_1$ and $x^m=x_2$ and $(x^i,x^{i+1})\in\NNS$ for all $i=1,...,m-1$.  Moreover this path will be chosen such that it does not vary too much from the segment between $x_1$ and $x_2$. To this end, fix $0<\delta<<1$ and consider the collection of segments
\begin{equation}\label{def-ray}
\mathcal{G}_{\delta}(x_1,x_2)=\{x+\lambda (x_2-x_1):\;x\in B_{\delta}(x_1),\,0\leq\lambda\leq 1\}.
\end{equation}
We argue that there exists a segment $g^*=\{x^*+\lambda(x_2-x_1):0\leq\lambda\leq 1\}\subset \mathcal{G}_{\delta}$ satisfying the following implication:
\begin{equation}\label{onlynncells}
g^*\cap \mathcal{C}(x)\cap\mathcal{C}(x^{\prime})\neq\emptyset\quad \Rightarrow\quad (x,x^{\prime})\in\NNS. 
\end{equation}
Indeed, assume by contradiction that the implication is false for all $x^*\in B_{\delta}(x_1)$. Since the number of $d$-dimensional Voronoi cells $\mathcal{C}(x)\in\mathcal{V}(\L)$ such that $\mathcal{C}(x)\cap \mathcal{G}_{\delta}\neq\emptyset$ is uniformly bounded, we could then find finitely many Voronoi facets of dimension less than $d-1$ whose projection onto the hyperplane containing $x_1$ and orthogonal to $x_2-x_1$ covers a $d-1$-dimensional set. Since projections onto hyperplanes are Lipschitz continuous, we obtain a contradiction.
\\
\hspace*{0.5cm}
The path connecting $x_1$ and $x_2$ is then given by the set $G(x_1,x_2):=\{x\in\L:\;g^*\cap \mathcal{C}(x)\neq\emptyset\}$, provided that $\delta$ is small enough. Observe that there exist $x,y\in G(x_1,x_2)$ such that $(x,y)\in\NNS$ and $u_{\e}(\e x)\neq u_{\e}(\e y)$. From the coercivity assumption in Hypothesis 1, we thus deduce that each path contributes to the energy. Moreover, by (\ref{thickness}) and the local construction of the paths, for any pair $(x,y)\in\NNS$ it holds that
\begin{equation*}
\#\{z\in P_k\L:\;G(x_1,x_2)\cap \{x,y\}\neq \emptyset\}\leq C.
\end{equation*} 
From these two facts we infer that
\begin{equation}\label{numberest}
\e^{k-1}\#\{\e z:\;\e\mathcal{C}_k(z)\cap A^{\prime}\neq\emptyset,\,Pu_{\e}(\e z)\notin\S,\}\leq CE_{\e}(u_{\e},A)\leq C,
\end{equation} 
where we have used that $\e G(x_1,x_2)\subset (P_k^{-1}A)\cap\e\L$ for $\e$ small enough. Since the measure of a Voronoi cell in $P_k\L$ can be bounded uniformly by a constant, by rescaling we deduce that
\begin{equation}\label{only1}
\left|\{Pu_{\e}\notin\S\}\cap A^{\prime}\right|\leq C\e.
\end{equation}
We continue bounding the total variation $|DPu_{\e}|(A^{\prime})$. Since $Pu_{\e}$ is equibounded and piecewise constant, it is enough to provide a bound for $\Hk(S_{Pu_{\e}}\cap A^{\prime})$. Note that the jump set $S_{Pu_{\e}}$ is contained in the facets of the Voronoi cells of the lattice $\e P_k\L$. Since $\L$ is thin admissible in the sense of Definition \ref{flatadmissible} and property (ii) is preserved by projection, for each such facet $F$ it holds that
\begin{equation*}
\Hk(F)\leq C\e^{k-1}.
\end{equation*}
For $\e$ small enough, we conclude that
\begin{align*}
\Hk(S_{Pu_{\e}}\cap A^{\prime})\leq C\e^{k-1}\#\{(z,z^{\prime})\in\mathcal{N\!N}(P_k\L):\;Pu_{\e}(\e z)\neq Pu_{\e}(\e z^{\prime}),\,\e z,\e z^{\prime}\in A^{\prime}+B_{R\e}(0)\}.
\end{align*}
Given $\e z,\e z^{\prime}\in A^{\prime}+B_{R\e}(0)$ such that $(z,z^{\prime})\in\mathcal{N\!N}(P_k\L)$ and $Pu_{\e}(\e z)\neq Pu_{\e}(\e z^{\prime})$, again we may find a path of nearest neighbours $G(z,z^{\prime})=\{x^0\in P_k^{-1}(z),x^1,...,x^m\in P_k^{-1}(z^{\prime})\}$ with $u_{\e}(\e x^0)\neq u_{\e}(\e x^m)$ and the paths are local in the sense that 
\begin{equation*}
\#\{(z,z^{\prime})\in\mathcal{N\!N}(P_k\L):\;G(z,z^{\prime})\cap\{x,y\}\neq\emptyset\}\leq C
\end{equation*}
for all $(x,y)\in\NNS$. Reasoning as in the first part of the proof we find that
\begin{align*}
\e^{k-1}\#\{(z,z^{\prime})\in\mathcal{N\!N}(P_k\L):\;Pu_{\e}(\e z)\neq Pu_{\e}(\e z^{\prime}),\,\e z,\e z^{\prime}\in A^{\prime}+B_{R\e}(0)\}&\leq CE_{\e}(u_{\e},A)\leq C.
\end{align*}
By well-known compactness properties of $BV$-functions (see for example \cite[Corollary 3.49]{AFP}) and (\ref{only1}), there exists a subsequence (not relabeled) such that $Pu_{\e}\to u$ in $L^1(A^{\prime})$ for some $u\in BV(A^{\prime},\S)$. Since $A^{\prime}$ was arbitrary, the claim follows by a diagonal argument combined with equiboundedness which rules out concentrations close to the boundary.
\end{proof}

We will also use the following auxiliary result about the convergence introduced in Definition \ref{defconv}.

\begin{lemma}\label{sumconv}
Let $A\in\mathcal{A}(D)$ be such that $|\partial A|=0$ and let $u_{\e},v_{\e}:\e\L\to \S$ both converge in $A$ to $u$ in the sense of Definition \ref{defconv} and assume both have equibounded energy on $A$. Then
\begin{equation*}
\lim_{\e\to 0}\sum_{\substack{\e x\in\e\L \\ \e P_k(x)\in A}}\e^k|u_{\e}(\e x)-v_{\e}(\e x)|=0.
\end{equation*}
\end{lemma}

\begin{proof}
Fix a set $A^{\prime}\subset\subset A$ such that $A^{\prime}\in\Ard$. By (\ref{localratio}) and equiboundedness of $u_{\e}$ and $v_{\e}$ it is enough to show that
\begin{equation*}
\lim_{\e\to 0}\sum_{\substack{\e x\in\e\L \\ \e P_k(x)\in A^{\prime}}}\e^k|u_{\e}(\e x)-v_{\e}(\e x)|=0.
\end{equation*}
Using the fact that $u_{\e},v_{\e}$ both have finite energy in $A$, we can argue as in the derivation of (\ref{numberest}) to show that
\begin{equation*}
\#\{\e x\in\e\L:\;\e P_k(x)\in A^{\prime},\,Pu_{\e}(\e P_k(x))\neq u_{\e}(\e x)\text{ or }Pv_{\e}(\e P_k(x))\neq v_{\e}(\e x)\}\leq C\e^{1-k}.
\end{equation*}
Inserting this estimate and using that $\L$ satisfies (\ref{thickness}) we obtain
\begin{align*}
\sum_{\substack{\e x\in\e\L \\ \e P_k(x)\in A^{\prime}}}\e^k|u_{\e}(\e x)-v_{\e}(\e x)|&\leq C\sum_{\substack{\e z\in \e P_k\L \\ \e z\in A^{\prime}}}\e^k|Pu_{\e}(\e z)-Pv_{\e}(\e z)|+C\e.
\end{align*}
Thus it is enough to control the last sum. Since the Voronoi cells in the projected lattice may become degenerate, we can only use bounds on the number of cells. To this end fix $L>1$ large enough such that, for all $z_L\in L\mathbb{Z}^k$, we have
\begin{equation}\label{largerlattice}
1\leq\#\left(\e P_k\L \cap (\e z_L+[0,L\e)^{k})\right)\leq C.
\end{equation}
Define $I_{\e}:=\{z_L\in L\mathbb{Z}^k:\;(\e z_L+[0,L\e)^k)\cap A^{\prime}\neq\emptyset\}$ and subdivide this set again as
\begin{align*}
I_{\e}^1&:=\{z_L\in I_{\e}:\;Pu_{\e}\text{ is not constant on }\e z_L+[0,L\e)^k\},\\
I_{\e}^2&:=\{z_L\in I_{\e}:\;Pv_{\e}\text{ is not constant on }\e z_L+[0,L\e)^k\},\\
I_{\e}^3&:=I_{\e}\backslash(I_{\e}^1\cup I_{\e}^2).
\end{align*}
Since every scaled $k$-dimensional Voronoi cell $\e \mathcal{C}_k(z)$ can only intersect finitely many cubic cells $\e z_L+[0,L\e)^k$ with a uniform bound on the cardinality, we can again use the energy bound in $A$ and argue as for (\ref{numberest}) to conclude that
\begin{equation}\label{rarecubes}
\#(I_{\e}^1\cup I_{\e}^2)\leq C\e^{1-k}.
\end{equation}
Combining (\ref{largerlattice}) and (\ref{rarecubes}) we infer from the definition of the set $I_{\e}^3$ that
\begin{align*}
\sum_{\substack{\e z\in \e P_k\L \\ \e z\in A^{\prime}}}&\e^k|Pu_{\e}(\e z)-Pv_{\e}(\e z)|\leq C\e+\sum_{z_L\in I_{\e}^3}\sum_{\substack{\e z\in \e P_k\L \\ \e z\in \e z_L+[0,L\e)^k}}\e^k|Pu_{\e}(\e z)-Pv_{\e}(\e z)|\\
&\leq C\e+C\sum_{z_L\in I_{\e}^3}\int_{\e z_L+[0,L\e)^k}|Pu_{\e}(s)-Pv_{\e}(s)|\,\mathrm{d}s\leq C\e+C\|Pu_{\e}-Pv_{\e}\|_{L^1(A)}.
\end{align*}
This concludes the proof, since the last term tends to $0$ by assumption.
\end{proof}
\hspace*{0,5cm}
Following some ideas in \cite{ACG2} we introduce an auxiliary deterministic square lattice on which we will rewrite the energies $E_{\e}$. This lattice will turn out to be a convenient way to control the long-range interactions.\\

On setting  $r^{\prime}=\frac{r}{\sqrt{d}}$ it follows that $\#\{\L\cap\{\alpha+[0,r^{\prime})^d\}\}\leq 1$  for all $\alpha\in \rzd$. We now set
\begin{align*}
\mathcal{Z}_{r^{\prime}}(\L):=&\{\a\in r^{\prime}\mathbb{Z}^d:\;\#\left(\L\cap \{\a+[0,r^{\prime})^d\}\right)=1\},
\\
x_{\a}:=& \L\cap\{\a+[0,r^{\prime})^d\},\quad\a\in\mathcal{Z}_{r^{\prime}}(\L)
\end{align*}
and, for $\xi\in\rzd,\,U\subset\R^k$ and $\e>0$, 
\begin{align*}
R^{\xi}_{\e}(U)&:=\{\a:\;\a,\a+\xi\in\mathcal{Z}_{r^{\prime}}(\L),\,\e x_{\a},\e x_{\a+\xi}\in P_k^{-1}U\}.
\end{align*}
Note that by (\ref{thickness}), enlarging $M$ if necessary, it is enough to consider $\xi\in \rzd_M:=\rzd\cap (\R^k\times [-2M,2M]^{d-k})$. We can then rewrite the localized energy as
\begin{align*}
E_{\e}(u,A)&=\sum_{\xi\in \rzd_M}\sum_{\a\in R^{\xi}_{\e}(A)}\e^{k-1}f_{\e}(x_{\a},x_{\a+\xi},u(\e x_{\a}),u(\e x_{\a+\xi})).
\end{align*}

\begin{remark}\label{flatdeday}
Observe that we can write 
\begin{equation*}
\{\xi\in\rzd_M\}=\bigcup_{\substack{z\in r^{\prime}\mathbb{Z}^{d-k}\\|z|_{\infty}\leq 2M}}\{\xi=(\xi_k,z_1,\dots,z_{d-k}):\;\xi_k\in r^{\prime}\mathbb{Z}^k\}.
\end{equation*}
Hence the monotonicity assumption from Hypothesis 1 allows to transfer the decay  of long-range interactions to the discrete environment as follows: Given $\delta>0$, there exists $L_{\delta}>0$ such that 
\begin{equation}\label{lrdecay}
\sum_{\substack{\xi\in \rzd_M\\|\xi|>L_{\delta}}}J_{lr}(|\hat{\xi}|)|\xi|\leq \delta,
\end{equation}
where $\hat{\xi}\in \xi+[-r^{\prime},r^{\prime}]^d$ is such that $|\hat{\xi}|=\dist([0,r^{\prime})^d,[0,r^{\prime})^d+\xi)$. This decay property along with Lemma \ref{longtoshort} below will be crucial to control the long-range interactions. However note that $L_{\delta}$ in general depends on $M$.
\end{remark}
The following lemma asserts that on convex domains we can essentially control the long-range interactions by considering only nearest neighbours.

\begin{lemma}\label{longtoshort}
Let $B\subset\mathcal{A}(\R^k)$ be convex and $B^{\e}=\{x\in\R^k:\;\dist(x,B)< 3(R+M)\e\}$. Then there exists a constant $C$ depending only on $r,R,M$ in Definition \ref{flatadmissible} such that for every $\xi\in \rzd_M$ and every $u:\e\L\to\S$ it holds
\begin{equation*}
\sum_{\a\in R_{\e}^{\xi}(B)}f_{\e}(x_{\a},x_{\a+\xi},u(\e x_{\a}),u(\e x_{\a+\xi}))\leq CJ_{lr}(|\hat{\xi}|)|\xi|\sum_{\substack{(x,y)\in\mathcal{N\!N}(\L)\\ \e x,\e y\in P_k^{-1}B^{\e}}}f_{\e}(x,y,u(\e x),u(\e y)).
\end{equation*}
\end{lemma}

\begin{proof}
Let $\a\in R_{\e}^{\xi}(B)$. As in the proof of Lemma \ref{compactness} we consider the collection of segments $\mathcal{G}_{\delta}(x_{\a},x_{\a+\xi})$ defined as in (\ref{def-ray}). By the same argument there exists a segment $g^*\subset \mathcal{G}_{\delta}(x_{\a},x_{\a+\xi})$ satisfying (\ref{onlynncells}). Consider then the set $G(\a,\xi)=\{x\in\L:\,g^*\cap \mathcal{C}(x)\neq\emptyset\}$. By construction we can number $G(\a,\xi)=\{x_{\a}=x^0,\dots,x^N=x_{\a+\xi}\}$ such that $(x^i,x^{i+1})\in\mathcal{N\!N}(\L)$. By the bounds of Hypothesis 1 it holds that
\begin{align}\label{singlecontribution}
f_{\e}(x_{\a},x_{\a+\xi},u(\e x_{\a}),u(\e x_{\a+\xi}))&\leq J_{lr}(|\hat{\xi}|)|u(\e x_{\a})-u(\e x_{\a+\xi})|\leq J_{lr}(|\hat{\xi}|)\sum_{\substack{(x,y)\in\mathcal{N\!N}(\L)\\x,y\in G(\a,\xi)}}|u(\e x)-u(\e y)|\nonumber
\\
&\leq CJ_{lr}(|\hat{\xi}|)\sum_{\substack{(x,y)\in\mathcal{N\!N}(\L)\\x,y\in \frac{1}{\e}P_k^{-1}B^{\e}\cap G(\a,\xi)}}f_{\e}(x,y,u(\e x),u(\e y)),
\end{align}
where we used that by convexity we have $G(\a,\xi)\subset \frac{1}{\e}P_k^{-1}B^{\e}$ provided $\delta$ is small enough. Now given $(x,y)\in\mathcal{N\!N}(\L)\cap \frac{1}{\e}P_k^{-1}B^{\e}$ we set
\begin{equation*}
T_{\e}^{\xi}(x,y):=\{\a\in R_{\e}^{\xi}(B):\;\{x,y\}\cap G(\a,\xi)\neq\emptyset\}.
\end{equation*}
Note that if $\a\in T_{\e}^{\xi}(x,y)$, then 
\begin{equation*}
x_{\a}\in\{z+t\xi:\;|z-x|\leq C,\,|t|\leq C\}
\end{equation*}
for some $C>0$, and hence $\#T_{\e}^{\xi}(x,y)\leq C|\xi|$ by Definition \ref{flatadmissible}. The claim now follows by summing (\ref{singlecontribution}) over all $\a\in R_{\e}^{\xi}(B)$.
\end{proof}

\section{Integral representation on the flat set}
Our first aim is to characterize all possible variational limits of energies $E_{\e}$ that satisfy Hypothesis 1. As for the case $k=d$ and $\S=\{\pm 1\}$ treated in \cite{ACR}, the following version of Theorem 3 in \cite{BFLM} will be the key ingredient:

\begin{theorem}\label{represent}
Let ${\mathcal F}: BV(D,\S)\times {\mathcal A}(D)\to [0,+\infty)$ satisfy the following hypotheses:
\begin{itemize}
	\item[(i)] ${\mathcal F}(u, \cdot)$ is the restriction to ${\mathcal A}(D)$ of a Radon measure;
	\item[(ii)] ${\mathcal F}(u, A)={\mathcal F}(v, A)$ whenever $u = v$ a.e. on $A\in{\mathcal A}(D)$;
	\item[(iii)] ${\mathcal F}(\cdot, A)$ is $\lud$ lower semicontinuous for every $A\in {\mathcal A}(D)$;
	\item[(iv)] there exists $c>0$ such that
	$$
	\frac 1 c \mathcal{H}^{k-1}(S_u\cap A)\leq  {\mathcal F}(u, A) \leq c \,\mathcal{H}^{k-1}(S_u\cap A)
	$$
	for every $(u, A)\in BV(D,\S)\times {\mathcal A}(D)$.
\end{itemize}
Then for every $u\in BV(D,\S)$ and $A\in{\mathcal A}(D)$
$$
{\mathcal F}(u, A)=\int_{S_u\cap A}g(x,u^+,u^-,\nu_u)\,\mathrm{d}\mathcal{H}^{k-1},
$$
with
$$
g(x_0,s_i,s_j,\nu)=\limsup_{\rho\to 0}\frac{m(u^{ij}_{x_0,\nu}, Q_\nu(x_0,\rho))}{\rho^{k-1}},
$$
where, for all $s_i,s_j\in\S$,
\begin{align*}
u^{ij}_{x_0,\nu}:=
\begin{cases}
s_i & \mbox{if $\langle x-x_0,\nu\rangle\geq 0$,} \\
s_j & \mbox{otherwise,}
\end{cases}
\end{align*}
and for any $(v,A)\in BV(D,\S)\times {\mathcal A}(D)$ we set
\begin{equation*}
m(v,A)= \inf\{ {\mathcal F}(u, A): u\in BV(A,\S),\ u=v\ \hbox{in a neighbourhood of } \partial A\}.
\end{equation*}	
\end{theorem}

\hspace*{0,5cm}
The following theorem is the main result of this section. 

\begin{theorem}\label{mainthm1}
Let $\L$ be a thin admissible lattice and let $f_{nn}^{\e}$ and $f_{lr}^{\e}$ satisfy Hypothesis 1. For every sequence of $\e\to 0^+$ there exists a subsequence $\e_{n}$ such that the functionals $E_{\e_{n}}$ $\Gamma$-converge with respect to the convergence of Definition \ref{defconv} with $A=D$ to a functional $E:L^1(D,\R^q)\rightarrow [0,+\infty]$ of the form
\begin{equation*}
E(u)=
\begin{cases}\displaystyle
\int_{S_u}\phi(x,u^+,u^-,\nu_u)\,\mathrm{d}\mathcal{H}^{k-1} &\mbox{if $u\in BV(D,\S)$,} \\
+\infty &\mbox{otherwise.}
\end{cases}
\end{equation*}
Moreover a local version of the statement above holds: For all $u\in BV(D,\S)$ and all $A\in\Ard$
\begin{equation*}
\Gamma\hbox{-}\lim_n E_{\e_{n}}(u,A)=\int_{S_u\cap A}\phi(x,u^+,u^-,\nu_u)\,\mathrm{d}\mathcal{H}^{k-1},
\end{equation*}
with respect to the same convergence as above.
\end{theorem}

\begin{remark}\label{1D}
If $k=1$, then a similar result holds. In this case we obtain a limit energy finite for $u\in BV(D,\S)$ and of the form
\begin{equation*}
E(u)=\sum_{x\in S_u}\phi(x,u^+,u^-).
\end{equation*}
\end{remark}
\hspace*{0,5cm}
The proof of Theorem \ref{mainthm1} will be given later and it is based on Theorem \ref{represent}. We now start proving several propositions that allow us to apply Theorem \ref{represent}.\\
\hspace*{0,5cm}
We start with the growth condition (iv) of Theorem \ref{represent}. Using the lower semicontinuity of the perimeter of level sets in $BV(D,\S)$, one can use the same argument as for Lemma \ref{compactness} to prove the following lower bound for $E^{\prime}(u,A)$:

\begin{proposition}\label{lb}
Assume that Hypothesis 1 holds. Then $E^{\prime}(u,A)<+\infty$ only if $u\in BV(A,\S)$ and there exists a constant $c>0$ independent of $A$ such that	
\begin{equation*}
\frac{1}{c}\Hk(S_u\cap A)\leq E^{\prime}(u,A).
\end{equation*}
\end{proposition}

In the next step we provide a suitable upper bound for $E^{\prime\prime}(u,A)$.

\begin{proposition}\label{ub}
Assume Hypothesis 1 holds. Then there exists a constant $c>0$ such that, for all $A\in\Ard$ and all $u\in BV(D,\S)$,
\begin{equation*}
E^{\prime\prime}(u,A)\leq c\,\Hk(S_u\cap A).
\end{equation*}
\end{proposition}

\begin{proof}
First, assume that $u$ is a polyhedral function on $\R^k$, which means that all level sets have boundaries that coincide (up to $\mathcal{H}^{k-1}$-null sets) with a finite union of $k-1$-dimensional simplexes. We define a sequence $u_{\e}:\e\L\to \S$ by setting 
\begin{equation*}
u_{\e}(\e x):=u(\e P_k(x))
\end{equation*}
Note that $u_{\e}\to u$ in the sense of Definition \ref{defconv}. Given $\delta>0$, we choose $L_{\delta}>0$ such that (\ref{lrdecay}) holds. We further set $A^{\delta}=A+B_{\delta}(0)$. For $|\xi|\leq L_{\delta}$, we can argue as in the proof of Lemma \ref{longtoshort} to show that, for $\e$ small enough, it holds that
\begin{align}\label{polyhedralshort}
\sum_{\a\in R_{\e}^{\xi}(A)}\e^{k-1}f_{\e}(x_{\a},x_{\a+\xi},u_{\e}(\e x_{\a}),u_{\e}(\e x_{\alpha+\xi}))&\leq CJ_{lr}(|\hat{\xi}|)|\xi|\sum_{\substack{(x,y)\in\mathcal{N\!N}(\L)\\ \e x,\e y\in P_k^{-1}A^{\delta}}}\e^{k-1}|u_{\e}(\e x)-u_{\e}(\e y)|\nonumber\\
&\leq CJ_{lr}(|\hat{\xi}|)|\xi| \mathcal{H}^{k-1}(S_u\cap A^{\delta}),
\end{align}
where the last estimate follows from the regularity of $S_u$. Next we consider the interactions where $|\xi|>L_{\delta}$. Let $u$ be a polyhedral function; applying Lemma \ref{longtoshort} we deduce for any $\e>0$ the weaker bound
\begin{align}\label{polyhedrallong}
\sum_{\a\in R_{\e}^{\xi}(A)}\e^{k-1}f_{\e}(x_{\a},x_{\a+\xi},u_{\e}(\e x_{\a}),u_{\e}(\e x_{\alpha+\xi}))&\leq \sum_{\a\in R_{\e}^{\xi}(\R^k)}\e^{k-1}f_{\e}(x_{\a},x_{\a+\xi},u_{\e}(\e x_{\a}),u_{\e}(\e x_{\alpha+\xi}))\nonumber
\\
&\leq CJ_{lr}(|\hat{\xi}|)|\xi|\mathcal{H}^{k-1}(S_u).
\end{align}
Combining (\ref{polyhedralshort}),(\ref{polyhedrallong}) and (\ref{lrdecay}) and the integrability assumption from Hypothesis 1, we deduce that
\begin{equation*}
E^{\prime\prime}(u,A)\leq \limsup_{\e}E_{\e}(u_{\e},A)\leq C\Hk(S_u\cap A^{\delta})+C\delta\Hk(S_u).
\end{equation*}
As $\delta>0$ was arbitrary we obtain that
\begin{equation}\label{withclosure}
E^{\prime\prime}(u,A)\leq C\mathcal{H}^{k-1}(S_u\cap\overline{A}).
\end{equation}
\hspace*{0,5cm}
Now we use locality and a density argument. Indeed, for every $u\in BV(D,\S)$ we can find a function $\tilde{u}\in BV_{\rm loc}(\R^k,\S)$ such that $u=\tilde{u}$ on $A$ and $\mathcal{H}^{k-1}(S_{\tilde{u}}\cap\partial A)=0$ (see Lemma 2.7 in \cite{BrCoGa}). From Remark \ref{almostgamma} it follows that $E^{\prime\prime}(u,A)=E^{\prime\prime}(\tilde{u},A)$. Then, by \cite[Corollary 2.4]{BrCoGa} there exists a sequence $u_n\in BV_{\rm loc}(\R^k,\S)$ of polyhedral functions such that $u_n\to \tilde{u}$ in $L^1(D)$ and $\Hk(S_{u_n}\cap D)\to\Hk(S_{\tilde{u}}\cap D)$. By the $L^1(D)$-lower semicontinuity of $E^{\prime\prime}(\cdot,A)$ stated in Remark \ref{almostgamma} and (\ref{withclosure}) we obtain that
\begin{equation*}
E^{\prime\prime}(u,A)\leq \liminf_nE^{\prime\prime}(u_n,A) \leq C\limsup_n\mathcal{H}^{k-1}(S_{u_n}\cap \overline{A})\leq C\mathcal{H}^{k-1}(S_{\tilde{u}}\cap \overline{A})=C\mathcal{H}^{k-1}(S_u\cap A),
\end{equation*}
where the last inequality is a consequence of the $L^1(D)$-lower semicontinuity of $u\mapsto \mathcal{H}^{k-1}(S_u\cap D\backslash\overline{A})$ for $u\in BV(D,\S)$.
\end{proof}

\hspace*{0,5cm}
As usual for applying integral-representation theorems we next establish a weak subadditivity property of $A\mapsto E^{\prime\prime}(u,A)$.

\begin{proposition}\label{subadditive}
Let $f_{nn}^{\e}$ and $f_{lr}^{\e}$ satisfy Hypothesis 1. Then, for every $A,B\in\mathcal{A}^R(D)$, every $A^{\prime}\subset\Ard$ such that $A^{\prime}\subset\subset A$ and every $u\in BV(D,\S)$,	
\begin{equation*}
E^{\prime\prime}(u,A^{\prime}\cup B)\leq E^{\prime\prime}(u,A)+E^{\prime\prime}(u,B).
\end{equation*} 
\end{proposition}

\begin{proof}
We may assume that $E^{\prime\prime}(u,A)$ and $E^{\prime\prime}(u,B)$ are both finite. Let $u_{\e},v_{\e}:\e\L\to \S$ both converge to $u$ in the sense of Definition \ref{defconv} such that	
\begin{equation}\label{optimalsequences}
\limsup_{\e\to 0}E_{\e}(u_{\e},A)=E^{\prime\prime}(u,A),\quad\limsup_{\e\to 0}E_{\e}(v_{\e},B)=E^{\prime\prime}(u,B).
\end{equation}
\textbf{Step 1} Extensions to convex domains
\\
Let $Q_D$ be a cube containing $\overline{D}$. Since $D\in\Ard$, we can extend $u$ (without relabeling) to a function $u\in BV_{{\rm loc}}(\mathbb{R}^k,\S)$. We first show that we can modify $u_{\e}$ and $v_{\e}$ on $\e\L\backslash A$ and $\e\L\backslash B$ respectively, such that they converge to $u$ on $L^1(Q_D)$ and such that they have equibounded energy on the larger set $Q_D$. We will show the argument for $u_{\e}$. Take another cube $Q^{\prime}$ such that $Q_D\subset\subset Q^{\prime}$. Arguing as in the proof of Proposition \ref{ub} we find a sequence $\tilde{u}_{\e}:\e\L\to\S$ such that $\tilde{u}_{\e}\to u$ on $Q^{\prime}$ and $\limsup_{\e\to 0}E_{\e}(\tilde{u}_{\e},Q^{\prime})\leq C\mathcal{H}^{k-1}(S_u\cap Q^{\prime})$. We then set $\bar{u}\in\mathcal{PC}_{\e}(\L)$ as 
\begin{equation*}
\bar{u}(\e x)=\mathds{1}_A(P_k(\e x))u_{\e}(\e x)+(1-\mathds{1}_A(P_k(\e x)))\tilde{u}_{\e}(\e x).
\end{equation*}
Then $\bar{u}_{\e}\to u$ on $Q_D$ and applying Lemma \ref{longtoshort} combined with Hypothesis 1 and (\ref{thickness}) yields
\begin{align*}
E_{\e}(\bar{u}_{\e},Q_D)\leq& C\sum_{\xi\in\rzd_M}J_{lr}(|\hat{\xi}|)|\xi|\sum_{\substack{(x,y)\in\NNS\\ \e x,\e y\in Q^{\prime}}}\e^{k-1}f_{\e}(x,y,\bar{u}_{\e}(\e x),\bar{u}_{\e}(\e y))
\\
\leq &C\left(E_{\e}(u_{\e},A)+E_{\e}(\tilde{u}_{\e},Q^{\prime}\backslash A)+\frac{1}{\e}|\partial A+B_{4R\e}(0)|\right).
\end{align*}
The first and second term remain bounded by construction, while the third term converges to a multiple of the Minkowski content of $\partial A$ which agrees with $\mathcal{H}^{k-1}(\partial A)$ as $A\in\Ard$.
\\
\textbf{Step 2} Energy estimates
\\
Again, given $\delta>0$ we choose $L_{\delta}$ such that (\ref{lrdecay}) holds. Fix $d^{\prime}\leq\frac{1}{2}{\rm dist}(A^{\prime},\partial A)$ and let $N_{\e}:=\lfloor\frac{d^{\prime}}{\e (L_{\delta}+2r)}\rfloor$, where $\lfloor\cdot\rfloor$ denotes the integer part. For $j\in\mathbb{N}$ we define
\begin{equation*}
A_{\e,j}:=\{x\in A:\;{\rm dist}(x,A^{\prime})<j\e(L_{\delta}+2r)\}.
\end{equation*}
We let $w^j_{\e}\in \mathcal{PC}_{\e}(\L)$ be the interpolation defined by
\begin{equation*}
w^j_{\e}(\e x)=\mathds{1}_{A_{\e,j}}(P_k(\e x))u_{\e}(\e x)+(1-\mathds{1}_{A_{\e,j}}(P_k(\e x)))v_{\e}(\e x).
\end{equation*}
Note that for each fixed $j\in\mathbb{N}$, $w^j_{\e}\to u$ on $D$ in the sense of Definition \ref{defconv}. We set
\begin{equation*}
S_j^{\xi,\e}:=\{x=y+t\,P_k(\xi^{\prime}):\;y\in\partial A_{\e,j},\,|t|\leq\e,\xi^{\prime}\in\xi+[-r^{\prime},r^{\prime}]^d\}\cap (A\cup B).
\end{equation*}
For $j\leq N_{\e}$ we have 
\begin{align}\label{almostsubeq}
E_{\e}(w^j_{\e},A^{\prime}\cup B)\leq& E_{\e}(u_{\e},A_{\e,j})+E_{\e}(v_{\e},B\backslash A_{\e,j})\nonumber
\\
&+\sum_{\xi\in\rzd_M}\sum_{\a\in R^{\xi}_{\e}(S_j^{\xi,\e})}\underbrace{\e^{k-1}f_{\e}(x_{\alpha},x_{\a+\xi},w^j_{\e}(\e x_{\a}),w^j_{\e}(\e x_{\a+\xi}))}_{=:\rho_j^{\xi,\e}(\a)}\nonumber
\\
\leq &E_{\e}(u_{\e},A)+E_{\e}(v_{\e},B)
+\sum_{\xi\in\rzd_M}\sum_{\a\in R^{\xi}_{\e}(S_j^{\xi,\e})}\rho_j^{\xi,\e}(\a).
\end{align}
We now distinguish between two types of interactions depending on $L_{\delta}$. If $|\xi|>L_{\delta}$, we use Lemma \ref{longtoshort}. Since $A\cup B\subset\subset Q_D$, we deduce that 
\begin{equation*}
\sum_{|\xi|>L_{\delta}}\sum_{\a\in R^{\xi}_{\e}(S_j^{\xi,\e})}\rho_j^{\xi,\e}(\a)\leq C\sum_{|\xi|>L_{\delta}}J_{lr}(|\hat{\xi}|)|\xi|\sum_{\substack{(x,y)\in\NNS\\ \e x,\e y\in P_k^{-1}Q_D}}\e^{k-1}f_{\e}(x,y,w^j_{\e}(\e x),w^j_{\e}(\e y)).
\end{equation*}
We have $P_k^{-1}Q_D\subset P_k^{-1}A_{\e,j}\cup P_k^{-1}(Q_D\backslash A_{\e,j})$. Nearest-neighbour interactions between those two sets are contained in $P_k^{-1}S_{k}^{\xi,\e}$ for some $\xi\in\rzd_M$ with $|\xi|\leq 4R$. Therefore, we can further estimate the last inequality via
\begin{equation}\label{longrangeest}
\sum_{|\xi|>L_{\delta}}\sum_{\a\in R^{\xi}_{\e}(S_j^{\xi,\e})}\rho_j^{\xi,\e}(\a)\leq C\delta\Big(E_{\e}(u_{\e},A)+E_{\e}(v_{\e},Q_D)
+\sum_{|\xi|\leq L_{\delta}}\sum_{\a\in R^{\xi}_{\e}(S_j^{\xi,\e})}\rho_j^{\xi,\e}(\a)\Big).
\end{equation}
Now we treat the interactions when $|\xi|\leq L_{\delta}$. Consider any points $\e x,\e y\in\e\L$. If $w^j_{\e}(\e x)\neq w^j_{\e}(\e y)$ then either $\e x,\e y\in A_{\e,j}$, $\e x,\e y\notin A_{\e,j}$ or $\e x\in A_{\e,j}$ but $\e y\notin A_{\e,j}$ (the reverse case can be treated similar). In the last case we have a contribution only if $u_{\e}(\e x)\neq v_{\e}(\e y)$. Then either $u_{\e}(\e y)=v_{\e}(\e y)$ or $f_{\e}(x,y,u_{\e}(\e x),v_{\e}(\e y))\leq C|u_{\e}(\e y)-v_{\e}(\e y)|$. Summarizing all cases we obtain the inequality
\begin{align*}
\rho_j^{\xi,\e}(\a)\leq &\e^{k-1}f_{\e}(x,y,u_{\e}(\e x),u_{\e}(\e y))+\e^{k-1}f_{\e}(x,y,v_{\e}(\e x),v_{\e}(\e y))+C\e^{k-1}|u_{\e}(\e y)-v_{\e}(\e y)|.
\end{align*}
By our construction we have $S_j^{\e,\xi}\subset (A_{\e,j+1}\backslash A_{\e,j-1})=: S_j^{\e}$. We deduce that
\begin{equation*}
\sum_{|\xi|\leq L_{\delta}}\sum_{\a\in R^{\xi}_{\e}(S_j^{\xi,\e})}\rho_j^{\xi,\e}(\a)
\leq E_{\e}(u_{\e},S_j^{\e})+E_{\e}(v_{\e},S_j^{\e})
+C_{\delta}\sum_{\substack{y\in\L \\ \e P_k(y)\in S_j^{\e}}}\e^{k-1}|u_{\e}(\e y)-v_{\e}(\e y)|,
\end{equation*}
where $C_{\delta}$ depends only on $L_{\delta}$. Observe that by definition every point can only lie in at most two sets $S_{j_1}^{\e},S_{j_2}^{\e}$. Thus averaging combined with (\ref{longrangeest}), Step 1 and the last inequality yields
\begin{align*}
I_{\e}&:=\frac{1}{N_{\e}}\sum_{j=1}^{N_{\e}}\sum_{\xi\in\rzd_M}\sum_{\a\in R^{\xi}_{\e}(S_j^{\xi,\e})}\rho_j^{\xi,\e}(\a)\leq \frac{2}{N_{\e}}\sum_{j=1}^{N_{\e}}\sum_{|\xi|\leq L_{\delta}}\sum_{\a\in R^{\xi}_{\e}(S_j^{\xi,\e})}\rho_j^{\xi,\e}(\a)+C\delta
\\
&\leq \frac{4}{N_{\e}}\left(E_{\e}(u_{\e},Q_D)+E_{\e}(v_{\e},Q_D)\right)+C_{\delta}\sum_{\substack{y\in\L \\ \e P_k(y)\in D}}\e^{d}|u_{\e}(\e y)-v_{\e}(\e y)|+C\delta
\\
&\leq\frac{C}{N_{\e}}+C_{\delta}\sum_{\substack{y\in\L \\ \e P_k(y)\in D}}\e^{d}|u_{\e}(\e y)-v_{\e}(\e y)|+C\delta.
\end{align*}
Due to Step 1 we can apply Lemma \ref{sumconv} to deduce that $\limsup_{\e\to 0}I_{\e}\leq C\delta$. For every $\e>0$, let $j_{\e}\in\{1,\dots,N_{\e}\}$ be such that
\begin{equation}\label{boundrho}
\sum_{\xi\in\rzd_M}\sum_{\a\in R^{\xi}_{\e}(S_{j_{\e}}^{\xi,\e})}\rho_{j_{\e}}^{\xi,\e}(\a)\leq I_{\e}
\end{equation}
and set $w_{\e}:=w_{\e}^{j_{\e}}$. Note that, as a convex combination, $w_{\e}$ still converges to $u$ on $D$. Hence, using (\ref{almostsubeq}) and (\ref{boundrho}), we conclude that
\begin{equation*}
E^{\prime\prime}(u,A^{\prime}\cup B)\leq\limsup_{\e\to 0}E_{\e}(w_{\e},A^{\prime}\cup B)\leq E^{\prime\prime}(u,A)+E^{\prime\prime}(u,B)+C\,\delta.
\end{equation*}
The arbitrariness of $\delta$ proves the claim.
\end{proof}

\begin{proof}[Proof of Theorem \ref{mainthm1}] From Propositions \ref{ub} and \ref{subadditive} it follows by standard arguments that $E^{\prime\prime}(u,\cdot)$ is inner regular on $\Ard$ (see, for example, Proposition 11.6 in \cite{BrDe}). Therefore, given a sequence $\e_n\to 0^+$ we can use Remark \ref{almostgamma} and the compactness property of $\Gamma$-convergence (see \cite{GCB} Section 1.8.2) to construct a subsequence $\e_n$ (not relabeled) such that	
\begin{equation*}
\Gamma\hbox{-}\lim_n E_{\e_n}(u,A)=:\tilde{E}(u,A)
\end{equation*}	
exists for every $(u,A)\in\lud\times\Ard$. By Proposition \ref{lb} we know that $\tilde{E}(u,A)$ is finite only if $u\in BV(A,\S)$. We extend $\tilde{E}(u,\cdot)$ to $\mathcal{A}(D)$ setting
\begin{equation*}
E(u,A):=\sup\,\{\tilde{E}(u,A^{\prime}):\;A^{\prime}\subset\subset A,\, A^{\prime}\in\Ard\}.
\end{equation*}	
To complete the proof, it is enough to show that $E$ satisfies the assumptions of Theorem \ref{represent}. Again by standard arguments $E(u,\cdot)$ fulfills the assumptions of the De Giorgi-Letta criterion (\cite{GCB} Section 16) so that $E(u,\cdot)$ is the trace of a Borel measure. By Proposition \ref{ub}, it is indeed a Radon measure. The locality property follows from Remark \ref{almostgamma}. By the properties of $\Gamma$-limits and again Remark \ref{almostgamma} we know that $\tilde{E}(\cdot,A)$ is $\lud$-lower semicontinuous and so is $E(\cdot,A)$ as the supremum of lower semicontinuous functions. The growth conditions (iv) in Theorem \ref{represent} follow from Propositions \ref{lb} and \ref{ub} which still hold for $E$ in place of $\tilde{E}$. The local version of the theorem is a direct consequence of our construction.
\end{proof}

\section{Convergence of boundary value problems}\label{bvproblems}
In this section we consider the convergence of minimum problems with Dirichlet-type boundary data. In order to model boundary conditions in our discrete setting we need to introduce a suitable notion of trace taking into account possible long range interactions (see also \cite{ACR}). In what follows we will further assume a continuous spatial dependence of the integrand of the limit continuum energy. Without such a condition we can still obtain a weaker result stated in Lemma \ref{approxminprob}. On the other hand continuity assumptions are always fulfilled in the case of the homogenization problem that we are going to treat in Section \ref{homogenization}.
\\
\hspace*{0,5cm}
Consider $A\in\Ard$ and fix boundary data $u_0\in BV(\R^k_{{\rm loc}},\S)$. We assume that the boundary data are well-prepared in the sense that, setting $u_{\e,0}\in \mathcal{PC}_{\e}(\L)$ as $u_{\e,0}(\e x)=u_0(P_k(\e x))$, we have $u_{\e,0}\to u_0$ on $D$ and
\begin{equation}\label{nojump}
\limsup_{\e\to 0}E_{\e}(u_{\e,0},B)\leq C\mathcal{H}^{k-1}(S_{u_0}\cap \overline{B}),\quad\quad\mathcal{H}^{k-1}(S_{u_0}\cap\partial A)=0.
\end{equation}
with $C$ independent of $B\in\mathcal{A}^R(\R^k)$. Observe that as in the proof of Proposition \ref{ub} we may allow for any polyhedral function such that $\mathcal{H}^{k-1}(S_{u_0}\cap\partial A)=0$, but more generally it suffices that all level sets are Lipschitz sets.
\\
\hspace*{0.5cm}
We define a {\em discrete trace constraint} as follows: Let $l_{\e}>0$ be such that
\begin{equation}\label{slowboundary}
\lim_{\e\to 0}l_{\e}=+\infty,\quad \lim_{\e\to 0}l_{\e}\e= 0.
\end{equation}
We set $\mathcal{PC}_{\e,u_0}^{l_{\e}\e}(\L,A)$ as the space of those $u$ that agree with $u_{0}$ at the discrete boundary of $A$, by setting
\begin{equation*}
\mathcal{PC}_{\e,u_0}^{l_{\e}\e}(\L,A):=\{u:\e\L\to\S:\;u(\e x)=u_{0}( P_k(\e x))\text{ if }\dist(P_k(\e x),\partial A)\leq l_{\e}\e\}.
\end{equation*} 
For $\e>0$ and $l_{\e}>0$ we consider the restricted functional $E^{l_{\e}}_{\e,u_0}(\cdot,A):\mathcal{PC}_{\e,u_0}^{l_{\e}\e}(\L,A)\to[0,+\infty]$ defined as
\begin{equation}\label{boundarytype}
E^{l_{\e}}_{\e,u_0}(u,A):=
{E}_{\e}(u,A).
\end{equation}
We need some further notation. Given $u\in BV(D,\S)$, we set $u_{A,0}:\mathbb{R}^k\to\S$ as
\begin{equation*}
u_{A,0}(x):=
\begin{cases}
u(x) &\mbox{if $x\in A$,}\\
u_{0}(x) &\mbox{otherwise.}
\end{cases}
\end{equation*}
Since $A$ is regular we have $u_{A,0}\in BV_{loc}(\mathbb{R}^k,\S)$. The following convergence result holds:

\begin{theorem}\label{constrainedproblem}
Let $\L$ be a thin admissible lattice and let $f_{nn}^{\e}$ and $f_{lr}^{\e}$ satisfy Hypothesis 1. For every sequence converging to $0$, let $\e_{n}$ and $\phi$ be as in Theorem \ref{mainthm1}. Assume that the limit integrand $\phi$ is continuous on $D\times\S^2\times S^{k-1}$. Then, for every set $A\in\Ard$, $A\subset\subset D$, the functionals $E^{l_{\e_n}}_{\e_n,u_0}(\cdot,A)$ defined in (\ref{boundarytype}) $\Gamma$-converge with respect to the convergence on $A$ in Definition \ref{defconv} to the functional $E_{u_0}(\cdot,A):L^1(D,\R^q)\to [0,+\infty]$ that is finite only for $u\in BV(A,\S)$, where it takes the form
\begin{equation*}
E_{u_0}(u,A)=
\int_{S_{u_{A,0}}\cap \overline{A}}\phi(x,u_{A,0}^+,u_{A,0}^-,\nu_{u_{A,0}})\,\mathrm{d}\Hk.
\end{equation*}
\end{theorem}

\begin{proof}
By Proposition \ref{lb} we know that the limit energy is finite only for $u\in BV(A,\S)$. To save notation, we replace the subsequence $\e_n$ again by $\e$. \smallskip\\		
\noindent {\bf Lower bound:}	
Without loss of generality let $u_{\e}\to u$ on $A$ in the sense of Definition \ref{defconv} be such that 
\begin{equation}\label{finiteenergy}
\liminf_{\e} E^{l_{\e}}_{\e,u_0}(u_{\e},A)\leq C.
\end{equation}	
Passing to a subsequence, we may assume that $u_{\e}\in \mathcal{PC}_{\e,u_0}^{l_{\e}\e}(\L,A)$. We define a new sequence $v_{\e}:\e\L\to\S$ by	
\begin{equation*}
v_{\e}(\e x)=\mathds{1}_{A}( P_k(\e x))u_{\e}(\e x)+(1-\mathds{1}_{A}( P_k(\e x)))u_{0}(\e P_k(x)).
\end{equation*}	
Note that  by our assumptions on $u_0$ we have $v_{\e} \to u_{A,0}$ on $D$ in the sense of Definition \ref{defconv}. Now fix $A_1\subset\subset A\subset\subset A_2$ such that $A_1,A_2\in\Ard$. Setting	
\begin{equation*}
S^{\xi,\e}:=\{\alpha\in R^{\xi}_{\e}(A_2):\;\e x_{\a}\in P_k^{-1}A,\,\e x_{\a+\xi}\notin P_k^{-1}A\text{ or vice versa}\},
\end{equation*}
it holds that 	
\begin{align}
E_{\e}(v_{\e},A_2)\leq& E_{\e,u_0}^{l_{\e}}(u_{\e},A)+E_{\e}(u_{\e,0},A_2\backslash \overline{A_1})\nonumber
\\
&+\sum_{\xi\in\rzd_M}\sum_{\a\in S^{\xi,\e}}\e^{k-1}f_{\e}(x_{\alpha},x_{\a+\xi},v_{\e}(\e x_{\alpha}),v_{\e}(\e x_{\a+\xi})),\label{fundestimate}
\end{align}
Given $\delta>0$, let $L_{\delta}>0$ be such that (\ref{lrdecay}) holds. To bound the long-range interactions, we fix again a large cube $Q_D$ containing $\overline{D}$. Then Lemma \ref{longtoshort} and the coercivity assumption in Hypothesis 1 yield
\begin{align}\label{neglectlr}
\sum_{|\xi|>L_{\delta}}\sum_{\a\in S^{\xi,\e}}&\e^{k-1}f_{\e}(x_{\a},x_{\a+\xi},v_{\e}(\e x_{\a}),v_{\e}(\e x_{\a+\xi}))\leq C\sum_{|\xi|>L_{\delta}}J_{lr}(|\hat{\xi}|)|\xi|\sum_{\substack{(x,y)\in\mathcal{N\!N}(\L)\\ \e x,\e y\in P_k^{-1}Q_D}}\e^{k-1}f_{\e}(x,y,v_{\e}(\e x),v_{\e}(\e y))\nonumber
\\
&\leq C\delta\bigg(E_{\e}(u_{\e},A)+E_{\e}(u_{\e,0},Q_D)+\sum_{|\xi|\leq L_{\delta}}\sum_{\a\in S^{\xi,\e}}\e^{k-1}f_{\e}(x_{\a},x_{\a+\xi},v_{\e}(\e x_{\a}),v_{\e}(\e x_{\a+\xi}))\bigg) .
\end{align}	
For interactions with $|\xi|\leq L_{\delta}$ and $\e$ small enough, we have that $S^{\xi,\e}\subset A_2\backslash\overline{A_1}$. Moreover, if $l_{\e}>L_{\delta}+2r$, then by the boundary conditions on $u_{\e}$ we get
\begin{equation*}
\sum_{|\xi|\leq L_{\delta}}\sum_{\a\in S^{\xi,\e}}\e^{k-1}f_{\e}(x_{\alpha},x_{\a+\xi},v_{\e}(\e x_{\alpha}),v_{\e}(\e x_{\a+\xi}))
\leq E_{\e}(u_{\e,0},A_2\backslash\overline{A_1}).
\end{equation*}	
From the local version of Theorem \ref{mainthm1}, (\ref{nojump}), (\ref{finiteenergy}), (\ref{fundestimate}) and (\ref{neglectlr}) we infer	
\begin{equation*}
E(u_{A,0},A_2)\leq\liminf_{\e}E^{l_{\e}}_{\e,u_0}(u_{\e},A)+C\,\delta(1+\mathcal{H}^{d-1}(S_{u_0}\cap \overline{Q_D}))+C\mathcal{H}^{d-1}(S_{u_0}\cap \overline{A_2}\backslash A_1).
\end{equation*}	
The lower bound follows by letting $A_2\downarrow\overline{A}$ and $A_1\uparrow A$ combined with (\ref{nojump}) and the arbitrariness of $\delta$.\smallskip
\\
\noindent {\bf Upper bound:}
We first provide a recovery sequence in the case when $u=u_{0}$ in a neighbourhood of $\partial A$. Let $u_{\e}:\e\L\to\S$ converge to $u$ on $D$ in the sense of Definition \ref{defconv} and be such that	
\begin{equation}\label{rec}
\lim_{\e\to 0}E_{\e}(u_{\e},A)=E(u,A).
\end{equation}
Again, given $\delta>0$ we let $L_{\delta}>0$ be such that (\ref{lrdecay}) holds. Now choose regular sets $A_1\subset\subset A_2\subset\subset A$ such that	
\begin{equation}	
u=u_{0}\quad\text{on }A\backslash \overline{A_1}.\label{bordo}
\end{equation}	
The remaining argument is similar to the proof of Proposition \ref{subadditive} and therefore we only sketch it. Fix $d^{\prime}\leq\frac{1}{2}\dist(A_1,\partial A_2)$ and set $N_{\e}=\lfloor\frac{d^{\prime}}{\e( L_{\delta}+2r)}\rfloor$. For $j\in\mathbb{N}$ we define the sets	
\begin{equation*}
A_{\e,j}:=\{x\in A:\;\dist(x,A_1)<j\e (L_{\delta}+2r)\}.
\end{equation*}	
We further define $u^j_{\e}:\e\L\to\S$ setting	
\begin{equation*}
u^j_{\e}(\e x)=
\begin{cases}
u_{0}(\e x) &\mbox{if $P_k(\e x)\notin A_{\e,j}$,}\\
u_{\e}(\e x) &\mbox{otherwise.}
\end{cases}
\end{equation*}	
It holds that 
\begin{align*}
E_{\e}(u^j_{\e},A)\leq &E_{\e}(u_{\e},A)+E_{\e}(u_{\e,0},A\backslash \overline{A_1})
+\sum_{\xi\in\rzd_M}\e^{k-1}\sum_{\a\in R^{\xi}_{\e}(S_j^{\xi,\e})}f_{\e}(x_{\a},x_{\a+\xi},u^j_{\e}(\e x_{\a}),u^j_{\e}(\e x_{\a+\xi})),
\end{align*}	
where the set $S_j^{\xi,\e}$ is defined as 
\begin{equation*}
S_j^{\xi,\e}:=\{x=y+t\,P_k(\xi^{\prime}):\;y\in\partial A_{\e,j},\,|t|\leq\e,\,\xi^{\prime}\in\xi+[-r^{\prime},r^{\prime}]^d\}\cap A.
\end{equation*}	
As for (\ref{neglectlr}), using (\ref{nojump}) and (\ref{rec}) we can show that	
\begin{align*}
\sum_{\xi\in\rzd_M}&\e^{k-1}\sum_{\a\in R^{\xi}_{\e}(S_j^{\xi,\e})}f_{\e}(x_{\a},x_{\a+\xi},u^j_{\e}(\e x_{\a}),u^j_{\e}(\e x_{\a+\xi}))
\\
&\leq C\delta+C\sum_{|\xi|\leq L_{\delta}}\sum_{\a\in R^{\xi}_{\e}(S_j^{\xi,\e})}\e^{k-1}f_{\e}(x_{\a},x_{\a+\xi},u^j_{\e}(\e x_{\a}),u^j_{\e}(\e x_{\a+\xi})).
\end{align*}
To estimate the interactions where $|\xi|\leq L_{\delta}$, note that due to (\ref{bordo}) we can use the averaging technique like in Step $2$ of Proposition \ref{subadditive} to obtain $j_{\e}\in\{1,\dots,N_{\e}\}$ and the corresponding sequence $u^{j_{\e}}_{\e}$ satisfying the boundary conditions (at least for small $\e$ because of (\ref{slowboundary})) such that	
\begin{equation*}
\limsup_n E^{l_{\e}}_{\e,u_0}(u^{j_{\e}}_{\e},A)\leq E(u,A)+C\Hk(S_{u_0}\cap (\overline{A}\backslash A_1))+C\delta,
\end{equation*}
where we used (\ref{nojump}). Moreover, due to the assumptions on $u_0$ and (\ref{bordo}) we know that $u_{\e}^{j_{\e}}\to u$ on $A$. Letting first $\delta\to 0$ and then $A_1\uparrow A$ we finally get	
\begin{equation*}
\Gamma\hbox{-}\limsup_{\e} E^{l_{\e}}_{\e,u_0}(u,A)\leq E(u,A)=E_{u_0}(u,A).
\end{equation*}	
\hspace*{0,5cm}
For a general function $u\in BV(A,\S)$ we argue by approximation. To this end we take any $B\in\Ard$ such that $A\subset\subset B$. By Lemma \ref{strictinterior} we obtain a sequence $u_n\in BV(D,\S)$ such that $u_n=u_0$ in a neighbourhood of $\partial A$ and moreover $u_n\to u_{A,0}$ in $L^1(B)$ and $\mathcal{H}^{k-1}(S_{u_n}\cap B)\to\mathcal{H}^{k-1}(S_u\cap B)$. By $L^1(A)$-lower semicontinuity and the previous argument we obtain
\begin{align*}
\Gamma\hbox{-}\limsup_{\e} E^{l_{\e}}_{\e,u_0}(u,A)\leq \liminf_n E(u_n,A)\leq\liminf_n E(u_n,B)=E(u_{A,0},B).
\end{align*}
In the last step we used the continuity assumption on the integrand and a Reshetnyak-type continuity result for functionals defined on partitions that is proven in \cite{R16}. Letting $B\downarrow\overline{A}$ we obtain the claim.
\end{proof}



\begin{remark}\label{remarkfiniterange}

\begin{itemize}
\item[(i)] It is a direct consequence of our proof, that if we have only finite range of interactions, that is $f^{\e}_{lr}(x,y)=0$ for $|x-y|\geq L$, then it is enough to take $l_{\e}\geq L$.
\item[(ii)] By Remark \ref{almostgamma} the above Theorem \ref{constrainedproblem} implies the usual convergence of minimizers in the spirit of $\Gamma$-convergence. 
\end{itemize}
\end{remark}

Finally we prove an auxiliary result about convergence of boundary value problems that holds without any continuity assumptions. This result will be useful to treat homogenization problems as in Section \ref{homogenization}. To this end we replace the discrete width $l_{\e}$ by a macroscopic value $\eta$ and then take first the limit when $\e\to 0$ and let $\eta\to 0$ in a second step. Given $\eta>0$ and $A\in\Ard$, we set
\begin{equation*}
\partial A_{\eta}=\{x\in A:\;{\rm dist}(x,\partial A)\leq\eta\}.
\end{equation*}
We let $u_0$ be as before. Using a similar notation to that in Theorem \ref{represent} we define the quantities
\begin{equation*}
\begin{split}
&m^{\eta}_{\e}(u_0,A)=\inf\{E_{\e}(v,A):\;v\in\mathcal{PC}_{\e,u_0}^{\eta}(\L,A)\},
\\
&m(u_0,A)=\inf\{E(v,A):\;v=u_0\text{ in a neighbourhood of }\partial A\},
\end{split}
\end{equation*}
where the limit functional $E$ is given (up to subsequences) by Theorem \ref{mainthm1}. Note that the mapping $\eta\mapsto m^{\eta}_{\e}(u_0,A)$ is non-decreasing. Then we have the following weak version of Theorem \ref{constrainedproblem}.
\begin{lemma}\label{approxminprob}
Let $\e_n$ and $E$ be as in Theorem \ref{mainthm1}.	Then it holds that
\begin{equation*}
\lim_{\eta\to 0}\liminf_nm^{\eta}_{\e_n}(u_0,A)=\lim_{\eta\to 0}\limsup_{n}m^{\eta}_{\e_n}(u_0,A)=m(u_0,A).
\end{equation*}
\end{lemma}
\begin{proof}
First note that by monotonicity the limits for $\eta\to 0$ are well-defined. Moreover, by the first assumption in (\ref{nojump}) we have that $m^{\eta}_{\e}(u_0,A)$ is equibounded. Now for any $n\in\mathbb{N}$ let $u_n\in\mathcal{PC}^{\eta}_{\e_n,u_0}(\L,A)$ be such that $m^{\eta}_{\e_n}(u_0,A)=E_{\e_n}(u_n,A)$. By Proposition \ref{compactness} we know that, up to a subsequence (not relabeled), $u_n\to u$ on $A$ and by the assumptions on $u_0$ it follows that $u=u_0$ on $\partial A_{\eta}$. Extending $u$ we can assume that $u$ is admissible in the infimum problem defining $m(u_0,A)$ and using Theorem \ref{mainthm1} we obtain
\begin{equation*}
m(u_0,A)\leq E(u,A)\leq\liminf_nE_{\e_n}(u_n,A)\leq\liminf_n m^{\eta}_{\e_n}(u_0,A).
\end{equation*}
Since $\eta$ is arbitrary, we conclude that $m(u_0,A)\leq\lim_{\eta\to 0}\liminf_nm^{\eta}_{\e_n}(u_0,A)$.
\\
\hspace*{0,5cm}
In order to prove the remaining inequality, given $\gamma>0$ we let $u\in BV(A,\S)$ be such that $u=u_0$ in a neighbourhood of $\partial A$ and $E(u,A)\leq m(u_0,A)+\gamma$. Now let $u_{n}:\e\L\to \S$ be a recovery sequence for $u$. Repeating the argument for the upper bound in Theorem \ref{constrainedproblem}, given $\delta>0$ we can modify $u_n$ to a function $\bar{u}_n\in\mathcal{PC}^{\eta}_{\e_n,u_0}(\L,A)$ for some $\eta=\eta(\delta)>0$ such that
\begin{equation*}
\limsup_n E_{\e_n}(\bar{u}_n,A)\leq E(u,A)+\delta,
\end{equation*}
By the choice of $u$ we obtain
\begin{equation*}
\lim_{\eta\to 0}\limsup_nm^{\eta}_{\e_n}(u_0,A)\leq\limsup_n E_{\e_n}(\bar{u}_n,A)+\delta
\leq m(u_0,A)+\gamma+\delta.
\end{equation*}
The claim now follows letting first $\delta\to 0$ and then $\gamma\to 0$.
\end{proof}

\section{Homogenization results for stationary lattices}\label{homogenization}
We now replace the deterministic lattice $\L$ by a random point set. In what follows we introduce the probabilistic framework. To this end let $(\Omega,\mathcal{F},\mathbb{P})$ be a probability space with a complete $\sigma$-algebra $\mathcal{F}$.

\begin{definition}
We say that a family $(\tau_z)_{z\in \mathbb{Z}^k},\tau_z:\Omega\rightarrow\Omega$, is an {\em additive group action} on $\Omega$ if
\begin{equation*}
\tau_{z_1+z_2}=\tau_{z_2}\circ\tau_{z_1}\text{ for all } z_1,z_2\in\mathbb{Z}^k.
\end{equation*} Such an additive group action is called {\em measure preserving} if
\begin{equation*}
\mathbb{P}(\tau_z B)=\mathbb{P}(B)\text{ for all }  B\in\mathcal{F},\,z\in\mathbb{Z}^k.
\end{equation*}
Moreover $(\tau_z)_{z\in\mathbb{Z}^k}$ is called {\em ergodic} if, in addition, for all $B\in\mathcal{F}$ we have the implication
\begin{equation*}
(\tau_z(B)=B\text{ for all }  z\in \mathbb{Z}^k)\quad\Rightarrow\quad\mathbb{P}(B)\in\{0,1\}.
\end{equation*}
\end{definition}
\hspace*{0.5cm}
For general $m\in\mathbb{N}$ we denote by $[a,b):=\{x\in\R^{m}:\;a_i\leq x_i<b_i\;\text{ for all } i\}$ the $m$-dimensional coordinate parallelepiped with opposite vertices $a$ and $b$, and we set $\mathcal{I}_m=\{[a,b):a,b\in \mathbb{Z}^m, a\neq b\}$. Next, we introduce the notion of regular families and discrete subadditive stochastic processes:

\begin{definition}\label{regular}
Let $\{I_n\}\subset\mathcal{I}_m$ be a family of sets. Then $\{I_n\}$ is called {\em regular} if there exists another family $\{I^{'}_n\}\subset\mathcal{I}_m$ and a constant $C>0$ such that	
\begin{itemize}
	\item[(i)] $I_n\subset I^{'}_n\text{ for all }  n$,
	\item[(ii)] $I^{'}_{n_1}\subset I^{'}_{n_2}$ whenever $n_1<n_2$,
	\item[(iii)] $0<\mathcal{H}^m(I^{'}_n)\leq C\,\mathcal{H}^m(I_n)\text{ for all } n$.
\end{itemize}
Moreover, if $\{I^{'}_n\}$ can be chosen such that $\mathbb{R}^m=\bigcup_n I^{'}_n$, then we write $\lim_{n\to\infty}I_n=\mathbb{R}^m$.
\end{definition}

\begin{definition}{\label{discreteprocess}}
A function $\mu:\mathcal{I}_m\to L^1(\Omega)$ is said to be a {\em discrete subadditive stochastic process} if the following properties hold $\mathbb{P}$-almost surely:	
\begin{itemize}
\item [(i)] for every $I\in\mathcal{I}_m$ and for every finite partition $(I_j)_{j\in J}\subset\mathcal{I}_m$ of $I$ we have
\begin{equation*}
\mu(I,\omega)\leq \sum_{j\in J}\mu(I_j,\omega).
\end{equation*}
\item [(ii)] $
\inf\left\{\frac{1}{\mathcal{H}^m(I)}\int_{\Omega}\mu(I,w)\;\mathrm{d}\mathbb{P}(\omega):\;I\in\mathcal{I}_m\right\} >-\infty.
$
\end{itemize}
\end{definition}

One of the key ingredients for our stochastic homogenization result will be the following pointwise ergodic theorem (see Theorem 2.7 in \cite{AkKr}).

\begin{theorem}{\label{ergodicthm}}
Let $\mu:\mathcal{I}_m\to L^1(\Omega)$ be a discrete subadditive stochastic process and let $I_n$ be a regular family in $\mathcal{I}_m$. If $\mu$ is stationary with respect to a measure-preserving group action $(\tau_z)_{z\in\mathbb{Z}^{m}}$, that means
\begin{equation*}
\text{for all } I\in\mathcal{I}_m,\; z\in\mathbb{Z}^{m}\ \ \mu(I+z,\omega)=\mu(I,\tau_z\omega)\ \text{almost surely},
\end{equation*}
then there exists $\mu^{\infty}:\Omega\to\mathbb{R}$ such that, for $\mathbb{P}$-almost every $\omega$,
\begin{equation*}
\lim_{n\to +\infty}\frac{\mu(I_n,\omega)}{\mathcal{H}^m(I_n)}=\mu^{\infty}(\omega).
\end{equation*}
\end{theorem}

The statement is written for a generic $m$ since in this section we will use Theorem \ref{ergodicthm} for $m=k-1$, while in the next one also for $m=k$. We require some geometric and probabilistic properties of the random point set. 

\begin{definition}\label{defstatiolattice}
A random variable $\mathcal{L}:\Omega\rightarrow (\R^d)^{\mathbb{N}}$, $\w\mapsto \Lw=\{\Lw_i\}_{i\in\mathbb{N}}$ is called a {\em stochastic lattice}. We say that $\mathcal{L}$ is a {\em thin admissible lattice} if $\Lw$ is a thin admissible lattice in the sense of Definition \ref{flatadmissible} and the constants $M,r,R$ can be chosen independent of $\w$ $\mathbb{P}$-almost surely. The stochastic lattice $\mathcal{L}$ is said to be {\em stationary} if there exists a measure-preserving group action $(\tau_z)_{z\in\mathbb{Z}^k}$ on $\Omega$ such that, for $\mathbb{P}$-almost every $\w\in\Omega$,
\begin{equation*}
\mathcal{L}(\tau_z\omega)=\Lw+z.
\end{equation*}
If in addition $(\tau_z)_{z\in\mathbb{Z}^k}$ is ergodic, then $\mathcal{L}$ is called {\em ergodic}, too.
\end{definition}

In order to prove a homogenization result we make the following {\em structural assumption}\/: \\

{\bf Hypothesis 2} There exist functions $f_{nn},f_{lr}:\R^k\times\R^{2(d-k)}\rightarrow [0,+\infty)$ such that, setting $\Delta_k(x,y)=(y_1-x_1,\dots,y_k-x_k,x_{k+1},y_{k+1},\dots,x_d,y_d)$, it holds
\begin{equation*}
f_{nn}^{\e}(x,y)=f_{nn}(\Delta_k(x,y)),\quad
f_{lr}^{\e}(x,y)=f_{lr}(\Delta_k(x,y)).
\end{equation*}
Note that nearest-neighbour and long-range interaction coefficients are deterministic, but the set of nearest neighbours becomes now random. In the following we let $E_{\e}(\w)$ be the discrete energy defined in the previous section, with the stochastic lattice $\Lw$ in place of $\L$. As a general rule we will replace $\L$ by $\omega$ to indicate the dependence on the stochastic lattice $\Lw$. \\
\hspace*{0,5cm}
In view of Theorem \ref{represent} and Lemma \ref{approxminprob} we can further characterize the $\Gamma$-limits of the family $E_{\e}(\w)$ by investigating the quantities $m^{\eta}_{\e}(u_0,Q)$ for suitable oriented cubes and $u_0=u^{ij}_{x,\nu}$. Due to the decay assumptions of Hypothesis 1 it will be enough to consider truncated interactions. To this end, for fixed $L\in\mathbb{N}$ we will replace the long-range coefficients by
\begin{equation*}
f^L_{lr}(x,y):=f_{lr}(\Delta_k(x,y))\mathds{1}_{|x-y|\leq L}
\end{equation*}
and denote the corresponding energy by $E^L_{\e}(\w)(u,A)$. By Remark \ref{remarkfiniterange} the $\Gamma$-limit of the truncated energies is characterized by the minimum problem defined below: For $s_i,s_j\in\S$, $\nu\in S^{k-1}$ and a cube $Q_{\nu}(x,\rho)$ we set
\begin{equation}\label{randommin}
m_1^{\eta,L}(\w)(u^{ij}_{x,\nu},Q_{\nu}(x,\rho)):=\inf\left\{E^L_1(\w)(u,Q_{\nu}(x,\rho)):\; u\in \mathcal{PC}_{1,u^{ij}_{x,\nu}}^{\eta}(\w,Q_{\nu}(x,\rho))\right\}.
\end{equation}
The following technical auxiliary result will be used several times.
\begin{lemma}\label{stable}
Let $Q=Q_{\nu}(z,\rho)\subset\R^k$ be a cube and let $\{Q_n=Q_{\nu}(z_n,\rho_n)\}_n$ be a finite family of disjoint cubes with the following properties:

\begin{itemize}
	\item[(i)] $\min_n\rho_n\geq 4L$,  
	\item[(ii)] $z_n-z_1\in \{\nu\}^\perp$,
	\item[(iii)] $\dist(z_1,\{\nu\}^\perp+z)\leq \frac{1}{4}\min_n\rho_n$,
	\item[(iv)] $\bigcup_{n}Q_n\subset Q$,
	\item[(v)] {\rm either} $\dist(\partial \bigcup_nQ_n,\partial Q)> \eta$ {\rm or} $z_1-z\in\{\nu\}^\perp$.
\end{itemize}
\smallskip
Then there exists $C=C_L>0$ such that for all $\eta\geq L$
\begin{align*}
m_1^{\eta,L}(\w)(u^{ij}_{z,\nu},Q)\leq& \sum_{n}m_1^{\eta,L}(\w)(u^{ij}_{z_n,\nu},Q_n)+C\mathcal{H}^{k-1}\Big(\Big(Q\backslash\bigcup_n \overline{Q_n}\Big)\cap(\{\nu\}^\perp+z)\Big)
\\
&+C\sum_{n}\Big(\mathcal{H}^{k-2}\left((\partial Q_n\backslash\partial Q)\cap(\{\nu\}^\perp+z_1)\right)+\mathcal{H}^{k-1}(\partial Q_n\cap S_{\nu}(z,z_1))\Big),
\end{align*}
where $S_{\nu}(z,z_1)$ is the infinite (possibly, flat) stripe enclosed by the two hyperplanes $\{\nu\}^\perp+z$ and $\{\nu\}^\perp+z_1$.
\end{lemma}

\begin{proof}
During this proof, given $y\in\R^k$, we denote by $P_{\nu,y}$ the projection onto the affine space $\{\nu\}^\perp+y$. For each $n$ let $u_n$ be a minimizer for the problem in (\ref{randommin}) with $Q_{\nu}(x,\rho)=Q_n$. By assumptions (ii) and (v), the function $v:\Lw\to\S$ defined as
\begin{equation*}
v(x)=\begin{cases}
u_n(x) &\mbox{if $P_k(x)\in \overline{Q_n}$ for some $n$,}\\
u^{ij}_{z,\nu}(P_k(x)) &\mbox{otherwise}
\end{cases}
\end{equation*}
is well-defined and belongs to $\mathcal{PC}_{1,u^{ij}_{z,\nu}}^{\eta}(\w,Q)$. For $x,y\in\Lw\cap Q$ with $|x-y|\leq L$, we say that
\begin{itemize}
	\item[(I)] {\rm  holds if }$P_k(x)\in \overline{Q_n}$ and $P_k(y)\in \overline{Q_m}$ for $n\neq m$ or $P_k(x),P_k(y)\in\partial Q_n$,
	\item[(II)] {\rm  holds if }$P_k(x)\in Q\backslash\bigcup_n\overline{Q_n}$ and $P_k(y)\in \overline{Q_n}$ for some $n$.
\end{itemize}
By (iv) and Hypothesis 1 we can estimate
\begin{align}\label{splittingestimate}
m_1^{\eta,L}(\w)(u^{ij}_{z,\nu},Q)\leq E_1^L(\w)(v,Q)\leq& \sum_nm_1^{\eta,L}(\w)(u^{ij}_{z_n,\nu},Q_n)+E_1^L(\w)\Big(v,Q\backslash\bigcup_n\overline{Q_n}\Big)\nonumber\\
&+C\sum_{\substack{|x-y|\leq L\\\text{(I) or (II) hold}}}|v(x)-v(y)|.
\end{align}
We start with estimating the contribution of $x,y\in Q\backslash\bigcup_n\overline{Q_n}$. Suppose that $v(x)\neq v(y)$. Then $P_k(x)$ and $P_k(y)$ lie on different sides of the hyperplane $\{\nu\}^\perp+z$. Then it holds true that $P_{\nu,z}(P_k(x))\in Q\backslash\bigcup_n\overline{Q_n}$, otherwise assumptions (i) and (iii) would imply
\begin{align*}
L\geq |P_k(x)-P_k(y)|\geq |P_k(x)-P_{\nu,z}(P_k(x))|\geq \frac{\rho_n}{2}-\frac{\rho_n}{4}\geq 2L.
\end{align*} 
Thus $\dist(P_k(x),(Q\backslash\bigcup_n\overline{Q_n})\cap(\{\nu\}^\perp+z))\leq L$ and, using the properties of Definition \ref{flatadmissible}, it follows that
\begin{equation}\label{outsideunion}
E_1^L(\w)\Big(v,Q\backslash\bigcup_n\overline{Q_n}\Big)\leq C\mathcal{H}^{k-1}\Big(\Big(Q\backslash\bigcup_n \overline{Q_n}\Big)\cap(\{\nu\}^\perp+z)\Big).
\end{equation}
Next we have to control the interactions in Case (I). Given such $x,y$ with $|x-y|\leq L$, we know that by the definition of $v$, the boundary conditions on the smaller cubes and (ii) that  $v(x)=u^{ij}_{z_1,\nu}(P_k(x))$ and $v(y)=u^{ij}_{z_1,\nu}(P_k(y))$, so that if they contribute to the energy we conclude from assumption (ii) that $P_k(x)$ and $P_k(y)$ must lie on different sides of the hyperplane $\{\nu\}^\perp+z_1$. We deduce that $|P_{\nu,z_1}(P_k(x))-P_k(x)|\leq L$. Since by (iv) the segment $[P_{\nu,z_1}(P_k(x)),P_{\nu,z_1}(P_k(y))]$ intersects the $(k-2)$-dimensional set $(\partial Q_n\backslash\partial Q)\cap (\{\nu\}^\perp+z_1)$, it follows that 
\begin{equation*}
\dist\left(P_k(x),(\partial Q_n\backslash\partial Q)\cap (\{\nu\}^\perp+z_1)\right)\leq 2L.
\end{equation*}
Again, by Definition \ref{flatadmissible} and the above inequality we derive the estimate
\begin{equation}\label{betweencubes}
\sum_{\substack{|x-y|\leq L\\ \text{(I) holds}}}|v(x)-v(y)|\leq C\sum_{n}\mathcal{H}^{k-2}\left((\partial Q_n\backslash\partial Q)\cap(\{\nu\}^\perp+z_1)\right).
\end{equation}
It remains to estimate the contributions coming from Case (II). For such $x,y$ with $|x-y|\leq L$, due to the boundary conditions on the smaller cubes, a positive energy contribution implies $u^{ij}_{z,\nu}(P_k(x))\neq u^{ij}_{z_1,\nu}(P_k(y))$. Thus the segment $[P_k(x),P_k(y)]$ intersects $\partial Q_n$ in (at least) one point $x_n$ and also $S_{\nu}(z,z_1)$ in (at least) one point $x_S$. Denote by $x_{n,S}$ the projection of $x_S$ onto the facet of the cube $Q_n$ containing $x_n$. Since this facet cannot be parallel to $\{\nu\}^\perp$ by (i) and (iii), it holds $x_{n,S}\in \partial Q_n\cap S_{\nu}(z,z_1)$ and
\begin{equation*}
|P_k(x)-x_{n,S}|\leq |P_k(x)-x_S|+|x_S-x_{n,S}|\leq L+|x_S-x_n|\leq 2L,
\end{equation*}
which yields the estimate 
\begin{equation}\label{stripeestimate}
\dist(P_k(x),\partial Q_n\cap S_{\nu}(z,z_1))\leq 2L.
\end{equation}
This set may be not $(k-1)$-dimensional in the second possibility of (v). In this case one can bound the interactions by the right hand side of (\ref{outsideunion}). Otherwise, using (\ref{stripeestimate}) we obtain the estimate
\begin{equation}\label{cubeandnot}
\sum_{\substack{|x-y|\leq L\\ \text{(II) holds}}}|v(x)-v(y)|\leq C\sum_n\mathcal{H}^{k-1}(\partial Q_n\cap S_{\nu}(z,z_1)).
\end{equation}
In any case the claim now follows from (\ref{splittingestimate}), (\ref{outsideunion}), (\ref{betweencubes}) and (\ref{cubeandnot}).
\end{proof}

\begin{remark}\label{alsorectangles}
Lemma \ref{stable} still holds if we replace cubes by $k$-parallelepipeds of the type $I_{\nu}(z,\{\rho_m\}_m)=z+\{x\in\R^k:\;|\langle x,\nu_m\rangle |<\frac{\rho_m}{2}\}$. Then the cubes $Q_n$ are replaced by the collection $I_n=I_{\nu}(z_n,\{\rho_m^n\}_m)$ and in the assumptions (i) and (iii) we have to replace $\rho_n$ by $\min_m \rho_{m}^n$.  
\end{remark}

The next theorem is the main result of this section.

\begin{theorem}\label{mainthm2}
Let $\mathcal{L}$ be a stationary, thin admissible stochastic lattice and let $f_{nn}$ and $f_{lr}$ satisfy Hypotheses 1 and 2. 
For $\mathbb{P}$-almost every $\w$ and for all $s_i,s_j\in\S$ and $\nu\in S^{k-1}$ there exists
\begin{equation*}
\phi_{{\rm hom}}(\w;s_i,s_j,\nu):=\lim_{\eta\to 0}\limsup_{t\to +\infty}\frac{1}{t^{k-1}}\inf\left\{E_1(\w)(u,Q_{\nu}(0,t)):\; u\in \mathcal{PC}_{1,u^{ij}_{0,\nu}}^{\eta t}(\w,Q_{\nu}(0,t))\right\}.
\end{equation*}
The functionals $E_{\e}(\w)$ $\Gamma$-converge with respect to the convergence of Definition \ref{defconv} to the functional $E_{{\rm hom}}(\w):L^1(D,\R^q)\to [0,+\infty]$ defined by
\begin{equation*}
E_{{\rm hom}}(\w)(u)=
\begin{cases}\displaystyle
\int_{S_u} \phi_{{\rm hom}}(\w;u^+,u^-,\nu_u)\,\mathrm{d}\Hk &\mbox{if $u\in BV(D,\S)$,}\\
+\infty &\mbox{otherwise.}
\end{cases}
\end{equation*}
If $\mathcal{L}$ is ergodic, then $\w\mapsto \phi_{{\rm hom}}(\w,s_i,s_j,\nu)$ is almost-surely constant.
\end{theorem}

\begin{proof}
Fix any sequence $\e\to 0$.  According to Theorem \ref{mainthm1}, for all $\w\in\Omega$ such that $\Lw$ is admissible, there exists a ($\w$-dependent) subsequence $\e_{n}$ such that
\begin{equation*}
\Gamma\hbox{-}\lim_nE_{\e_{n}}(\w)(u,A)=\int_{S_u\cap A}\phi(\w;x,u^+,u^-,\nu)\,\mathrm{d}\mathcal{H}^{k-1}
\end{equation*}
for all $u\in BV(D,\S)$ and every $A\in\Ard$. According to Theorem \ref{represent} and Lemma \ref{approxminprob}, for any $x\in D$, $s_i,s_j\in\S$ and $\nu\in S^{k-1}$ it holds that
\begin{equation*}
\phi(\w;x,s_i,s_j,\nu)=\limsup_{\rho\to 0}\frac{1}{\rho^{k-1}}m(\w)(u^{ij}_{x,\nu},Q_{\nu}(x,\rho))
=\limsup_{\rho\to 0}\frac{1}{\rho^{k-1}}\lim_{\eta\to 0}\limsup_{n}m_{\e_{n}}^{\eta}(\w)(u^{ij}_{x,\nu},Q_{\nu}(x,\rho)).
\end{equation*}
If we change the variables via $t_n=\e_{n}^{-1}$ and $v(x)=u(t_n^{-1}x)$, the above characterization reads
\begin{equation*}
\phi(\w;x,s_i,s_j,\nu)=\limsup_{\rho\to 0}\lim_{\eta\to 0}\limsup_n\frac{1}{(\rho t_n)^{k-1}}m_1^{\eta t_n}(\w)(u^{ij}_{t_nx,\nu},t_nQ_{\nu}(x,\rho)).
\end{equation*}
Except for the claim on ergodicity, due to the Urysohn property of $\Gamma$-convergence (recall Remark \ref{almostgamma}) it is enough to show that for a set of full probability the limit in $\rho$ can be neglected and the remaining limits do not depend on $x$ or the subsequence $t_n$. We divide the proof into several steps.\\
\textbf{Step 1} Truncating the range of interactions\\
First we show that it is enough to consider the case of finite range interactions. We argue that it is enough to prove that there exists $\phi^L_{{\rm hom}}(\w;\nu)$ and a set $\Omega_L$ of full probability such that for all $\w\in\Omega_L$, $x\in D$, every cube $Q_{\nu}(x,\rho)$ and every sequence $t_n\to +\infty$ it holds
\begin{equation}\label{claimL}
\phi^L_{{\rm hom}}(\w;s_i,s_j,\nu)=\lim_{\eta\to 0}\limsup_n\frac{1}{(\rho t_n)^{k-1}}m_1^{\eta t_n,L}(\w)(u^{ij}_{t_nx,\nu},t_nQ_{\nu}(x,\rho)),
\end{equation}
where $m^{\eta t_n,L}_1(\w)$ is defined in (\ref{randommin}). Indeed, if (\ref{claimL}) is proven, then for all $\w\in\bigcap_L \Omega_L$ we find a configuration $v^L_n:\Lw\to\S$ with the correct boundary conditions (extended to the whole space) that minimizes $E^L_1(\w)(\cdot,t_nQ_{\nu}(x,\rho))$ in (\ref{randommin}). Using Lemma \ref{longtoshort} we obtain the estimate
\begin{align*}
0&\leq\frac{m_1^{\eta t_n}(\w)(u^{ij}_{t_nx,\nu},t_nQ_{\nu}(x,\rho))-m_1^{\eta t_n,L}(\w)(u^{ij}_{t_nx,\nu},t_nQ_{\nu}(x,\rho))}{(\rho t_n)^{k-1}}
\\
&\leq\frac{E_1(\w)(v^L_n,t_nQ_{\nu}(x,\rho))-E_1^L(\w)(v^L_n,t_nQ_{\nu}(x,\rho))}{(\rho t_n)^{k-1}}
\\
&\leq \frac{C}{(\rho t_n)^{k-1}}\sum_{2|\xi|>L}J_{lr}(|\hat{\xi}|)|\xi|\sum_{\substack{(x,y)\in\NNw \\ x,y\in (t_nQ_{\nu}(x,\rho))^{3(R+M)}}}f_{nn}(x,y,v^L_n(x),v^L_n(y)).
\end{align*}
The inner sum can be bounded by the energy plus interactions close to $\partial t_nQ_{\nu}(x,\rho)$. Due to the boundary conditions these are of order $(\rho t_n)^{k-2}$. Using the trivial a priori bound $m^{\eta}_1(\w)(u^{ij}_{t_nx,\nu},t_nQ_{\nu}(x,\rho))\leq C(\rho t_n)^{k-1}$ we deduce that
\begin{equation*}
0\leq\frac{m_1^{\eta t_n}(\w)(u^{ij}_{t_nx,\nu},t_nQ_{\nu}(x,\rho))-m_1^{\eta t_n,L}(\w)(u^{ij}_{t_nx,\nu},t_nQ_{\nu}(x,\rho))}{(\rho t_n)^{k-1}}\leq C\sum_{2|\xi|>L}J_{lr}(|\hat{\xi}|)|\xi|. 
\end{equation*}
Due to the integrability assumption of Hypothesis 1, we infer that $\phi^L_{{\rm hom}}(\w;s_i,s_j,\nu)$ is a Cauchy sequence with respect to $L$ and moreover, in combination with (\ref{claimL}), we deduce that
\begin{align*}
\lim_{L}\phi^L_{{\rm hom}}(\w;s_i,s_j,\nu)=\lim_{\eta\to 0}\limsup_n\frac{1}{(\rho t_n)^{k-1}}m_1^{\eta t_n}(\w)(u^{ij}_{t_nx,\nu},t_nQ_{\nu}(x,\rho))
\end{align*}
exists, is independent of $x,\rho$ and the sequence $t_n$. Therefore it remains to show (\ref{claimL}). For clarity of the argument we first consider an auxiliary problem where we replace the varying boundary width $\eta t_n$ by $L$. As an intermediate result we show that there exists
\begin{equation}\label{auxprocess}
\phi^L_{ij}(\w;\nu)=\lim_n\frac{1}{(\rho t_n)^{k-1}}m_1^{L,L}(\w)(u^{ij}_{t_nx,\nu},t_nQ_{\nu}(x,\rho))
\end{equation}
and this limit does not depend on $x,\rho$ and the sequence $t_n$.
\\
\textbf{Step 2} Existence of $\phi_{ij}^L$ for $x=0$ and rational directions\\
Fix $L\in\mathbb{N}$. We have to show that, for $\mathbb{P}$-almost every $\w\in \Omega$ and every $s_i,s_j\in\S$ and $\nu\in S^{k-1}$, there exists the limit in (\ref{auxprocess}). We start with the case $x=0$ and $\nu\in S^{k-1}\cap\mathbb{Q}^k$. For this choice we can use the subadditive ergodic theorem in $(k-1)$-dimensions.
\\
\textbf{Substep 2.1} Defining a stochastic process\\
We need a few preliminaries: Given $\nu\in S^{k-1}$ there exists an orthogonal matrix $A_{\nu}\in \mathbb{R}^{k\times k}$ such that $A_{\nu}e_k=\nu$, the mapping $\nu\mapsto A_{\nu}e_i$ is continuous on $S^{k-1}\backslash\{-e_k\}$ and if $\nu\in \mathbb{Q}^k$ then $A_{\nu}\in\mathbb{Q}^{k\times k}$ (it suffices to consider the orthogonal transformation that keeps the vector $\nu+e_k$ fix and reverses the orthogonal complement). We now fix a rational direction $\nu\in S^{k-1}\cap\mathbb{Q}^{k}$. Then there exists an integer $N=N(\nu)>4L$ such that $NA_{\nu}(z,0)\in\mathbb{Z}^k$ for all $z\in \mathbb{Z}^{k-1}$. We now define a discrete stochastic process (see Definition \ref{discreteprocess}). To $I=[a_1,b_1)\times \dots\times[a_{k-1},b_{k-1})\in\mathcal{I}_{k-1}$ we associate the set $Q_I\subset\mathbb{R}^k$ defined by 
\begin{equation*}
Q_I:=NA_{\nu}\Big(\text{int}\,I\times (-\frac{s_{\max}}{2},\frac{s_{\max}}{2})\Big),
\end{equation*}
where $s_{\max}=\max_{i}|b_i-a_i|$ is the maximal side length. Then we define the process $\mu:\mathcal{I}_{k-1}\to L^1(\Omega)$ as
\begin{equation}\label{process}
\mu(I,\w):=\inf\left\{E^L_1(\w)(v,Q_I):\,v\in \mathcal{PC}_{1,u^{ij}_{0,\nu}}^{L}(\w,Q_I) \right\}+C_{\mu}\mathcal{H}^{k-2}(\partial I),
\end{equation}
where $C_{\mu}$ is a constant to be chosen later. We first have to show that $\mu(I,\cdot)$ is a $L^1(\Omega)$-function. Testing the $\Lw$-interpolation of $u_{0,\nu}$ as candidate in the infimum problem, one can use the growth assumptions from Hypothesis 1 and Definition \ref{flatadmissible} to show that there exists a constant $C>0$ such that
\begin{equation}\label{minestimate}
\mu(I,\w)\leq CN^{k-1}\mathcal{H}^{k-1}(I)
\end{equation}
for all $I\in\mathcal{I}_{k-1}$ and almost every $\w\in\Omega$ so that $\mu(I,\cdot)$ is essentially bounded. $\mathcal{F}$-measurability can be proven similar to \cite[Lemma A.2]{ACR}.
\\
\hspace*{0,5cm}
We continue with proving lower-dimensional stationarity of the process. Let $z\in\mathbb{Z}^{d-1}$. Note that $Q_{I-z}=Q_I-z^N_{\nu}$, where $z^N_{\nu}:=NA_{\nu}(z,0)\in \{\nu\}^\perp\cap\mathbb{Z}^k$. By the stationarity of $\mathcal{L}$ it holds that $v\in \mathcal{PC}_{1,u^{ij}_{0,\nu}}^{L}(\w,Q_{I-z})$ if and only if $u(\cdot)=v(\cdot-z^N_{\nu})\in \mathcal{PC}_{1,u^{ij}_{0,\nu}}^{L}(\tau_{z^N_{\nu}}\w,Q_I)$. Moreover, by definition of the nearest neighbours, Hypothesis 2 and again stationarity of $\L$ we obtain that $E_1^L(\w)(v,Q_{I-z})=E_1^L(\tau_{z^N_{\nu}}\w)(u,Q_I)$. By the shift invariance of the Hausdorff measure we conclude that $\mu(I-z,\w)=\mu(I,\tau_{z^N_{\nu}}\w)$. Setting $\tilde{\tau}_z=\tau_{-z^N_{\nu}}$ we obtain a measure-preserving group action on $\mathbb{Z}^{k-1}$ such that $\mu(I,\tilde{\tau}_z\w)=\mu(I+z)(\w)$, which yields stationarity.
\\
\hspace*{0.5cm}
To show subadditivity, let $I\in\mathcal{I}_{k-1}$ and let $\{I_n\}_{n}\subset\mathcal{I}_{k-1}$ be a finite disjoint family such that $I=\bigcup_{n}I_n$. Note that $Q_I$ and the family $\{Q_{I_n}\}_{n}$ fulfill the assumptions of Lemma \ref{stable} (in the sense of Remark \ref{alsorectangles}). We conclude
\begin{equation*}
m_1^{L,L}(\w)(u^{ij}_{0,\nu},Q_I)\leq \sum_{n}m_1^{L,L}(\w)(u^{ij}_{0,\nu},Q_{I_n})+C\sum_{n}\mathcal{H}^{k-2}((\partial Q_{I_n}\backslash\partial Q_I)\cap\{\nu\}^\perp).
\end{equation*}
Applying the definition of $\mu(I,\w)$ yields
\begin{align*}
\mu(I,\w)&=m_1^{L,L}(\w)(u^{ij}_{0,\nu},Q_I)+C_{\mu}\mathcal{H}^{k-2}(\partial Q_I\cap\{\nu\}^\perp)
\\
&\leq \sum_{n}\mu(I_n,\w)+(C-C_{\mu})\sum_{n}\mathcal{H}^{k-2}((\partial Q_{I_n}\backslash Q_I)\cap\{\nu\}^\perp),
\end{align*}
which yields subadditivity if we choose $C_{\mu}>C$. Property (ii) in Definition \ref{discreteprocess} is trivial since $\mu(I,\w)$ is always nonnegative. By Theorem \ref{ergodicthm} there exists $\phi_{ij}^L(\w;\nu)$ such that almost surely, for rational directions $\nu\in S^{k-1}$, it holds
\begin{equation*}
\phi_{ij}^L(\w;\nu)=\lim_{n\to +\infty}\frac{1}{(2Nn)^{k-1}}m_1^{L,L}(\w)(u^{ij}_{0,\nu},Q_{\nu}(0,2Nn)),
\end{equation*}
where we used that the term $C_{\mu}\mathcal{H}^{k-2}(\partial I)$ is negligible for the limit.
\\
\textbf{Substep 2.2} From integer sequences to all sequences\\
Next we consider an arbitrary sequence $t_n\to +\infty$. From the previous step we know that
\begin{equation*}
\phi_{ij}^L(\w;\nu)=\lim_{n\to +\infty}\frac{1}{(2N\lfloor t_n\rfloor)^{k-1}}m_1^{L,L}(\w)(u^{ij}_{0,\nu},Q_{\nu}(0,2N\lfloor t_n\rfloor))
\end{equation*}
exists almost surely. To shorten notation we set $\Lambda_n=2Nt_n$ and $\lambda_n=2N\lfloor t_n\rfloor$. For $n$ large enough, we can apply Lemma \ref{stable} to the cube $Q_{\nu}(0,\Lambda_n)$ and singleton family $\{Q_{\nu}(0,\lambda_n)\}$ and obtain
\begin{align*}
m_1^{L,L}(\w)(u^{ij}_{0,\nu},Q_{\nu}(0,\Lambda_n))\leq &m_1^{L,L}(\w)(u^{ij}_{0,\nu},Q_{\nu}(0,\lambda_n))+\mathcal{H}^{k-2}(\partial (Q_{\nu}(0,\lambda_n))\cap\{\nu\}^\perp)
\\
&+C\mathcal{H}^{k-1}((Q_{\nu}(0,\Lambda_n)\backslash \overline{Q_{\nu}(0,\lambda_n)})\cap\{\nu\}^\perp)
\\
\leq &m_1^{L,L}(\w)(u_{0,\nu},Q_{\nu}(0,\lambda_n))+C\Lambda_n^{k-2},
\end{align*}
which yields
\begin{equation}\label{upperbound}
\limsup_{j\to +\infty}\frac{1}{\Lambda_n^{k-1}}m_1^{L,L}(\w)(u^{ij}_{0,\nu},Q_{\nu}(0,\Lambda_n))\leq\phi_{ij}^L(\w;\nu).
\end{equation}
Similarly, one can prove that
\begin{equation}\label{lowerbound}
\phi_{ij}^L(\w;\nu)\leq\liminf_{n\to +\infty}\frac{1}{\Lambda_n^{k-1}}m_1^{L,L}(\w)(u^{ij}_{0,\nu},Q_{\nu}(0,\Lambda_n)).
\end{equation}
Combining (\ref{upperbound}) and (\ref{lowerbound}) yields almost surely the existence of the limit for arbitrary sequences.
\\
\textbf{Substep 2.3} Shift invariance in the probability space\\
Up to neglecting a countable union of null sets we may assume that the limit defining $\phi_{ij}^L(\w;\nu)$ exists for all rational directions $\nu$. We next prove that the function $\w\mapsto\phi_{ij}^L(\w;\nu)$ is invariant under the entire group action $\{\tau_z\}_{z\in\mathbb{Z}^k}$. This will be important to treat the ergodic case but also for the shift invariance in the physical space. Given $z\in\mathbb{Z}^k$ there exists $R=R(L,z)>0$ such that for all $t>0$
\begin{equation}\label{movecube}
Q_{\nu}(0,t)\subset Q_{\nu}(-z,R+t),\quad 2L\leq{\rm dist}(\partial Q_{\nu}(0,t),\partial Q_{\nu}(-z,R+t)).
\end{equation}
Similar to the stationarity of the stochastic process we have
\begin{align*}
\phi_{ij}^L(\tau_z\w;\nu)&\leq\limsup_{t\to +\infty}\frac{1}{(R+t)^{k-1}}m_1^{L,L}(\w)(u^{ij}_{-z,\nu},Q_{\nu}(-z,R+t))
\\
&=\limsup_{t\to +\infty}\frac{1}{t^{k-1}}m_1^{L,L}(\w)(u^{ij}_{-z,\nu},Q_{\nu}(-z,R+t)).
\end{align*}
Due to (\ref{movecube}) we can apply Lemma \ref{stable} to the cube $Q_{\nu}(-z,R+t)$ and the singleton family $\{Q_{\nu}(0,t)\}$ and deduce that there exists a constant $C=C(R,z)$ such that
\begin{equation*}
m_1^{L,L}(\w)(u^{ij}_{-z,\nu},Q_{\nu}(-z,R+t))\leq m_1^{L,L}(\w)(u^{ij}_{0,\nu},Q_{\nu}(0,t))+Ct^{k-2}.
\end{equation*}
Hence we get $\phi_{ij}^L(\tau_z\w;\nu)\leq \phi_{ij}^L(\w;\nu)$. The other inequality can be proven similar so that the limit indeed exists (which we implicitly assumed with our notation) and, for $\mathbb{P}$-almost every $\w\in\Omega$,
\begin{equation}\label{groupinvariant}
\phi_{ij}^L(\tau_z\w;\nu)=\phi_{ij}^L(\w;\nu).
\end{equation}
\textbf{Step 3} Shift invariance in the physical space\\
In this step we prove the existence of the limit defining ${\phi}_{ij}^L(\w;\nu)$ when we blow up a cube not centered in the origin. We further show that it agrees with the one already considered. We start with considering a cube $Q_{\nu}(x,\rho)$ with rational direction $\nu$, $x\in\mathbb{Z}^k\backslash\{0\}$ and $\rho\in\mathbb{Q}$. Given $\e>0$ and $N\in\mathbb{N}$ (not the same one of Step 2.1) we define the events
\begin{equation*}
\mathcal{Q}_N:=\left\{\w\in\Omega:\;\sup_{t\geq\frac{N}{2}}\Bigl|(t\rho)^{1-k}m_1^{L,L}(\w)(u^{ij}_{0,\nu},Q_{\nu}(0,t\rho))-\phi_{ij}^L(\w;\nu)\Bigr|\leq\e\right\}.
\end{equation*}
By Step 2 we know that the function $\mathds{1}_{\mathcal{Q}_N}$ converges almost surely to $\mathds{1}_{\Omega}$ when $N\to +\infty$. Denote by $\mathcal{J}_x$ the $\sigma$-algebra of invariant sets for the measure-preserving map $\tau_x$. Fatou's lemma for the conditional expectation yields
\begin{equation}\label{condexp}
\mathds{1}_{\Omega}=\mathbb{E}[\mathds{1}_{\Omega}|\mathcal{J}_x]\leq\liminf_{N\to +\infty}\mathbb{E}[\mathds{1}_{\mathcal{Q}_N}|\mathcal{J}_x].
\end{equation} 
By (\ref{condexp}), given $\delta>0$, almost surely we find $N_0=N_0(\w,\delta)$ such that
\begin{equation*}
1\geq\mathbb{E}[\mathds{1}_{\mathcal{Q}_{N_0}}|\mathcal{J}_x](\w)\geq 1-\delta.
\end{equation*} 
Now due to Birkhoff's ergodic theorem, almost surely, there exists $n_0=n_0(\w,\delta)$ such that, for any $n\geq \frac{n_0}{2}$,
\begin{equation*}
\bigg|\frac{1}{n}\sum_{l=1}^n\mathds{1}_{\mathcal{Q}_{N_0}}(\tau_{l x}\w)-\mathbb{E}[\mathds{1}_{\mathcal{Q}_{N_0}}|\mathcal{J}_x](\w)\bigg|\leq\delta.
\end{equation*}
Note that the set we exclude will be a countable union of null sets provided $\e\in\mathbb{Q}$.
\\
\hspace*{0,5cm}
For fixed $n\geq\max\{n_0,N_0\}$ we denote by $R$ the maximal integer such that for all $l=n+1,\dots,n+R$ we have $\tau_{lx}(\w)\notin \mathcal{Q}_{N_0}$. In order to bound $R$ let $\tilde{n}$ be the number of ones in the sequence $\{\mathds{1}_{\mathcal{Q}_{N_0}}(\tau_{lx}(\w))\}_{l=1}^{n}$. By definition of $R$ we have
\begin{equation*}
\delta\geq\left|\frac{\tilde{n}}{n+R}-\mathbb{E}[1_{\mathcal{Q}_{N_0}}|\mathcal{J}_x](\w)\right|=\left|1-\mathbb{E}[1_{\mathcal{Q}_{N_0}}|\mathcal{J}_x](\w)+\frac{\tilde{n}-n-R}{n+R}\right|\geq \frac{R+n-\tilde{n}}{n+R}-\delta.
\end{equation*}
Since $n-\tilde{n}\geq 0$ and without loss of generality $\delta\leq\frac{1}{4}$, this provides an upper bound by $R\leq 4n\delta$. 
\\
\hspace*{0,5cm}
So for any $n\geq\max\{n_0,N_0\}$ and $\tilde{R}=6n\delta$ we find $l_n\in [n+1,n+\tilde{R}]$ such that $\tau_{l_nx}(\w)\in \mathcal{Q}_{N_0}$. Then by (\ref{groupinvariant}) and stationarity we have for all $t\geq \frac{N_0}{2}$ that
\begin{equation}\label{epsestimate}
\left|(t\rho)^{1-k}m_1^{L,L}(\w)(u^{ij}_{-l_nx,\nu},Q_{\nu}(-l_nx,t\rho))-\phi_{ij}^L(\w;\nu)\right|\leq\e.
\end{equation}
Define $\beta_n=n+c_L\rho^{-1}|x|(l_n-n)$, where $c_L\in\mathbb{N}$ is chosen such that $Q_{\nu}(-nx,n\rho)\subset Q_{\nu}(-l_nx,\beta_n\rho)$ and ${\rm dist}(\partial Q_{\nu}(-nx,n\rho),\partial Q_{\nu}(-l_nx,\beta_n\rho))>L$. Observe that such $c_L$ exists as $l_n-n\geq 1$. Then each face of the cube $Q_{\nu}(-nx,n\rho)$ has at most distance $(\beta_n-n)\rho=c_L|x|(l_n-n)$ to the corresponding face in $Q_{\nu}(-l_nx,\beta_n\rho)$. Then, for $n$ large enough, we can apply Lemma \ref{stable} to the cube $Q(-l_nx,\beta_n\rho)$ and the singleton family $\{Q_{\nu}(-nx,n\rho)\}$ to obtain
\begin{align}\label{firstergodicestimate}
\frac{m_1^{L,L}(\w)(u^{ij}_{-l_nx,\nu},Q_{\nu}(-l_nx,\beta_n\rho))}{(\beta_n\rho)^{k-1}}&\leq \frac{m_1^{L,L}(\w)(u^{ij}_{-nx,\nu},Q_{\nu}(-nx,n\rho))}{(\beta_n\rho)^{k-1}}+C\tilde{R}(\beta_n\rho)^{-1}\nonumber
\\
&\leq \frac{m_1^{L,L}(\w)(u^{ij}_{-nx,\nu},Q_{\nu}(-nx,n\rho))}{(n\rho)^{k-1}}+6C\delta.
\end{align}
On the other hand we can define $\theta_n=n-c^{\prime}_L\rho^{-1}|x|(l_n-n)$ for a suitable $c^{\prime}_L\in\mathbb{N}$ and deduce from a similar reasoning that
\begin{equation}\label{secondergodicestimate}
\frac{m_1^{L,L}(\w)(u^{ij}_{-nx,\nu},Q_{\nu}(-nx,n\rho))}{(n\rho)^{k-1}}\leq\frac{m_1^{L,L}(\w)(u^{ij}_{-l_nx,\nu},Q_{\nu}(-l_nx,\theta_n\rho))}{(\theta_n\rho)^{k-1}}+6C\delta. 
\end{equation}
Now if $\delta$ is small enough (depending only on $x,L$ and $\rho$) we have $\beta_n\geq \theta_n\geq\frac{n}{2}\geq\frac{N_0}{2}$. Combining (\ref{firstergodicestimate}),(\ref{secondergodicestimate}) and (\ref{epsestimate}) we infer
\begin{equation*}
\limsup_{n\to +\infty}\left|\frac{m_1^{L,L}(\w)(u^{ij}_{-nx,\nu},Q_{\nu}(-nx,n))}{n^{k-1}}-\phi_{ij}^L(\w;\nu)\right|\leq 6C\delta+\e,
\end{equation*}
which yields the claim in (\ref{auxprocess}) for $Q_{\nu}(x,\rho)$ with $x\in\mathbb{Z}^k$ and rational $\nu$ and $\rho$. The extension to arbitrary sequences $t_n\to +\infty$ (and thus to rational centers $x$) can be achieved again by Lemma \ref{stable} comparing first the minimal energy on the two cubes $Q_{\nu}(\lfloor t_n\rfloor x,\lfloor t_n\rfloor\rho)$ and $Q_{\nu}(\lfloor t_n\rfloor x,t_n\rho)$ similar to Substep 2.2 and then the energy on the latter cube with the one on $Q_{\nu}(t_nx,t_n\rho)$ as in Substep 2.3. Eventually the convergence of irrational $\rho$ follows from the estimate
\begin{equation*}
m_1^{L,L}(\w)(u^{ij}_{t_nx,\nu},Q_{\nu}(t_nx,t_n\rho))\leq m_1^{L,L}(\w)(u^{ij}_{t_nx,\nu},Q_{\nu}(t_nx,t_n(\rho-\delta))+Ct_n\delta (t_n\rho)^{k-2},
\end{equation*}
which is a consequence of Lemma \ref{stable} applied to the cube $Q_{\nu}(t_nx,t_n\rho)$ and $\{Q_{\nu}(t_nx,t_n(\rho-\delta))\}$, when one neglects lower-order terms. Choosing $0<\delta_l\to 0$ such that $\rho-\delta_l\in\mathbb{Q}$ then yields
\begin{equation*}
\limsup_n\frac{m_1^{L,L}(\w)(u^{ij}_{t_nx,\nu},Q_{\nu}(t_nx,t_n\rho))}{(t_n\rho)^{k-1}}\leq\phi_{ij}^L(\w;\nu).
\end{equation*}
Using the same argument for the cube $Q_{\nu}(t_nx,t_n(\rho+\delta))$ and the family $\{Q_{\nu}(t_nx,t_n\rho)\}$ we find that the limit exists and agrees with $\phi_{ij}^L(\w;\nu)$. Finally, for irrational centers we can again use a perturbation argument based on Lemma \ref{stable} as we did for proving (\ref{firstergodicestimate}) and (\ref{secondergodicestimate}). We omit the details.
\\ 
\textbf{Step 4} From rational to irrational directions
\\
Now we extend the convergence from rational direction to all $\nu\in S^{k-1}$. As the argument is purely geometric similar to Lemma \ref{stable}, we assume without loss of generality that $x=0$. First note that the set of rational directions is dense in $S^{k-1}$ (as the inverse of the stereographic projection maps rational points to rational directions). Given $\nu\in S^{k-1}$ and a sequence $t_n\to +\infty$ we define
\begin{equation*}
\begin{split}
&\overline{\phi}_{ij}^{L}(\w;\nu)=\limsup_{n\to +\infty}\frac{1}{t_n^{k-1}}m_1^{L,L}(\w)(u^{ij}_{0,\nu},Q_{\nu}(0,t_n)),
\\
&\underline{\phi}_{ij}^{L}(\w;\nu)=\liminf_{n\to +\infty}\frac{1}{t_n^{k-1}}m_1^{L,L}(\w)(u^{ij}_{0,\nu},Q_{\nu}(0,t_n)).
\end{split}
\end{equation*}
Let $\nu\in S^{k-1}\backslash \mathbb{Q}^k$. By the construction of the matrix $A_{\nu}$ in Substep 2.1 we can assume that there exists a sequence of rational directions $\nu_l$ such that $A_{\nu_l}\to A_{\nu}$. Therefore, given $\delta>0$ we find $l_0\in\mathbb{N}$ such that for all $l\geq l_0$ the following properties hold:
\begin{itemize}
	\item[(i)] $Q_{\nu}(0,(1-2\delta))\subset\subset Q_{\nu_l}(0,1-\delta)\subset\subset Q_{\nu}(0,1)$,
	\item[(ii)] $0<\text{d}_{\mathcal{H}}(\{\nu\}^\perp\cap B_2(0),\{\nu_l\}^\perp\cap B_2(0))\leq\delta.$
\end{itemize} 
For fixed $l\geq l_0$ and $n\in\mathbb{N}$ let $u_{n,l}:\Lw\to\S$ be an admissible minimizer for $m_1^{L,L}(\w)(u^{ij}_{0,\nu_l},Q_{\nu_l}(0,(1-\delta)t_n))$. We define a test function $v_n:\Lw\to\S$ setting
\begin{equation*}
v_n(x):=
\begin{cases}
u_{n,l}(x) &\mbox{if $x\in Q_{\nu_l}(0,(1-\delta)t_n)$,}
\\
u_{0,\nu}(x) &\mbox{otherwise.}
\end{cases}
\end{equation*}
Note that if $P_k(x),P_k(y)\in Q_{\nu}(0,t_n)\backslash Q_{\nu_l}(0,(1-\delta)t_l)$ are such that $|x-y|\leq L$ and $v_n(x)\neq v_n(y)$, then by the choice of $l_0$ and (i), for $l$ large enough we have
\begin{equation}\label{cubeboundary}
{\rm dist}\left(P_k(x),(Q_{\nu}(0,t_n)\backslash Q_{\nu}(0,(1-2\delta)t_n))\cap\{\nu\}^\perp\right)\leq L.
\end{equation} 
If $P_k(x)\in Q_{\nu}(0,t_n)\backslash Q_{\nu_l}(0,(1-\delta)t_n)$ and $P_k(y)\in Q_{\nu_l}(0,(1-\delta)t_n)$ with $|x-y|\leq L$ and $v_n(x)\neq v_n(y)$, then, for $l$ large enough one can show that by (ii) either $P_k(x)$ or $P_k(y)$ must lie in the cone
\begin{equation*}
\mathcal{K}(\nu,\nu_l)=\{x\in \mathbb{R}^k:\;\langle x,\nu\rangle\cdot\langle x,\nu_l\rangle\leq 0\}.
\end{equation*}
As the segment $[P_k(x),P_k(y)]$ intersects $\partial Q_{\nu_l}(0,(1-\delta)t_n)$, we conclude that
\begin{equation}\label{diffplanes}
{\rm dist}(P_k(x),\left(\mathcal{K}(\nu,\nu_l)+B_L(0)\right)\cap\partial Q_{\nu_l}(0,(1-\delta)t_n))\leq L.
\end{equation}
By (i) it holds that $v_n\in\mathcal{PC}_{1,u^{ij}_{0,\nu}}^{L}(\w,Q_{\nu}(0,t_n))$ for $n$ large enough. From (\ref{cubeboundary}), (\ref{diffplanes}) and the choice of $l_0$ we deduce that for $l$ large enough
\begin{equation*}
m_1^{L,L}(\w)(u^{ij}_{0,\nu},Q_{\nu}(0,t_n)\leq m_1^{L,L}(\w)(u^{ij}_{0,\nu_l},Q_{\nu_n}(0,(1-\delta)t_n))+C\delta t_n^{k-1}.
\end{equation*}
Dividing the last inequality by $t_n^{k-1}$ and taking the $\limsup$ as $n\to+\infty$ we deduce
\begin{equation*}
\overline{\phi}_{ij}^L(\w;\nu)\leq \phi_{ij}^L(\w;\nu_l)+C\delta.
\end{equation*}
Letting first $l\to +\infty$ and then $\delta\to 0$ yields $\overline{\phi}_{ij}^L(\w;\nu)\leq \liminf_{l}{\phi}_{ij}^L(\w;\nu_l)$.
By a similar argument we can also prove that $\limsup_l{\phi}_{ij}^L(\w;\nu_l)\leq\underline{\phi}_{ij}^L(\w;\nu)$. Hence, we get almost surely the existence of the limit in (\ref{auxprocess}) for all directions $\nu$ and the limit does not depend on $x,\rho$ and the sequence $t_n$.
\\
\textbf{Step 5} Proof of (\ref{claimL})
\\
We claim that $\phi_{ij}^L(\w;\nu)=\phi^L_{{\rm hom}}(\w;s_i,s_j,\nu)$. By the preceding steps this would conclude the proof. First observe that by monotonicity it is enough to show that $\phi^L_{{\rm hom}}(\w;s_i,s_j,\nu)\leq \phi_{ij}^L(\w;\nu)$. Let $t_n\to +\infty$ and fix a cube $Q_{\nu}(x,\rho)$. By a trivial extension argument, for $\eta$ small enough (depending on $\rho$) it holds that
\begin{equation*}
m_1^{\eta t_n,L}(\w)(u^{ij}_{t_nx,\nu},Q(t_nx,t_n\rho))\leq m_1^{L,L}(\w)(u^{ij}_{t_nx,\nu},Q(t_nx,t_n\rho-\eta t_n))+C\eta t_n^{k-1}.
\end{equation*} 
Dividing by $(t_n\rho)^{k-1}$ and letting first $n\to +\infty$ and then $\eta\to 0$ we obtain the claim.
\\
\hspace*{0.5cm}
When the group action is ergodic, the additional statement in Theorem \ref{mainthm2} follows from (\ref{groupinvariant}) since in this case all the functions $\w\mapsto\phi_{ij}^L(\w;\nu)$ are constant and so is the pointwise limit when $L\to +\infty$.
\end{proof}	

\begin{remark}\label{simplerformula}
One can show that the surface tension can be obtained by one single limit procedure. Indeed, referring to (\ref{slowboundary}) and repeating Step 1 and 5 of the proof of Theorem \ref{mainthm2} it follows that
\begin{equation*}
\phi_{{\rm hom}}(\w;s_i,s_j,\nu)=\lim_{t\to +\infty}\frac{1}{t^{k-1}}\inf\Bigl\{E_1(\w)(u,Q_{\nu}(0,t)):\,u\in\mathcal{PC}_{1,u^{ij}_{0,\nu}}^{l_{1/t}}(\w,Q_{\nu}(0,t))\Bigr\}.
\end{equation*}
\end{remark}

\section{Volume constraints in the stationary case}\label{sect-vol}
In this section we will discuss the variational limit of the energies $E_{\e}(\w)$ when, for all $i=1,\dots,q$, we fix the number of lattice points where the configuration takes the value $s_i$. For general thin admissible lattices this might not converge without passing to a further subsequence, so we treat only the case of stationary lattices in the sense of Definition \ref{defstatiolattice}. In order to formulate the result, given $A\in\mathcal{A}^R(D)$ and a family $V_{\e}=\{V_{i,\e}\}_{i=1}^q\in\mathbb{N}^q$, we introduce the class
\begin{equation*}
\mathcal{PC}_{\e}^{V_{\e}}(\w):=\{u:\e\mathcal{L}(\w)\to\S:\;\#\{\e x\in\e\mathcal{L}(\w)\cap P_k^{-1}D:\;u(\e x)=s_i\}=V_{i,\e}\}.
\end{equation*}
Beside the natural compatibility condition $\sum_{i}V_{i,\e}=\#(\e\Lw\cap P_k^{-1}D)$, we assume that for all $i=1,\dots,q$ there exists $V_i>0$ such that
\begin{equation*}
\lim_{\e\to 0}\frac{V_{i,\e}}{\#(\e\mathcal{L}\cap P_k^{-1}D)}= V_i.
\end{equation*}
Note that we exclude the case $V_i=0$ for some $i$. This case contains some non-trivial aspects which are related to the concept of $(B)$-convexity studied in \cite{AmBrII}. Such conditions are not necessarily satisfied by our discrete energies. Of course the extreme case $V_{i,\e}=0$ for all $\e>0$ can be treated by changing the set $\S$ and thus the whole model. 
\\
\hspace*{0.5cm}
The following lemma describes how the volume constraint behaves for sequences with finite energy.

\begin{lemma}\label{convvol}
For $\mathbb{P}$-almost all $\w\in\Omega$ the following statement holds true: For all $u\in BV(D,\S)$ such that there exists a sequence $u_{\e}:\e\mathcal{L}(\w)\to\S$ with $u_{\e}\to u$ in the sense of Definition \ref{defconv} and 
\begin{equation*}
\sup_{\e>0}E_{\e}(\w)(u_{\e})\leq C,\quad\quad\lim_{\e\to 0}\frac{\#\{\e x\in\e\mathcal{L}(\w)\cap P_k^{-1}D:\;u_{\e}(\e x)=s_i\}}{\#\{\e x\in\e\mathcal{L}(\w)\cap P_k^{-1}D\}}=V_i^{\prime}
\end{equation*}
we have 
\begin{equation*}
|\{u=s_i\}|=V_i^{\prime}|D|.
\end{equation*}
\end{lemma}

\begin{proof}
Up to the transformation $T(s_i)=e_i$ we may assume that the vectors $s_i$ form a basis. For $\w\in\Omega$ we consider the sequence of nonnegative Borel measures $\gamma_{\e}(\w)$ on $D$ defined as
\begin{equation*}
\gamma_{\e}(\w)=\sum_{z\in P_k(\mathcal{L}(\w))\cap\frac{D}{\e}}\e^k\#\left(P_k^{-1}(z)\cap\mathcal{L}(\w)\right)\delta_{\e z}.
\end{equation*}
As $\gamma_{\e}(\w)(D)\leq C|D|$, up to subsequences we know that $\gamma_{\e}(\w)\overset{*}{\rightharpoonup}\gamma(\w)$ in the sense of measures. We now identify the limit measure. To this end we define a discrete stochastic process $\gamma:\mathcal{I}_k\to L^1(\Omega)$ as
\begin{equation}\label{randomdensity}
\gamma(I)(\w):=\sum_{y\in P_k(\mathcal{L}(\w))\cap I}\#\left(P_k^{-1}(y)\cap\mathcal{L}(\w)\right)=\#\left(x\in\mathcal{L}(\w):\;P_k(x)\in I\right).
\end{equation}
It follows from (\ref{localratio}) that $\gamma(I)$ is essentially bounded for every $I\in\mathcal{I}_k$. In addition it can be checked that $\gamma(I)$ is $\mathcal{F}$-measurable, thus we infer that $\gamma(I)\in L^{\infty}(\Omega)$. Upon redefining the group action as $\tilde{\tau}_z=\tau_{-z}$, the process $\gamma$ is stationary and (sub)additive. By Theorem \ref{ergodicthm} there exists $\gamma_0(\w)$ such that for almost every $\w\in\Omega$ and all $I\in\mathcal{I}_k$ we have
\begin{equation*}
\lim_{n\to +\infty}\frac{\gamma(nI)(\w)}{n^k|I|}=\gamma_0(\w).
\end{equation*}
It is straightforward to extend this result to all sequences $t_n\to +\infty$ and then to all cubes in $\R^k$ by a continuity argument. Now let $a,b\in \R^{k}$ and let $Q=[a,b)$. Then by definition
\begin{equation}\label{cubeconvergence}
\lim_{\e\to 0}\gamma_{\e}(\w)(Q)=\lim_{\e\to 0}\sum_{z\in P_k(\mathcal{L}(\w))\cap \frac{1}{\e}Q}\e^k\#\left(P_k^{-1}(z)\cap\mathcal{L}(\w)\right)=\gamma_0(\w)|Q|.
\end{equation}
Given any open set $A\in\mathcal{A}(D)$, for $\delta>0$ we consider the following interior approximation:
\begin{equation*}
A_{\text{int}}(\delta)=\bigcup_{z\in\delta\mathbb{Z}^k:\;z+[0,\delta)^k\subset A}z+[0,\delta)^k.
\end{equation*}
It can be checked by monotone convergence that $\lim_{\delta\to 0}|A(\delta)|=|A|$. By (\ref{cubeconvergence}) and additivity we obtain
\begin{equation*}
\liminf_{\e\to 0}\gamma_{\e}(\w)(A)\geq \liminf_{\e\to 0}\gamma_{\e}(\w)(A(\delta))=\gamma_0(\w)|A(\delta)|.
\end{equation*}
Letting $\delta\to 0$ we obtain $\liminf_{\e}\gamma_{\e}(\w)(A)\geq\gamma_0(\w)|A|$. By the Portmanteau-Theorem we conclude that $\gamma(\w)(B)=\gamma_0(\w)|B|$ for all Borel sets $B\subset D$. In particular the whole sequence converges in the sense of measures. On the other hand, if $A\in\mathcal{A}(D)$ is such that $|\partial A|=0$, then the outer approximation
\begin{equation*}
A_{\text{out}}(\delta)=\bigcup_{z\in\delta\mathbb{Z}^k:\;z+[0,\delta)^k\cap A\neq\emptyset}z+[0,\delta)^k
\end{equation*}
also fulfills $\lim_{\delta\to 0}|A(\delta)|=|A|$, hence 
\begin{equation}\label{setconvergence}
\lim_{\e\to 0}\gamma_{\e}(\w)(A)=\gamma_0(\w)|A|
\end{equation}
for all $A\in\mathcal{A}(D)$ such that $|\partial A|=0$. Given now $\delta>0$, we take any polyhedral function $u_{\delta}\in BV_{loc}(\R^k,\S)$ such that $\|u-u_{\delta}\|_{L^1(D)}\leq\delta$. As $u_{\delta}$ is Borel-measurable, we have 
\begin{equation*}
\int_DP{u}_{\e}\,\mathrm{d}\gamma_{\e}(\w)=\int_D(P{u}_{\e}-u_{\delta})\,\mathrm{d}\gamma_{\e}(\w)+\int_D u_{\delta}\,\mathrm{d}\gamma_{\e}(\w).
\end{equation*}
Since $u_{\delta}$ is a polyhedral function, we can use (\ref{setconvergence}) to obtain \begin{equation}\label{testpolyhedral}
\lim_{\e\to 0}\int_Du_{\delta}\,\mathrm{d}\gamma_{\e}(\w)=\gamma_0(\w)\int_Du_{\delta}\,\mathrm{d}x.
\end{equation}
What concerns the first term, by (\ref{thickness}) and the regularity of $S_{u_{\delta}}$ and $\partial D$ we have
\begin{align}
\left|\int_D(P{u}_{\e}-u_{\delta})\,\mathrm{d}\gamma_{\e}(\w)\right|&\leq C\sum_{z\in P_k(\mathcal{L}(\w))\cap \frac{D}{\e}}\e^k|P{u}_{\e}(\e z)-u_{\delta}(\e z)|\label{inttosum}
\end{align}
Now using the fact that ${u}_{\e}$ has equibounded energy, one can reason as in the proof of Lemma \ref{sumconv} to show that
\begin{equation*}
\limsup_{\e\to 0}\sum_{z\in P_k(\mathcal{L}(\w))\cap \frac{D}{\e}}\e^k|P{u}_{\e}(\e z)-u_{\delta}(\e z)|\leq C\|u-u_{\delta}\|_{L^1(D)}\leq C\delta.
\end{equation*}
Combining the above inequality with (\ref{testpolyhedral}) and (\ref{inttosum}) we finally obtain by the arbitrariness of $\delta$ that
\begin{equation*}
\lim_{\e\to 0}\int_DP{u}_{\e}\,\mathrm{d}\gamma_{\e}(\w)=\gamma_0(\w)\int_Du\,\mathrm{d}x=\gamma_0(\w)\sum_{i=1}^qs_i|\{u=s_i\}|
\end{equation*}
On the other hand, plugging in the definition and using again (\ref{setconvergence}), it holds
\begin{align*}
\lim_{\e\to 0}\int_DP{u}_{\e}\,\mathrm{d}\gamma_{\e}(\w)&=\lim_{\e\to 0}\sum_{i=1}^qs_i\#\{\e x\in\e\mathcal{L}(\w)\cap D:\,u_{\e}(\e x)=s_i\}\e^k\\
&=\sum_{i=1}^qs_iV_i^{\prime}|D|\gamma_0(\w).
\end{align*}
Since we assume the $s_i$ to form a basis we conclude the proof.
\end{proof}

In order to include the volume constraint in the functional, for almost every $\w\in\Omega$ we introduce $E_{\e}^{V_{\e}}(\w):\mathcal{PC}_{\e}(\w)\to[0,+\infty]$ as
\begin{equation*}
E_{\e}^{V_{\e}}(\w)(u)=
\begin{cases}
E_{\e}(\w)(u) &\mbox{if $u\in \mathcal{PC}_{\e}^{V_{\e}}(\w)$,}
\\
+\infty &\mbox{otherwise.}
\end{cases}
\end{equation*}
With the help of Lemma \ref{convvol} we can now prove the following theorem.

\begin{theorem}
Let $\mathcal{L}$ be a stationary stochastic lattice and let $f_{nn}$ and $f_{lr}$ satisfy Hypotheses 1 and 2. For $\mathbb{P}$-almost every $\w$ the functionals $E_{\e}^{V_{\e}}(\w)$ $\Gamma$-converge with respect to the convergence of Definition \ref{defconv} to the functional $E_{{\rm hom}}^{V}(\w):L^1(D,\R^q)\to[0,+\infty]$ defined by
\begin{equation*}
E_{{\rm hom}}^V(\w)(u)=
\begin{cases}\displaystyle
\int_{S_u}\phi_{{\rm hom}}(\w;u^+,u^-,\nu_u)\,\mathrm{d}\mathcal{H}^{k-1} &\mbox{if $u\in BV(D,\S)$ and $|\{u=s_i\}|=V_i|D|$ for all $i$},\\
+\infty &\mbox{otherwise.}
\end{cases}
\end{equation*}
\end{theorem}

\begin{proof}
The lower bound follows from Theorem \ref{mainthm2} and Lemma \ref{convvol}. In order to prove the upper bound, for the moment assume that $u\in BV(D,\S)$ satisfies the volume constraint and that each level set $\{u=s_i\}$ contains an interior point. In particular, in each level set we find $q$ disjoint open balls $B_{\eta}(x_i^l)\subset\subset\{u=s_i\}$ with $\eta<<1$. By Theorem \ref{mainthm2} we can find a sequence  $u_{\e}:\e\mathcal{L}(\w)\to\S$ such that $u_{\e}$ converges to $u$ in the sense of Definition \ref{defconv} and
\begin{equation}\label{recoverysequence}
\lim_{\e\to 0}E_{\e}(\w)(u_{\e})=E_{{\rm hom}}(\w)(u).
\end{equation}
Repeating the argument used for proving Proposition \ref{subadditive} one can show that without loss of generality we may assume that $u_{\e}(\e x)=s_i$ for all $\e x\in\e\Lw\cap B_{\eta}(x_i^l)$ and that $u_{\e}$ has equibounded energy on a large cube $Q_D$ containing $\overline{D}$. For each $i$ set $\tilde{V}_{i,\e}=\#\{\e x\in\e\mathcal{L}(\w)\cap P_k^{-1}D:\,u_{\e}(\e x)=s_i\}$. Applying Lemma \ref{convvol} we deduce that
\begin{equation}\label{minrate}
\lim_{\e\to 0}\frac{\tilde{V}_{i,\e}-V_{i,\e}}{\#\{\e x\in\e\mathcal{L}(\w)\cap P_k^{-1}D\}}=0.
\end{equation}
We now adjust the sequence $u_{\e}$ so that it belongs to $\mathcal{PC}_{\e}^{V_{\e}}(\w)$. This will be done locally on the balls $B_{\eta}(x_i^l)$. First we change the values on $B_{\eta}(x_1^1)$ and $B_{\eta}(x_2^1)$ so that the sequence satisfies the constraint for $i=1$. In general, for $i<q$ we change the sequence on $B_{\eta}(x_i^i)$ and $B_{\eta}(x_{i+1}^i)$ so that it satisfies the constraints for all $j\leq i$. At the end the constraint for $i=q$ follows by the compatibility assumption. Each modification will be such that $L^1$-convergence and convergence of the energies is conserved. We will provide the construction only for the first step. In what follows we consider the case $\tilde{V}_{1,\e}>V_{1,\e}$. We set $h_{\e}=(\tilde{V}_{1,\e}-V_{1,\e})^{\frac{1}{k}}$. Up to modifying $u_{\e}$ on a set of lattice points with diverging cardinality much less than $\e^{1-k}$ and contained in the complement of the union of the balls $B_{\eta}(x_i^l)$ (which yields again a recovery sequence), we may assume that $h_{\e}\to +\infty$.
\\
\hspace*{0,5cm}
Observe that (\ref{minrate}) and the properties of a thin admissible lattice imply that 
\begin{equation}\label{limith}
\lim_{\e\to 0}h_{\e}\e=0.
\end{equation}
We already know from the proof of Lemma \ref{convvol} that, almost surely, we can write
\begin{equation*}
q^{\w}(x_1^1,h_{\e}):=\#\{x\in\mathcal{L}(\w):\;P_k(x)\in Q_{e_1}(x_1^1,\gamma_0(\w)^{-1}h_{\e})\}=h_{\e}^k+h_{\e}^{k-1}\gamma_{\e},
\end{equation*}
for some sequence $\gamma_{\e}=\gamma_{\e}(\w,x_1^1)$ such that $\lim_{\e\to 0}\frac{\gamma_{\e}}{h_{\e}}=0$. In the following we assume that $\gamma_{\e}\leq 0$, but with a similar argument we can also treat the case $\gamma_{\e}>0$. As $\mathcal{L}(\w)$ is thin admissible in the sense of Definition \ref{flatadmissible}, one can show that for some appropriate $c=c(R)>0$ it holds true that
\begin{equation*}
\frac{1}{C}h_{\e}^{k-1}\leq q^{\w}(x_0,h_{\e}+n+c)-q^{\w}(x_0,h_{\e}+n)\leq Ch_{\e}^{k-1}
\end{equation*}
for any $0\leq n\leq h_{\e}$. In particular, there exist $n_{\e}=\mathcal{O}(\gamma_{\e})$ and nonnegative equibounded $c_{\e}$ such that
\begin{equation}\label{adjustedcube}
q^{\w}(x_0,h_{\e}+n_{\e})=h_{\e}^k+ c_{\e}h_{\e}^{k-1}.
\end{equation}
Now choose any set $G_{\e}\subset\R^d$ such that $P_kG_{\e}\subset B_{\eta}(x_2^1)$ and $\#\left(G_{\e}\cap\mathcal{L}(\w)\right)=c_{\e}h_{\e}^{k-1}$. To reduce notation, set $Q_{\e}:=Q_{e_1}(x_1^1,\gamma_0(\w)^{-1}\e(h_{\e}+n_{\e}))$. We define 
\begin{equation*}
\bar{u}_{\e}(\e x)=
\begin{cases}
s_2 &\mbox{if $\e P_k(x)\in Q_{\e}$,}\\
s_1 &\mbox{if $\e x\in G_{\e}$,}\\
u_{\e}(\e x) &\mbox{otherwise.}
\end{cases}
\end{equation*}
Note that by (\ref{limith}) we have $Q_{\e}\subset\subset B_{\eta}(x_1^1)$ for $\e$ small enough and therefore $\#\{\e x\in\e\Lw\cap P_k^{-1}D:\;\bar{u}(\e x)=s_1\}=V_{1,\e}$. Again by (\ref{limith}) we still have that $\bar{u}_{\e}\to u$ in the sense of Definition \ref{defconv}.  From Hypothesis 1 we deduce
\begin{align}\label{recenergy}
E_{\e}(\w)(\bar{u}_{\e})\leq& E_{\e}(\w)(u_{\e})+C\sum_{\xi\in\rzd_M}J_{lr}(|\hat{\xi}|)\#(G_{\e}\cap\e\Lw)\e^{k-1}\nonumber
\\
&+\sum_{\xi\in\rzd_M}\sum_{\substack{\a\in R^{\xi}_{\e}(D)\\ \e P_k([x_{\a}, x_{\a+\xi}])\cap\partial Q_{\e}\neq\emptyset}}\e^{k-1}f_{\e}(x_{\a},x_{\a+\xi},\bar{u}_{\e}(\e x_{\a}),\bar{u}_{\e}(\e x_{\a+\xi})).
\end{align} 
It remains to bound the last term since the second one vanishes by (\ref{limith}) and integrability of $J_{lr}$. We split the interactions according to (\ref{lrdecay}). By Lemma \ref{longtoshort} and Hypothesis 1, for $\e$ small enough we have by construction
\begin{align}\label{volshortrange}
&\sum_{|\xi|\leq L_{\delta}}\sum_{\substack{\a\in R^{\xi}_{\e}(D)\\ \e P_k([x_{\a}, x_{\a+\xi}])\cap\partial Q_{\e}\neq\emptyset}}\e^{k-1}f_{\e}(x_{\a},x_{\a+\xi},\bar{u}_{\e}(\e x_{\a}),\bar{u}_{\e}(\e x_{\a+\xi}))\nonumber\\
&\leq C\sum_{|\xi|\leq L_{\delta}}J_{lr}(|\hat{\xi}|)|\xi|\sum_{\substack{(x,y)\in\mathcal{NN}(\w)\\ \e x,\e y\in B_{\eta}(x_1^1)}}\e^{k-1}f_{\e}(x,y,\bar{u}_{\e}(\e x),\bar{u}_{\e}(\e y))
\leq C\mathcal{H}^{k-1}(\partial Q_{\e}) \leq C (\e h_{\e})^{k-1},
\end{align}
so that the left hand side vanishes when $\e\to 0$. To control the remaining interactions, recall that $u_{\e}$ has finite energy on the larger cube $Q_D$. Hence Lemma \ref{longtoshort} and Hypothesis 1 yield
\begin{align*}
&\sum_{|\xi|>L_{\delta}}\sum_{\substack{\a\in R^{\xi}_{\e}(D)\\ \e P_k([x_{\a}, x_{\a+\xi}])\cap\partial Q_{\e}\neq\emptyset}}\e^{k-1}f_{\e}(x_{\a},x_{\a+\xi},\bar{u}_{\e}(\e x_{\a}),\bar{u}_{\e}(\e x_{\a+\xi}))
\\
&\leq C\delta\sum_{\substack{(x,y)\in\mathcal{NN}(\w)\\ \e x,\e y\in Q_D}}\e^{k-1}f_{\e}(x,y,\bar{u}_{\e}(\e x),\bar{u}_{\e}(\e y))
\\
&\leq C\delta\left(E_{\e}(\w)(u_{\e},Q_D)+\mathcal{H}^{k-1}(\partial Q_{\e})+\# (G_{\e}\cap\e\Lw)\e^{k-1}\right)\leq C\delta.
\end{align*}
As $\delta>0$ was arbitrary, we infer from (\ref{recoverysequence}), (\ref{recenergy}) and (\ref{volshortrange}) that
\begin{equation*}
\limsup_{\e\to 0}E_{\e}(\w)(\bar{u}_{\e})=\limsup_{\e\to 0}E_{\e}(\w)(u_{\e})=E_{\rm hom}(\w)(u).
\end{equation*}
The case when $V^{\prime}_{\e}\leq V_{\e}$ can be treated by an almost symmetric argument. Repeating this construction for the remaining phases as described at the beginning of this proof, we obtain 
\begin{equation*}
\Gamma\hbox{-}\limsup_{\e\to 0}E_{\e}^{V_{\e}}(\w)(u)=E_{{\rm hom}}(\w)(u).
\end{equation*}
Now for a general $u\in BV(D,\S)$ such that $|\{u=s_i\}|=V_i|D|$, the statement follows by density. This procedure is classical (see \cite{AmBrI}) and therefore we omit the details.
\end{proof}

\section{A model for random deposition}\label{sec-example}
The general homogenization result proved in Section \ref{homogenization} describes only the qualitative phenomenon that interfaces may form on the flat subspace. In this final section we investigate the asymptotic behavior of the limit energy as a function of the average thickness. To simplify matter, we consider a $3$d to $2$d dimension reduction problem in which magnetic particles are deposited with vertical order on a two-dimensional flat substrate and interact via finite-range ferromagnetic interactions of Ising-type, which means in particular that $\S=\{\pm 1\}$.
We obtain information on the dependence of the limit energy on the average thickness when the latter is very small or very large. \\

\hspace*{0,5cm}
In order to model the substrate where the particles are deposited, we take a two-dimensional deterministic lattice, which we choose for simplicity to be $\mathcal{L}^0=\mathbb{Z}^2\times\{0\}$. We then consider an independent random field $\{X_i^p\}_{i\in\mathbb{Z}^3}$, where the $X_i^p$ are Bernoulli random variables with $\mathbb{P}(X_i^p=1)=p\in(0,1)$ and, for fixed $M\in\mathbb{N}$, we define a random point set as follows:
\begin{equation}\label{randomdeposition}
\mathcal{L}_p^M(\w):=\left\{(i_1,i_2,i_3)\in\mathbb{Z}^3:\;0\leq i_3\leq\sum_{k=1}^MX_{(i_1,i_2,k)}^p(\w)\right\},
\end{equation}
\begin{figure}[h!]
\centerline{\includegraphics [width=3in]{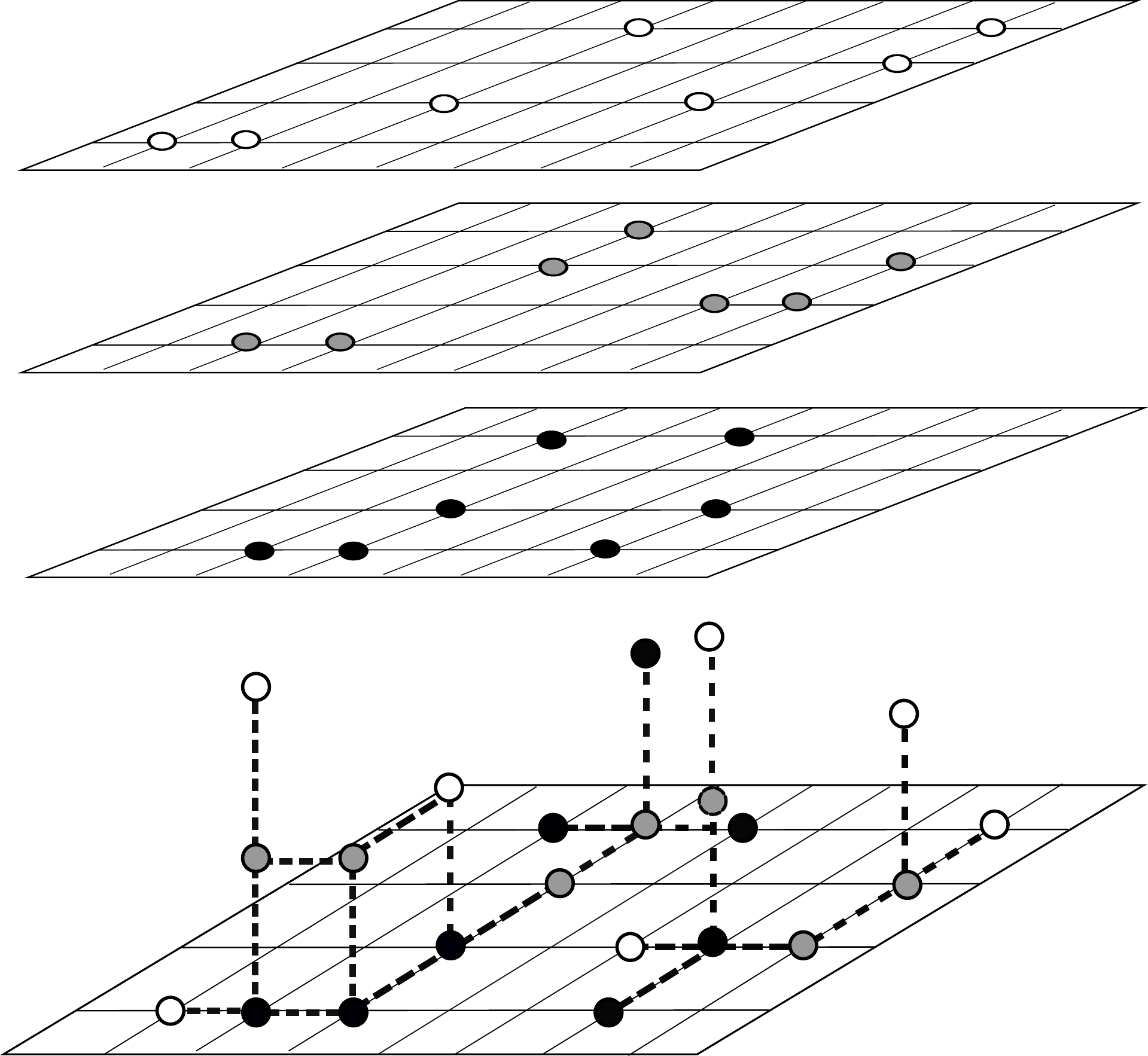}}
\caption{Three successive deposition steps (black, grey and white) in the construction of $\mathcal{L}_p^M(\w)$. The dashed bonds connect nearest neighbouring particles.} \label{deposition}
\end{figure}
which means that we successively deposit particles $M$ times independently on the flat lattice $\mathcal{L}^0$ and stack them over each other (see Figure \ref{deposition}). Note that the point set constructed in (\ref{randomdeposition}) is stationary with respect to integer translations in $\mathbb{Z}^2$ and ergodic by the independence assumption. Given $u:\e\mathcal{L}_p^M(\w)\to \{\pm 1\}$, we consider an energy of the form
\begin{equation}\label{depositionenergy}
E^p_{\e,M}(\w)(u,A)=\sum_{\substack{x,y\in\mathcal{L}_p^M(\w)\\ P_2(x),P_2(y)\in \frac{A}{\e}}}\e c(x-y)|u(\e x)-u(\e y)|,
\end{equation}
where the interaction $c:\R^3\to [0,+\infty)$ fulfills
\begin{itemize}
	\item[(i)] $c(z)\leq C$ for all $z\in\R^3$,
	\item[(ii)] $c(z)=0$ if $|z|\geq L$,
	\item[(iii)] $c(z)\geq c_0>0$ if $|z|=1$.
\end{itemize}
\begin{remark}\label{weakercoercivity}
Coefficients as above satisfy Hypothesis 2, but in general are not coercive as required in Hypothesis 1. However the results obtained in the first part of this paper still hold true. This is due to the vertical order of the deposition model which makes the proof of coercivity much simpler. However note that for instance the constant in Lemma \ref{longtoshort} now depends strongly on $M$. 	
\end{remark}
Due to Remark \ref{weakercoercivity} we can apply Theorem \ref{mainthm2} and thus we know that there exists the effective (deterministic) surface tension 
\begin{equation*}
\phi^p_{{\rm hom}}(M;\nu):=\lim_{t\to +\infty}\frac{1}{t}\inf\{E^p_{1,M}(\w)(v,Q_{\nu}(0,t)):\;v(x)=u_{0,\nu}(P_2(x))\text{ if }\dist(P_2(x),\partial Q_{\nu}(0,t))\leq 2L\},
\end{equation*}
where we used the alternative formula in Remark \ref{simplerformula} and Remark \ref{remarkfiniterange}. Note that due to symmetry reasons the surface tension does not depend on the traces (see also \cite{ACR}).
\\
\hspace*{0,5cm}
We are interested in the asymptotic behavior of $\phi^p_{\text{hom}}(M;\nu)$ when $M\to +\infty$. First, we define some auxiliary quantities. Given $p\in(0,1]$, $0\leq N<M$ and $u:\mathbb{Z}^3\to\{\pm 1\}$ we set
\begin{equation*}
E^{p}_{[N,M]}(\w)(u,O):=\sum_{\substack{x,y\in\mathcal{L}^M_p(\w) \\ x,y\in O\times [N,M]}}c(x-y)|u(x)-u(y)|
\end{equation*}
and omit the dependence on $\w$ of $E^{p}_{[N,M]}$ when $p=1$. In that case, given $\nu\in S^1$ we further introduce the corresponding surface tension
\begin{equation*}
\phi^{1,M}(\nu)=\lim_{t\to +\infty}\frac{1}{t}\inf\{E^{1}_{[0,M]}(u,Q_{\nu}(0,t)):\;v(x)=u_{0,\nu}(P_2(x)) \text{ if }\dist(P_2(x),\partial Q_{\nu}(0,t))\leq 2L\}.
\end{equation*}
Note that the existence of this limit follows by standard subadditivity arguments. The next lemma shows that the auxiliary surface tensions converge when $M\to +\infty$.
\begin{lemma}\label{auxlemma}
For any $\nu\in S^1$ there exists the limit
\begin{equation*}
\phi^1(\nu):=\lim_{M\to +\infty}\frac{1}{M}\phi^{1,M}(\nu).
\end{equation*}
\end{lemma}
\begin{proof}
We define a sequence $a_k=\phi^{1,k-1}(\nu)$. It is enough to show that $a_k$ is superadditive. To reduce notation, similar to (\ref{randommin}) we introduce  
\begin{equation*}
m_{[N,M]}(u_{0,\nu},Q_{\nu}(x,\rho)):=\inf\{E^{1}_{[N,M]}(u,Q_{\nu}(x,\rho)):\;u\in \mathcal{PC}_{1,u_{0,\nu}}^{2L}(Q_{\nu}(x,\rho))\}.
\end{equation*}
Note that by periodicity $m_{[N,M]}(u_{0,\nu},Q_{\nu}(x,\rho))=m_{[N+k,M+k]}(u_{0,\nu},Q_{\nu}(x,\rho))$ for every $k\in\mathbb{N}$. For fixed $t>>1$ one can take any admissible configuration for $m_{[0,M+M^{\prime}-1]}(u_{0,\nu},Q_{\nu}(0,t))$ and restrict it to the sets $Q_{\nu}(0,t)\times [0,M-1]$ and $Q_{\nu}(0,t)\times [M,M+M^{\prime}-1]$ to obtain the inequality
\begin{align*}
\frac{1}{t}m_{[0,M+M^{\prime}-1]}(u_{0,\nu},Q_{\nu}(0,t)) &\geq \frac{1}{t}m_{[0,M-1]}(u_{0,\nu},Q_{\nu}(0,t))+\frac{1}{t}m_{[M,M+M^{\prime}-1]}(u_{0,\nu},Q_{\nu}(0,t))
\\
&=\frac{1}{t}m_{[0,M-1]}(u_{0,\nu},Q_{\nu}(0,t))+\frac{1}{t}m_{[0,M^{\prime}-1]}(u_{0,\nu},Q_{\nu}(0,t)),
\end{align*}
where we neglected the interactions between the two cubes and used periodicity in the last equality. Letting $t\to +\infty$, we obtain superadditivity of the sequence $a_k$.
\end{proof}

The next result shows the asymptotic behaviour of the surface tension when the average number of layers $pM$ diverges.
\begin{proposition}\label{layerdependence}
Let $\phi^1$ be defined as in the previous lemma. For $\nu\in S^1$ it holds that 
\begin{equation*}
\lim_{M\to +\infty}\frac{\phi^p_{{\rm hom}}(M;\nu)}{pM}=\phi^1(\nu).
\end{equation*}
\end{proposition}

\begin{proof}
Throughout this proof we assume without loss of generality that $L\in\mathbb{N}$ and we set $\mathbb{Z}^2_M=\mathbb{Z}^2\times\{0,\dots,M\}$. Fix $\nu\in S^1$ (we will drop the dependence on $\nu$ for several quantities). We separately show two inequalities. For the moment we also fix $M$. Consider a sequence of minimizing configurations $u_N$ such that $\lim_N\frac{1}{N}E^{1}_{[0,M]}(u_N,Q_{\nu}(0,N))=\phi^{1,M}(\nu)$.
As we show now, we can assume that $u_N$ is a plane-like configuration as given by Theorem \ref{existenceplanelike}. Indeed, applying that theorem we find a plane-like ground state $u_{\nu}$ for the energy 
\begin{equation*}
E_M(u,Q_{\nu}(0,N)):=\sum_{\substack{x\in \mathbb{Z}^2_M\\ P_2(x)\in Q_{\nu}(0,N)}}\sum_{y\in\mathbb{Z}^2_M}c(x-y)|u(x)-u(y)|.
\end{equation*}
To reduce notation, we set
\begin{equation*}
S_{\nu}(N,\lambda)=\{x\in\R^2:\;x\in Q_{\nu}(0,N),\,\dist(x,\{\nu\}^\perp)\leq 4(\lambda+L)\}
\end{equation*}
so that the energy of $u_{\nu}$ is concentrated on $S_{\nu}(N,\lambda)\times[0,M]$ with $\lambda\leq CM$ (see Theorem \ref{existenceplanelike}). For any $N\in\mathbb{N}$ we define two configurations $\overline{u}_N,\tilde{u}_N:\mathbb{Z}^2_M\to\{\pm 1\}$ via 
\begin{equation*}
\begin{split}
&\overline{u}_N(x)=
\begin{cases}
u_{0,\nu}(P_2(x)) &\mbox{if $\dist(P_2(x),\mathbb{R}^2\backslash Q_{\nu}(0,N)\leq 2L$,}
\\
u_{\nu}(x) &\mbox{otherwise.}
\end{cases}
\\
&\tilde{u}_N(x)=
\begin{cases}
u_{\nu}(x) &\mbox{if $\dist(P_2(x),\mathbb{R}^2\backslash \left(Q_{\nu}(0,N)\right)\leq L$,}
\\
u_N(x) &\mbox{otherwise.}
\end{cases}
\end{split}
\end{equation*}
Then $\overline{u}_N$ is a plane-like configuration whose energy is again concentrated on $S_{\nu}(N,\lambda)\times[0,M]$. Using the boundary conditions and the finite range assumptions one can prove  that
\begin{align*}
E^1_{[0,M]}(u_N,Q_{\nu}(0,N))&\leq E^1_{[0,M]}(\overline{u}_N,Q_{\nu}(0,N))\leq E_M(u_{\nu},Q_{\nu}(0,N))+CM^2
\\
&\leq E_M(\tilde{u}_N,Q_{\nu}(0,N))+CM^2\leq E^1_{[0,M]}(u_N,Q_{\nu}(0,N))+2CM^2.
\end{align*} 
Dividing by $N$ and letting $N\to +\infty$ we see that asymptotically we can replace $u_N$ by the plane-like configuration $\overline{u}_N$. From now on we denote by $u_{N,M}$ a plane-like minimizer whose energy is concentrated on $S_{\nu}(N,\lambda)\times[0,M]$ with $\lambda\leq CM$ and such that $\phi^{1,M}(\nu)=\lim_N\frac{1}{N}E^1_{[0,M]}(u_{N,M},Q_{\nu}(0,N))$. 
We extend $u_{N,M}$ to $\mathbb{Z}^3$ setting $u_{N,M}(x)=u_{0,\nu}(P_2(x))$ for $x_3\notin\{0,\dots,M\}$. For $\delta>0$ small enough, we separate the contribution of the bottom and the first $M^p_{\delta}:=\lceil (p+\delta)M\rceil$ random layers and estimate the remaining interactions. This leads to
\begin{align*}
\frac{1}{M}\phi^p_{{\rm hom}}(M;\nu)\leq&\frac{1}{M}\liminf_{N\to +\infty}\frac{1}{N}\mathbb{E}[E_{1,M}^p(\w)(u_{N,M^p_{\delta}},Q_{\nu}(0,N))]
\\
\leq&\frac{1}{M}\liminf_{N\to +\infty}\frac{1}{N}\mathbb{E}[E^{1}_{[0,M^p_{\delta}]}(u_{N,M^p_{\delta}},Q_{\nu}(0,N))]
\\
&+\frac{C}{M}\limsup_{N\to +\infty}\frac{1}{N}\mathbb{E}\left[\#\{x\in\mathcal{L}^M_p(\w):\; x\in S_{\nu}(N,\lambda)\times (M^p_{\delta}-L,M]\}\right]
\\
\leq&\frac{1}{M}\phi^{1,M^p_{\delta}}(\nu)+C\mathbb{E}[\#\big\{x\in\mathcal{L}_{p}^{M}(\w):\;x\in\{(0,0)\}\times (M^p_{\delta}-L,M]\}
\\
\leq &\frac{1}{M}\phi^{1,M^p_{\delta}}(\nu)+C\sum_{k=M^p_{\delta}-L}^{M}(k-M^p_{\delta}+L){M\choose k} p^k(1-p)^{M-k},
\end{align*}
where in the last step we have used that the probability of having $k$ points in $\{(0,0)\}\times(M^p_{\delta}-L,M]$ is the same as having $k+M^p_{\delta}-L$ successes out of $M$ trials in a Bernoulli experiment. In order to bound the last sum, we use Hoeffding's inequality which yields, for $M$ large enough depending on $L,\delta$,
\begin{equation*}
\mathbb{P}\Big(\sum_{i=1}^MX_{(0,0,i)}^p\geq k+M^p_{\delta}-L\Big)\leq\mathbb{P}\Big(\sum_{i=1}^MX_{(0,0,i)}^p\geq k+\Big(p+\frac{\delta}{2}\Big)M\Big)\leq \exp\Big(-2M\Big(\frac{\delta}{2}+\frac{k}{M}\Big)^2\Big).
\end{equation*}
From this bound we infer the estimate
\begin{equation*}
\sum_{k=M^p_{\delta}-L}^{M}(k-M^p_{\delta}+L){M\choose k} p^k(1-p)^{M-k}\leq\sum_{k=1}^{M}k\exp\Big(-\frac{1}{2}M\delta^2\Big)\exp(-2\delta k).
\end{equation*}
Since the right hand side vanishes when $M\to +\infty$, by Lemma \ref{auxlemma} we deduce 
$\limsup_{M}\frac{1}{M}\phi^p_{{\rm hom}}(M;\nu)\leq (p+\delta)\,\phi^1(\nu)$. Since $\delta$ was arbitrary the first inequality is proven.
\\
\hspace*{0,5cm}
It remains to show the reverse inequality. Given any admissible function $v_N:\mathcal{L}^M_p(\w)\to\{\pm 1\}$ we can neglect the interactions coming from $Q_{\nu}(0,N)\times [M^p_{-\delta}+1,M]$ which yields the estimate
\begin{equation*}
E_{1,M}^p(\w)(v_N,Q_{\nu}(0,N))\geq E^p_{[0,M^p_{-\delta}]}(\w)(v_N,Q_{\nu}(0,N)).
\end{equation*}
Minimizing on both sides and dividing by $N$, we obtain in the limit that
\begin{equation}\label{almostdone}
\frac{1}{M}\phi^p_{{\rm hom}}(M;\nu)\geq \frac{1}{M}\phi^{p,M^p_{-\delta}}(\nu).
\end{equation}
Now the idea is to estimate the error when we replace $\phi^{p,M^p_{-\delta}}(\nu)$ by $\phi^{1,M^p_{-\delta}}(\nu)$. Let $u_N$ be a sequence of plane-like configurations as in the first part of the proof. We also consider an optimal sequence $u_N^{p,\delta}=u_N^{p,\delta}(\w)$ such that 
\begin{equation*}
\phi^{p,M^p_{-\delta}}(\nu)=\lim_{N\to +\infty}\frac{1}{N}\mathbb{E}[E^p_{[0,M^p_{-\delta}]}(\w)(u_N^{p,\delta},Q_{\nu}(0,N))].
\end{equation*}
Since the deterministic surface tension dominates the random one, we have
\begin{align*}
0&\leq \phi^{1,M^p_{-\delta}}(\nu)-\phi^{p,M^p_{-\delta}}(\nu)=\lim_{N}\frac{1}{N}\mathbb{E}\left[E_{[0,M^p_{-\delta}]}^1(u_N,S_{\nu}(N,\lambda))-E^p_{[0,M^p_{-\delta}]}(\w)(u_N^{p,\delta}(\w),Q_{\nu}(0,N))\right]
\\
&\leq \limsup_{N}\frac{1}{N}\mathbb{E}\left[E_{[0,M^p_{-\delta}]}^1(u^{p,\delta}_N,S_{\nu}(N,\lambda))-E^p_{[0,M^p_{-\delta}]}(\w)(u_N^{p,\delta}(\w),S_{\nu}(N,\lambda))\right]
\\
&\leq C \limsup_{N}\frac{1}{N}\mathbb{E}[\#\{x\in \left(S_{\nu}(\lambda,N)\times[1,M^p_{-\delta}]\right)\cap\mathbb{Z}^3:\;x\notin\mathcal{L}^M_p(\w)\}]
\\
&\leq CM\mathbb{E}\Big[\max\{M^p_{-\delta}-\sum_{i=1}^{M}X_{(0,0,i)}^p,0\}\Big]\leq CM\sum_{k=1}^{M^p_{-\delta}}k\,\mathbb{P}\Big(M^p_{-\delta}-\sum_{i=1}^{M}X_{(0,0,i)}^p\geq k\Big).
\end{align*}
Here we used that the number of missing interactions can be estimated by the number of missing lattice points since each point can only interact with finitely many others. Now we apply again Hoeffding's inequality which yields
\begin{equation*}
\mathbb{P}\Big(M^p_{-\delta}-\sum_{i=1}^{M}X_{(0,0,i)}^p\geq k\Big)\leq \mathbb{P}\Big(M\Big(p-\frac{\delta}{2}\Big)-k\geq\sum_{i=1}^{M}X_{(0,0,i)}^p\Big)\leq\exp\Big(-2M\Big(\frac{\delta}{2}+\frac{k}{M}\Big)^2\Big).
\end{equation*}
We conclude the bound
\begin{equation*}
\sum_{k=1}^{M^p_{-\delta}}k\mathbb{P}\Big(M^p_{-\delta}-\sum_{i=1}^{M}X_{(0,0,i)}^p\geq k\Big)\leq \sum_{k=1}^{M^p_{-\delta}}k\exp\Big(-\frac{1}{2}M\delta^2\Big)\exp(-2\delta k).
\end{equation*}
Again the right-hand side vanishes when $M\to +\infty$ and thus $\lim_M\frac{1}{M}|\phi^{1,M^p_{-\delta}}(\nu)-\phi^{p,M^p_{-\delta}}(\nu)|=0$, so that Lemma \ref{auxlemma} and (\ref{almostdone}) imply the estimate
\begin{equation*}
\liminf_{M\to +\infty}\frac{1}{M}\phi^p_{{\rm hom}}(M;\nu)\geq \lim_{M\to +\infty}\frac{1}{M}\phi^{1,M^p_{-\delta}}(\nu)=(p-\delta)\phi^1(\nu).
\end{equation*}
Again the desired estimate follows by the arbitrariness of $\delta>0$.
\end{proof}
\begin{remark}\label{perc}
If we had not included the initial layer $\mathcal{L}^0$, then Proposition \ref{layerdependence} would still hold. However then the surface tension may not be related to an appropriate $\Gamma$-limit since the compactness of sequences with bounded energy becomes a nontrivial issue. We refer to \cite{Dilute} for a possible approach to this problem in the case of nearest-neighbour interactions and bond-percolation models.
\end{remark}
\subsection*{A percolation-type phenomenon} 
We close this final section with a result on the growth of the averaged surface tension when the number of layers increases. We let $\mathcal{L}^M_p(\w)$ be defined as in (\ref{randomdeposition}) but restrict the analysis to nearest-neighbour interactions and make them non-periodic in the sense that their magnitude is very small when one of the particles belongs to the initial layer $\mathcal{L}^0$. More precisely, given $0<\eta<<1$ we consider functions of the form
\begin{equation*}
c_{\eta}(x-y)=
\begin{cases}
0 &\mbox{if $|x-y|>1$,}
\\
\eta &\mbox{if $|x-y|=1$ and $x_3\cdot y_3=0$,}
\\
c(x-y) &\mbox{otherwise,}
\end{cases}
\end{equation*} 
where $x\mapsto c(x)$ is strictly positive on the unit circle. Then the coefficients satisfy Hypothesis 2 and fulfill (a slightly weaker version of) Hypothesis 1. We define $E^{p,\eta}_{\e,M}$ as in (\ref{depositionenergy}) with $c$ replaced by $c_{\eta}$. According to Theorem \ref{mainthm2}, again there exists the limit
\begin{equation*}
\phi^{p,\eta}_{{\rm hom}}(M;\nu):=\lim_{t\to +\infty}\frac{1}{t}\inf\{E^{p,\eta}_{1,M}(\w)(v,Q_{\nu}(0,t)):\;v(x)=u_{0,\nu}(P_2(x))\text{ if }\dist(P_2(x),\partial Q_{\nu}(0,t))\leq 2\}.
\end{equation*}
In contrast to Proposition \ref{layerdependence}, for this model we also consider the case of small $M$. We will show that if $p<1-p_{{\rm site}}$, where $p_{{\rm site}}$ is the critical site percolation probability on $\mathbb{Z}^2$, then it holds that
\begin{equation*}
\phi_{{\rm hom}}^{p,\eta}(1;\nu)\leq C_p\,\eta,
\end{equation*}
where $C_p$ may blow up only for $p\to 1-p_{{\rm site}}$. Note that we do not claim here that $p_{{\rm site}}$ is the optimal bound. We can actually improve the result in the sense that for all $M\in\mathbb{N}$ such that $(1-p)^M>p_{{\rm site}}$, then we have
\begin{equation*}
\phi_{{\rm hom}}^{p,\eta}(M;\nu)\leq C_p\,\eta.
\end{equation*}
This shows that when the probability is very small but finite, the surface tension can be arbitrary small depending on the strength of the interaction in the substrate layer, on the other hand we will establish an analogue of Proposition \ref{layerdependence} asserting that if the average number of layers increases further, even the normalized surface tension approaches a value independent of $\eta$. This result can be interpreted as the equivalent to the percolation phenomenon described in the introduction of the paper for the model without initial layer ($\eta=0$). Before proving this result, we introduce the typical energy of one slice. Given $q\in(0,1]$ and $u:\mathbb{Z}^2\to\{\pm 1\}$ we set
\begin{equation*}
E^q_{sl}(\w)(u,A):=\sum_{\substack{x,y\in\mathcal{L}^1_q(\w)\backslash\mathcal{L}^0\\ P_2(x),P_2(y)\in A}}c(x-y)|u(x)-u(y)|
\end{equation*}
and omit the dependence on $\w$ if $q=1$. We further introduce the corresponding surface tension
\begin{equation*}
\phi_{sl}^q(\nu)=\lim_{t\to +\infty}\frac{1}{t}\inf\{E^q_{sl}(\w)(u,Q_{\nu}(0,t)):\;v(x)=u_{0,\nu}(x) \text{ if }\dist(x,\partial Q_{\nu}(0,t))\leq 2\}.
\end{equation*}
Note that the existence of this deterministic limit follows again from the subadditive ergodic theorem as in the proof of Theorem \ref{mainthm2}, since we used the coercivity only for passing from finite range to decaying interactions in Step 4. In general the random variables $\w\mapsto E^q_{sl}(\w)(u,A)$ are not defined on the same probability space but we will use them only for slices of the large set $\mathcal{L}^M_p(\w)$.
\begin{theorem}\label{almostpercolation}
Let $p\in (0,1)$ and $M\in\mathbb{N}$ be such that $(1-p)^M>p_{{\rm site}}$. There exists a constant $C_{p,M}$ locally bounded for $(1-p)^M\in(p_{{\rm site}},1)$ such that \begin{equation*}
\phi_{{\rm hom}}^{p,\eta}(M;\nu)\leq C_{p,M}\eta.
\end{equation*}
On the other hand, for any $p\in (0,1)$ it holds that
\begin{equation*}
\lim_{M\to +\infty}\frac{1}{M}\phi_{{\rm hom}}^{p,\eta}(M;\nu)=2p\Big(\big(c(e_1)+c(-e_1)\big)|\nu_1|+\big(c(e_2)+c(-e_2)\big)|\nu_2|\Big)
\end{equation*}
\end{theorem}
\begin{proof}
In order to prove the first statement, we start with the case $\nu=e_2$ and use results from percolation theory which show that the contribution from the random layers is negligible: For $q:=(1-p)^M>p_{{\rm site}}$, we consider the so-called Bernoulli site percolation on $\mathbb{Z}^2$, that is we assign independently a weight $X_i(\w)\in\{\pm 1\}$ to all the vertices $i\in\mathbb{Z}^2$ such that $\mathbb{P}(X_i=1)=q$. We say that $i_0,\dots,i_k$ is an occupied path if $|i_{n}-i_{n+1}|=1$ and $X_{i_n}(\w)=1$ for all $n=0,\dots,k$. Theorem 11.1 in \cite{Kestbook} yields that there exist universal constants $c_j,d_j$ such that
\begin{align*}
\mathbb{P}\Big(&\text{at least }c_1(q-p_{{\rm site}})^{d_1}n\text{ disjoint occupied paths from }\{0\}\times [0,n]\text{ to }\{m\}\times[0,n]
\\
&\text{ and contained in }[0,m]\times [0,n]\text{ exist}\Big)\geq 1-c_2(m+1)\exp(-c_3(q-p_{{\rm site}})^{d_2}n).
\end{align*}
Given $N\in \mathbb{N}$, we first combine this estimate with the Borel-Cantelli lemma  and, using stationarity, we obtain that for almost every $\w\in\Omega$ there exists $N_0=N_0(\w)$ such that for all $N\geq N_0$ we find at least $c_1(q-p_{{\rm site}})^{d_1}2\sqrt{N}$ disjoint occupied paths connecting the vertical boundary segments of the rectangle $R_N:=[-\lfloor\frac{N}{2}\rfloor+2,\lfloor\frac{N}{2}\rfloor-2]\times [-\lceil \sqrt{N} \rceil,\lceil \sqrt{N}\rceil ]$. As the paths are disjoint and are contained in $R_N$, at least one of them uses at most $\frac{2}{c_1}(q-p_{{\rm site}})^{-d_1}N$ vertices. Now we come back to the actual proof. By definition of the random lattice in (\ref{randomdeposition}), using the above considerations in the layer $\mathbb{Z}^2\times\{1\}$, for $N\geq N_0$ we can find a path connecting the vertical boundary segments of the rectangle $R_N\times\{1\}$, contained in $R_N\times\{1\}$, using at most $c_{p,M}N$ vertices and none of them belongs to $\mathcal{L}^M_p(\w)$. This path separates $R_N\times\{1\}$ into two subregions $R_N^{-}\times\{1\}$ and $R_{N}^{+}\times\{1\}$. For $N\geq N_0$ we define a (random) configuration $u_{N}:\mathcal{L}^M_p(\w)\to\{\pm 1\}$ as
\begin{equation*}
u_N(x)=
\begin{cases}
u_{0,e_1}(P_2(x)) &\mbox{if $P_2(x)\notin R_N$,}\\
+1 &\mbox{if $P_2(x)\in R_N^{+}$,}\\
-1 &\mbox{otherwise.}
\end{cases}
\end{equation*}
Up to possibly exchanging the roles of $R_N^{\pm}$ we can assume that $u_N\in\mathcal{PC}_{1,u_{0,e_2}}^2(\w,Q_{e_2}(0,N))$. Hence by definition of $\phi_{{\rm hom}}^{p,\eta}(e_2)$ and the fact that $u_{N}$ depends not on the $z$-direction, it holds that
\begin{align}\label{separation}
\phi_{{\rm hom}}^{p,\eta}(e_2)\leq& \liminf_{N\to +\infty}\frac{1}{N}E_{1,M}^{p,\eta}(\w)(u_N,Q_{e_2}(0,N))\leq\limsup_{N\to +\infty}\frac{1}{N}\sum_{\substack{x,y\in Q_{e_2}(0,N)\cap\mathbb{Z}^2\\ |x-y|=1}}\eta |u_N(x)-u_N(y)|\nonumber
\\
&+\limsup_{N\to +\infty}\frac{1}{N}\sum_{k=1}^M\sum_{\substack{x,y\in\mathcal{L}^M_p(\w)\\ x,y\in Q_{e_2}(0,N)\times \{k\}}}c(x-y)|u_N(x)-u_N(y)|.
\end{align}
We now estimate each of the two terms on the right-hand side. Concerning the second one, we observe that if $x,y\in (Q_{e_2}(0,N)\times\{k\})\cap\mathcal{L}^M_p(\w)$ are such that $|x-y|=1$ and $u_N(x)\neq u_N(y)$, then either $P_2(x),P_2(y)\in \pm\frac{N}{2}e_1+\big([-4,4]\times [-2\sqrt{N},2\sqrt{N}]\big)$ or, without loss of generality, $P_2(x)\in R_N^{-}$ and $P_2(y)\in R_N^+$. In the second case, we note that either $(P_2(x),1)$ or $(P_2(y),1)$ has to be a vertex of the path constructed above, hence either $x\notin\mathcal{L}^M_p(\w)$ or $y\notin \mathcal{L}^M_p(\w)$. We then rule out the existence of such interactions and we may bound the second term via
\begin{equation}\label{sqrtNbound}
\limsup_{N\to +\infty}\frac{1}{N}\sum_{k=1}^M\sum_{\substack{x,y\in\mathcal{L}^M_p(\w)\\ x,y\in Q_{e_2}(0,N)\times \{k\}}}c(x-y)|u_N(x)-u_N(y)|\leq \limsup_{N\to +\infty}\frac{CM}{\sqrt{N}}=0.
\end{equation} 
Applying the same arguments for the first term, we may use the fact that the separating path uses at most $c_{p,M}N$ vertices and we deduce that
\begin{equation*}
\limsup_{N\to +\infty}\frac{1}{N}\sum_{\substack{x,y\in Q_{e_2}(0,N)\cap\mathbb{Z}^2\\ |x-y|=1}}\eta |u_N(x)-u_N(y)|\leq 4c_{p,M}\eta.
\end{equation*} 
From this estimate, the first claim in the case $\nu=e_2$ follows by (\ref{separation}) and (\ref{sqrtNbound}). The above argument can be adapted to the cases $\nu=-e_2$ and $\nu=\pm e_1$. By $L^1$-lower semicontinuity, the one-homogeneous extension of $\phi_{{\rm hom}}^{p,\eta}$ must be convex (see \cite{AmBrII}). For general $\nu\in S^1$ the claim then follows upon multiplying the constant by a factor $\sqrt{2}$.
\\
\hspace*{0,5cm}
In order to prove the second claim, we need to show two inequalities. Given a sequence of admissible configurations $u_N$ such that $\lim_N\frac{1}{N}E_{sl}^1(u_N,Q_{\nu}(0,N))=\phi^1_{{\rm sl}}(\nu)$, we define an admissible configuration $\overline{u}_N:\mathcal{L}^M_p(\w)\to\{\pm 1\}$ via
\begin{equation*}
\overline{u}_N(x)=u_N(P_2(x)).
\end{equation*}
Arguing as in the proof of Proposition \ref{layerdependence}, we may assume that $u_N$ is a plane-like configuration and its energy is concentrated in a stripe 
\begin{equation*}
S_{\nu}(N,\lambda)=\{x\in\R^2:\;x\in Q_{\nu}(0,N),\,\dist(x,\{\nu\}^\perp)\leq 4(\lambda+1)\},
\end{equation*}
where now $\lambda$ is independent of $N,M$. By definition and the fact that $\overline{u}_N$ gives no interaction in the $z$-direction, we obtain that for any $\delta>0$ small enough
\begin{align*}
\frac{\phi^{p,\eta}_{{\rm hom}}(M;\nu)}{M}&\leq \frac{1}{M}\liminf_{N\to +\infty}\frac{1}{N}\mathbb{E}[E_{1,M}^{p,\eta}(\w)(\overline{u}_N,Q_{\nu}(0,N))]
\\
&\leq\left(\liminf_{N\to +\infty}\frac{1}{M}\sum_{k=1}^{M}\frac{1}{N}\mathbb{E}[E^{p_k}_{sl}(\w)(u_N,Q_{\nu}(0,N))]\right)+\frac{C}{M}\limsup_{N\to +\infty}\frac{1}{N}\#\{z\in\mathbb{Z}^2\cap S_{\nu}(N,\lambda)\}
\\
&\leq\liminf_{N\to +\infty}\frac{1}{N}\left((p+\delta)E_{sl}^1(u_N,Q_{\nu}(0,N))+\frac{1}{M}\sum_{k>\lfloor (p+\delta)M\rfloor}^M\mathbb{E}[E_{sl}^{p_k}(\w)(u_N,Q_{\nu}(0,N))]\right)+\frac{C\lambda}{M}
\\
&=(p+\delta)\phi^1_{{\rm sl}}(\nu)+\sup_{k>\lfloor (p+\delta)M\rfloor}\liminf_{N\to +\infty}\frac{1}{N}\mathbb{E}[E_{{\rm sl}}^{p_k}(\w)(u_N,Q_{\nu}(0,N))]+\frac{C\lambda}{M},
\end{align*}
where $p_k=\sum_{l=k}^{M}{M\choose l} p^l(1-p)^{M-l}$ is the probability of having at least $k$ successes out of $M$ trials in a Bernoulli experiment. Note that here the new random variables are indeed defined on the same probability space and are coupled to the variables generating the stochastic lattice $\mathcal{L}^M_p(\w)$. As $\lambda$ is independent of $M$, the third term vanishes when $M\to +\infty$, so that we are left to show that also the second one converges to zero. In order to estimate the second term we use the fact that $u_N$ is a plane-like configuration, so that 
\begin{equation*}
\frac{1}{N}\mathbb{E}[E_{{\rm sl}}^{p_k}(\w)(u_N,Q_{\nu}(0,N))]=\frac{1}{N}\mathbb{E}[E_{{\rm sl}}^{p_k}(\w)(u_N,S_{\nu}(N,\lambda))]\leq p_kC\lambda.
\end{equation*} 
For any $k>\lfloor (p+\delta)M\rfloor$, by the law of large numbers it holds that $p_k\to 0$ when $M\to +\infty$. Hence we deduce $\limsup_{M}\frac{1}{M}\phi^{p,\eta}_{{\rm hom}}(M;\nu)\leq (p+\delta)\,\phi^1_{{\rm sl}}(\nu)$. As $\delta>0$ was arbitrary, we finally obtain \begin{equation*}
\limsup_{M}\frac{1}{M}\phi^{p,\eta}_{{\rm hom}}(M;\nu)\leq p\,\phi^1_{{\rm sl}}(\nu).
\end{equation*}
\hspace*{0,5cm}
We next show the reverse inequality. Given any admissible function $\overline{u}_N:\mathcal{L}^M_p(\w)\to\{\pm 1\}$ we can neglect the interactions in the $z$-direction and the lowest layer $\mathcal{L}^0$ and obtain the estimate
\begin{equation*}
E_{1,M}^{p,\eta}(\w)(\overline{u}_N,Q_{\nu}(0,N))\geq\sum_{k=1}^ME^{p_k}_{{\rm sl}}(\w)(\overline{u}_N(\cdot,k),Q_{\nu}(0,N)) \geq \sum_{k=1}^{\lceil (p-\delta)M\rceil}E_{{\rm sl}}^{p_k}(\w)(\overline{u}_N(\cdot,k),Q_{\nu}(0,N)).
\end{equation*}
Since $\overline{u}_N(\cdot,k)$ fulfills the correct boundary condition in every layer, we deduce that
\begin{equation*}
\frac{1}{M}\phi^{p,\eta}_{{\rm hom}}(M;\nu)\geq (p-\delta)\inf_{k\leq \lceil (p-\delta)M\rceil}\phi^{p_k}_{{\rm sl}}(\nu).
\end{equation*}
Again by the law of large numbers for an independent Bernoulli experiment it remains to show that the function $q\mapsto \phi_{sl}^q(\nu)$ is continuous in $q=1$, that means we can pass from a random to a deterministic lattice. This will be the last step.
\\
\hspace*{0,5cm}
In order to prove continuity let $u_N$ be a plane-like sequence of configurations as in the first part of the proof and consider an optimal sequence $u_N^q(\w)$ such that 
\begin{equation*}
\phi^q_{{\rm sl}}(\nu)=\lim_{N\to +\infty}\frac{1}{N}\mathbb{E}[E^q_{{\rm sl}}(\w)(u_N^q(\w),Q_{\nu}(0,N))].
\end{equation*}
Similar to the proof of Proposition \ref{layerdependence} we obtain
\begin{align*}
0&\leq \phi^1_{{\rm sl}}(\nu)-\phi^q_{{\rm sl}}(\nu)=\lim_{N}\frac{1}{N}\mathbb{E}[E_{{\rm sl}}^1(u_N,S_{\nu}(\lambda,N))-E^q_{{\rm sl}}(\w)(u_N^q(\w),Q_{\nu}(0,N))]
\\
&\leq\limsup_{N}\frac{1}{N}\mathbb{E}[E_{{\rm sl}}^1(u^q_N(\w),S_{\nu}(\lambda,N))-E^q_{sl}(\w)(v_N^q(\w),S_{\nu}(\lambda,N))]
\\
&\leq C \lim_{N}\frac{1}{N}\mathbb{E}[\#\{z\in (S_{\nu}(\lambda,N)\cap\mathbb{Z}^2)\times\{1\}:\;z\notin\mathcal{L}^1_q(\cdot)\}]
=C(1-q)\lambda.
\end{align*}
The estimate above clearly implies convergence of the surface tensions when $q\to 1$ which shows that $\limsup_{M}\frac{1}{M}\phi^{p,\eta}_{{\rm hom}}(M;\nu)\geq p\,\phi^1_{{\rm sl}}(\nu)$.
\\
\hspace*{0,5cm}
It remains to identify $\phi^1_{{\rm sl}}(\nu)$. We just sketch the argument. Any admissible configuration asymptotically has an interface containing at least $|\nu_1|$ interactions along the two directions $\pm e_1$ and $|\nu_2|$ interactions along the directions $\pm e_2$. Since any pair of interacting points is counted twice with reversing direction and $|u(x)-u(y)|\in\{0,2\}$ we find that $\phi^1_{{\rm sl}}(\nu)\geq 2(c(e_1)+c(-e_1))|\nu_1|+2(c(e_2)+c(-e_2))|\nu_2|$. On the other hand a suitable discretization of a plane attains this value, hence
\begin{equation*}
\phi^1_{{\rm sl}}(\nu)=2(c(e_1)+c(-e_1))|\nu_1|+2(c(e_2)+c(-e_2))|\nu_2|,
\end{equation*}
and the proof is finished.
\end{proof}

\appendix

\section{Plane-like minimizers for one-periodic dimension reduction problems}
In this first part of the appendix we prove that the results about plane-like minimizers for periodic interactions in \cite{CadlF} can be extended to dimension-reduction problems. We restrict the analysis to one-periodic interactions, which is the case when the coefficients depend only on the difference as in Hypothesis 2. Moreover, we focus on the physical case of reduction $3$-d to $2$-d. To fix notation, for any set $\Gamma\subset\mathbb{Z}^2$, we write $\Gamma_M=\Gamma\times (\mathbb{Z}\cap [0,M])$. In contrast to the main part of this paper, here we consider an interaction energy that takes into account also interactions outside the domain. To be more precise, given $u:\mathbb{Z}^2_M\to\{\pm 1\}$ we investigate finite-range energies of the form 
\begin{equation*}
E_M(u,\Gamma)=\sum_{x\in\Gamma_M}\sum_{y\in\mathbb{Z}^2_M}c(x-y)|u(x)-u(y)|,
\end{equation*} 
where the coefficients fulfill the following assumptions:
\begin{itemize}
	\item[(i)] $0\leq c(z)\leq C$ for all $z\in\R^3$ and $\min_ic(\pm e_i)\geq c_0>0$,
	\item[(ii)] there exists $L>0$ such that $c(z)=0$ for all $|z|\geq L$.
\end{itemize}
Before stating and proving the main theorem we need some definitions. 
\begin{definition}
We say that $u:\mathbb{Z}^2_M\to\{\pm 1\}$ is a ground state for the energy $E_M$ whenever $E_M(u,\Gamma)\leq E_M(u,\Gamma)$ for all finite sets $\Gamma\subset\mathbb{Z}^2$ and all $v:\mathbb{Z}^2_M\to\{\pm1 \}$ such that $u=v$ on $\{z\in \mathbb{Z}^2_M:\;\exists z^{\prime}\in(\mathbb{Z}^2\backslash\Gamma)_M\text{ with }|z-z^{\prime}|\leq L\}$.	
\end{definition}
\begin{remark}\label{monotonicity}
When $u$ and $\Gamma$ are such that $E_M(u,\Gamma)\leq E_{M}(v,\Gamma)$ for all $v$ such that $u=v$ on $\{z\in \mathbb{Z}^2_M:\;\exists z^{\prime}\in(\mathbb{Z}^2\backslash\Gamma)_M\text{ with }|z-z^{\prime}|\leq L\}$, then the same conclusion holds for every subset $\Gamma^{\prime}\subset\Gamma$. Indeed, take any $v$ such that $u=v$ on $\{z\in \mathbb{Z}^2_M:\;\exists z^{\prime}\in(\mathbb{Z}^2\backslash\Gamma^{\prime})_M\text{ with }|z-z^{\prime}|\leq L\}$. Then for any two points $x,y$ with $x\in (\Gamma\backslash\Gamma^{\prime})_M$ and $y\in\mathbb{Z}^2_M$ with $|x-y|\leq L$, it holds that $u(x)=v(x)$ and $u(y)=v(y)$. Hence it follows that
\begin{equation*}
E_M(u,\Gamma^{\prime})-E_M(v,\Gamma^{\prime})=E_M(u,\Gamma)-E_M(v,\Gamma)\leq 0.
\end{equation*}
\end{remark}
Using the same notation as for the stochastic group action, for $k\in\mathbb{Z}^2$ we denote by $\tau_k$ the shift operator acting on sets $\Gamma$ and configurations $u:\mathbb{Z}^2_M\to\{\pm 1\}$ via
\begin{equation*}
\tau_k\Gamma=\Gamma+k,\quad\quad \tau_ku(x)=u(x-(k,0)).
\end{equation*}
Then the following formula holds true: 
\begin{equation}\label{shiftenergy}
E_M(\tau_ku,\tau_k\Gamma)=E_M(u,\Gamma).
\end{equation}
The remaining part of this appendix will be devoted to the proof of the next theorem.
\begin{theorem}\label{existenceplanelike}
There exists $\lambda>0$ such that for all $\nu\in S^1$ there exists a ground state $u_{\nu}$ of $E_M$ such that $u(x)\neq u(y)$ implies $\dist(x,\{\nu\}^\perp)\leq \lambda$. Such a ground state is called plane-like. Moreover we can choose $\lambda\leq CM$ for some constant $C$ independent of $\nu,M$.	
\end{theorem}
The proof of this theorem is very similar to \cite{CadlF,CoDiVa}. We first construct a particular minimizer among periodic configurations that enjoys several geometric properties. To this end, we need further notation (see \cite{CadlF} for more details). Fix a rational direction $\nu\in S^1\cap\mathbb{Q}^2$; we define the $\mathbb{Z}$-module
$\mathbb{Z}_{\nu}=\{z\in\mathbb{Z}^2:\;\langle z,\nu\rangle=0\}$ and, given $m\in\mathbb{N}$, we let $\mathcal{F}_{m,\nu}$ be any fundamental domain of the quotient $\mathbb{Z}^2_{/ m\mathbb{Z}_{\nu}}$, that is for every $z\in\mathbb{Z}^2$ there exist unique $z_{1}\in m\mathbb{Z}_{\nu}$ and $z_{2}\in\mathcal{F}_{m,\nu}$ such that $z=z_{1}+z_{2}$. Given real numbers $\theta$ and $\lambda$, with $\theta<\lambda$, we further introduce
\begin{equation*}
\mathcal{F}_{m,\nu}^{\theta,\lambda}=\{z\in\mathcal{F}_{m,\nu}:\;\langle\nu, z\rangle\in [\theta,\lambda]\}.
\end{equation*}
Now we define an admissible class of periodic configurations: A function $u:\mathbb{Z}^2_M\to\{\pm 1\}$ is called $(m,\nu)$-{\it periodic} if $u(x)=u(x+m(z,0))$ for every $x\in\mathbb{Z}^2_M$ and every $z\in\mathbb{Z}_{\nu}$. We set
\begin{equation*}
\mathcal{A}_{m,\nu}^{\theta,\lambda}=\{u\text{ is }(m,\nu)\text{-periodic},\;u=+1\text{ if }\langle P_2(z),\nu\rangle <\theta,\,u(z)=-1\text{ if }\langle P_2(z),\nu\rangle >\lambda\}.
\end{equation*}
We start with a very elementary lemma, that shows how that for periodic functions any translation gives the same energy.
\begin{lemma}\label{elementary}
Let $u$ be $(m,\nu)$-periodic and $k\in\mathbb{Z}^2$. Then it holds that
\begin{equation*}
E_M(\tau_ku,\mathcal{F}_{m,\nu})=E_M(u,\mathcal{F}_{m,\nu}).
\end{equation*}	
\end{lemma}
\begin{proof}
Given $x\in(\tau_{-k}\mathcal{F}_{m,\nu})_M$, we find $z_1(x)\in m\mathbb{Z}_{\nu}$ and $z_2(x)\in\mathcal{F}_{m,\nu}$ such that $P_2(x)=z_1(x)+z_2(x)$. By $(m,\nu)$-periodicity, for any $y\in\mathbb{Z}^2_M$ it holds that
\begin{align*}
|u(x)-u(y)|&=|u(x-(z_1(x),0))-u(y-(z_1(x),0))|,
\\
c(x-y)&=c(x-(z_1(x),0)-y+(z_1(x),0)).
\end{align*}
Now assume that there exist another $x^{\prime}\in(\tau_{-k}\mathcal{F}_{m,\nu})_M\backslash\{x\}$ with $\langle x-x^{\prime},e_3\rangle=0$ and $z_2(x)=z_2(x^{\prime})$. Then $\tau_kP_2(x)-\tau_kP_2(x^{\prime})=z_1(x)-z_1(x^{\prime})\in m\mathbb{Z}_{\nu}\backslash\{(0,0)\}$. As $\tau_kP_2(x),\tau_kP_2(x^{\prime})\in\mathcal{F}_{m,\nu}$ this contradicts the fact that $\mathcal{F}_{m,\nu}$ is a fundamental domain. Using (\ref{shiftenergy}) we conclude by comparison that
\begin{equation*}
E_M(\tau_ku,\mathcal{F}_{m,\nu})=E_M(u,\tau_{-k}\mathcal{F}_{m,\nu})\leq E_M(u,\mathcal{F}_{m,\nu}).
\end{equation*}
Applying the above inequality to $\tau_{-k}$ and $\tilde{u}:=\tau_{k}u$, which is also $(m,\nu)$-periodic, we obtain the claim.
\end{proof}
\hspace*{0,5cm}
We define the class of minimizers for the energy $E_M(\cdot,\mathcal{F}_{m,\nu})$ on $\mathcal{A}_{m,\nu}^{\theta,\lambda}$ via
\begin{equation*}
\mathcal{M}_{m,\nu}^{\theta,\lambda}=\{u\in\mathcal{A}_{m,\nu}^{\theta,\lambda}:\;E_M(u,\mathcal{F}_{m,\nu})\leq E_M(v,\mathcal{F}_{m,\nu})\text{ for all }  v\in\mathcal{A}_{m,\nu}^{\theta,\lambda}\}.
\end{equation*}
As the set $\mathcal{A}_{m,\nu}^{\theta,\lambda}$ is finite, the class of minimizers is non-empty. Next we define the so-called infimal minimizer which has several useful properties.
\begin{equation*}
u_{m,\nu}^{\theta,\lambda}=\min\{u\in\mathcal{M}_{m,\nu}^{\theta,\lambda}\}\in \mathcal{A}_{m,\nu}^{\theta,\lambda}.
\end{equation*}
We next show that the infimal minimizer also belongs to the class of minimizers. This follows from the following elementary observation (see Lemma 2.1 and also Lemma 2.3 in \cite{CoDiVa}).
\begin{lemma}\label{maxmin}
Given any $u:\mathbb{Z}^2_M\to\{\pm 1\}$ and $\Gamma\in\mathbb{Z}^2$ finite, it holds that
\begin{equation*}
E_M(\min\{u,v\},\Gamma)+E_M(\max\{u,v\},\Gamma)\leq E_M(u,\Gamma)+E_M(v,\Gamma).
\end{equation*}
\end{lemma}
Iterating the above lemma finitely many times we find that $u_{m,\nu}^{\theta,\lambda}\in\mathcal{M}_{m,\nu}^{\theta,\lambda}$. 
\\
\hspace*{0,5cm}
We now turn to the first property of the infimal minimizer. This is the so-called absence of symmetry breaking, which says that the infimal minimizer does not depend on the length $m$ of the period.
\begin{lemma}\label{nosymmetrybreaking}
For any $m\in\mathbb{N}$ it holds that $u_{m,\nu}^{\theta,\lambda}=u_{1,\nu}^{\theta,\lambda}$.	
\end{lemma}
\begin{proof}
We define an auxiliary configuration via $u=\min\{\tau_ku_{m,\nu}^{\theta,\lambda}:\;k\in\mathbb{Z}_{\nu}\}$. By elementary arguments it follows that $u\in\mathcal{A}_{1,\nu}^{\theta,\lambda}$, while Lemma \ref{elementary} implies that $\tau_ku_{m,\nu}^{\theta,\lambda}\in\mathcal{M}_{m,\nu}^{\theta,\lambda}$ and by iterating Lemma \ref{maxmin} we obtain that $u\in\mathcal{M}_{m,\nu}^{\theta,\lambda}$. Since $u\leq u_{m,\nu}^{\theta,\lambda}$, by definition of the infimal minimizer and  we obtain that $u=u_{m,\nu}^{\theta,\lambda}$. Moreover, as $u$ and $u_{1,\nu}^{\theta,\lambda}$ are both $(1,\nu)$-periodic it follows that
\begin{equation}\label{periodicenergy}
E_M(u,\mathcal{F}_{1,\nu})=\frac{1}{m}E_M(u,\mathcal{F}_{m,\nu})\leq \frac{1}{m}E_M(u_{1,\nu}^{\theta,\lambda},\mathcal{F}_{m,\nu})=E_M(u^{\theta,\lambda}_{1,\nu},\mathcal{F}_{1,\nu}).
\end{equation}
In particular we deduce that $u\in\mathcal{M}_{1,\nu}^{\theta,\lambda}$ and thus $u\geq u_{1,\nu}^{\theta,\lambda}$. On the other hand, (\ref{periodicenergy}) must be an equality, so that $u_{1,\nu}^{\theta,\lambda}\in\mathcal{M}_{m,\nu}^{\theta,\lambda}$ and therefore $u_{1,\nu}^{\theta,\lambda}\geq u$. This proves the claim.
\end{proof}
We next establish the so-called Birkhoff property of the infimal minimizer which will be the main ingredient for the proof of Theorem \ref{existenceplanelike}.
\begin{lemma}\label{birkhoff}
Let $k\in\mathbb{Z}^2$. Then 
$\tau_ku_{1,\nu}^{\theta,\lambda}
\leq u_{1,\nu}^{\theta,\lambda }$ if $\langle k,\nu\rangle\leq 0$ and $\tau_ku_{1,\nu}^{\theta,\lambda}
\geq u_{1,\nu}^{\theta,\lambda}$ if $\langle k,\nu\rangle\geq 0$.
\end{lemma}
\begin{proof}
We start with the case $\langle k,\nu\rangle\leq 0$ and define the two configurations $m=\min\{u_{1,\nu}^{\theta,\lambda},\tau_ku_{1,\nu}^{\theta,\lambda}\}$ and $M=\max\{u_{1,\nu}^{\theta,\lambda},\tau_ku_{1,\nu}^{\theta,\lambda}\}$. By elementary considerations one can prove that $m\in\mathcal{A}_{1,\nu}^{\theta+\langle k,\nu\rangle,\lambda+\langle k,\nu\rangle}$ and $M\in\mathcal{A}_{1,\nu}^{\theta,\lambda}$. Using Lemma \ref{maxmin} we obtain
\begin{equation*}
E_M(m,\mathcal{F}_{1,\nu})+E_M(u_{1,\nu}^{\theta,\lambda},\mathcal{F}_{1,\nu})\leq E_M(m,\mathcal{F}_{1,\nu})+E_M(M,\mathcal{F}_{1,\nu})\leq E_M(\tau_ku_{1,\nu}^{\theta,\lambda},\mathcal{F}_{1,\nu})+E_M(u_{1,\nu}^{\theta,\lambda},\mathcal{F}_{1,\nu}),
\end{equation*}
which yields $E_M(m,\mathcal{F}_{1,\nu})\leq E_M(\tau_ku_{1,\nu}^{\theta,\lambda},\mathcal{F}_{1,\nu})$. We claim that $\tau_ku_{1,\nu}^{\theta,\lambda}=u_{1,\nu}^{\theta+\langle k,\nu\rangle,\lambda+\langle k,\nu\rangle}$. Indeed, as $\tau_ku_{1,\nu}^{\theta,\lambda}\in\mathcal{A}_{1,\nu}^{\theta+\langle k,\nu\rangle,\lambda+\langle k,\nu\rangle}$ this configuration is admissible and minimality follows by Lemma \ref{elementary}. Now assume it wouldn't be the infimal minimizer, then also $u_{1,\nu}^{\theta,\lambda}$ is not the infimal minimizer as we could construct a smaller one by translation of the other infimal minimizer.
\\
\hspace*{0,5cm}
By definition of the infimal minimizer we infer that $m\geq \tau_ku_{1,\nu}^{\theta,\lambda}$, which proves the claim by definition of $m$. The case $\langle k,\nu\rangle\geq 0$ follows upon applying the translation $\tau_k$ to the inequality $\tau_{-k}u_{1,\nu}^{\theta,\lambda}\leq u_{1,\nu}^{\theta,\lambda}$ which holds by the first part of the proof.
\end{proof}
In the next lemma we deduce a powerful property of configurations fulfilling the Birkhoff property.
\begin{lemma}\label{rigidity}
Let $u:\mathbb{Z}^2_M\to\{\pm 1\}$ satisfy the {\em Birkhoff property} with respect to $\nu\in S^1\cap\mathbb{Q}^2$; that means that $\tau_ku
\leq u$ if $\langle k,\nu\rangle\leq 0$, and $\tau_ku\geq u$ if $\langle k,\nu\rangle\geq 0$.
Assume further that $u(x_0)=-1$ for some $x_0\in\mathbb{Z}^2_M$. Then 
$u(x)=-1$ for all $x\in\mathbb{Z}^2_M$ such that $\langle x-x_0,e_3\rangle=0$ and $\langle P_2(x-x_0),\nu\rangle\geq 0$.
\end{lemma}
\begin{proof}
Every such $x$ can be written as $x=x_0-(k,0)$ with $k\in\mathbb{Z}^2$ such that $\langle k,\nu\rangle \leq 0$. Hence Lemma \ref{birkhoff} implies that $u(x)=\tau_ku(x_0)\leq u(x_0)=-1$, so that $u(x)=-1$.
\end{proof}
We are now in a position to prove that the infimal minimizer becomes unconstrained when we take $\theta=0$ and $\lambda$ large enough. To reduce notation, from now on we set $u_{\nu}^\lambda:=u_{1,\nu}^{0,\lambda}$.
\begin{lemma}\label{unconstrainedI}
There exists $\lambda_0>0$ (depending on $M$ in such a way that $\lambda_0\leq CM$) such that for all $\lambda\geq \lambda_0$ it holds
$u_{\nu}^\lambda(x)=-1$ for all $x\in\mathbb{Z}^2_M$ such that $\langle P_2(x),\nu\rangle\geq \lambda-\sqrt{2}$.
\end{lemma}
\begin{proof}
By Lemma \ref{rigidity} it is enough to show that for large enough $\lambda$, in every layer $\mathbb{Z}^2\times \{l\}$ with $l\in\{0,\dots,M\}$ there exists some $x_l$ such that $\langle P_2(x_l),\nu\rangle\leq \lambda-\sqrt{2}$ and $u_{\nu}^\lambda(x_l)=-1$. We will show that this is always the case provided $\lambda$ is large enough.
\\
\hspace*{0,5cm}
Assume that there exists a layer $\mathbb{Z}^2\times\{l\}$ such that $u_{\nu}^\lambda(x)=1$ for all $x\in\mathbb{Z}^2\times\{l\}$ with $\langle P_2(x),\nu\rangle\leq \lambda-\sqrt{2}$. We argue that in this case there must exists a second layer $\mathbb{Z}^2\times\{l^{\prime}\}$ and a point $x_{l^{\prime}}\in\mathbb{Z}^2\times\{l^{\prime}\}$ with $\langle P_2(x_{l^{\prime}}),\nu\rangle\leq\sqrt{2}$ and $u_{\nu}^\lambda(x_{l^{\prime}})=-1$. Indeed, if this would be false, then the function $\tau_ku_{\nu}^\lambda$ with any $k\in\{0,\pm 1\}^2$ such that $\langle k,\nu\rangle<0$ fulfills $\tau_ku_{\nu}^\lambda\in\mathcal{A}_{1,\nu}^{0,\lambda}$. By Lemma \ref{birkhoff} we further know that $\tau_ku_{\nu}^\lambda\leq u_{\nu}^\lambda$. On the other hand, by Lemma \ref{elementary} we have that $\tau_ku_{\nu}^\lambda\in\mathcal{M}_{1,\nu}^{0,\lambda}$, hence by definition of the infimal minimizer we obtain $\tau_ku_{\nu}^\lambda=u_{\nu}^\lambda$.  This contradicts the boundary conditions by the choice of $k$. Now applying Lemma \ref{rigidity} in the second layer $\mathbb{Z}^2\times\{l^{\prime}\}$ we obtain that $u_{\nu}^\lambda(x)=-1$ for all $x\in\mathbb{Z}^2\times\{l^{\prime}\}$ such that $\langle P_2(x),\nu\rangle\geq\sqrt{2}$. As we will see now, for fixed $M$ this will cost too much energy.
\\
\hspace*{0,5cm}
Without loss of generality we assume that $l>l^{\prime}$, the other case can be treated almost the same way. For every $r\in\{1,\dots,M\}$ there exists $x\in\mathbb{Z}^2\times\{r\}$ such that $u_{\nu}^\lambda(x_r)=-1$. Let $x_r$ be one of such points that minimizes $\langle P_2(x),\nu\rangle$ among all such points. According to Lemma \ref{rigidity} we obtain $u_{\nu}^\lambda(x)=-1$ for all $x\in\mathbb{Z}^2\times\{r\}$ with $\langle P_2(x),\nu\rangle\geq \langle P_2(x_r),\nu\rangle=:p_r$. Note that
\begin{equation}\label{totallength}
\left|\sum_{r=l^{\prime}}^{l-1}(p_{r+1}-p_r)\right|\geq \lambda-2\sqrt{2}.
\end{equation}
On the other hand, just counting the interactions between neighbouring layers, we obtain by the coercivity of the interactions and (\ref{totallength}) that
\begin{equation*}
E_M(u_{\nu}^\lambda,\mathcal{F}_{1,\nu})\geq c\sum_{r=1}^M|p_r-p_{r-1}|\geq c(\lambda-2\sqrt{2}).
\end{equation*}
Testing a discretized plane as a possible minimizer, by the finite range assumption we know an a priori bound of the form $E_M(u_{\nu}^\lambda,\mathcal{F}_{1,\nu})\leq CM$. Hence our assumption can only hold as long as $\lambda\leq CM$ for some constant $C$ not depending on $\nu$ nor on $M$ and the claim follows upon setting $\lambda_0=2CM$.
\end{proof}
The next (and last) lemma  bounds the oscillation of the jump set of the infimal minimizer $u_{\nu}^{\lambda_0}$.
\begin{lemma}\label{unconstrainedII}
Let $\lambda_0$ be as in Lemma \ref{unconstrainedI}. Then $u_{\nu}^{\lambda_0}\in\mathcal{M}_{m,\nu}^{-n,\lambda_0+n}$ for any $n,m\in\mathbb{N}$.
\end{lemma}
\begin{proof}
We first claim that $u_{\nu}^{\lambda_0}=u_{\nu}^{\lambda_0+l}$ for any $l\in\mathbb{N}$. This will be done iteratively. First note that for any $\lambda\geq \lambda_0$ it holds that $u_{\nu}^{\lambda}\in\mathcal{A}_{1,\nu}^{0,\lambda+1}$ and by Lemma \ref{unconstrainedI} it also holds that $u_{\nu}^{\lambda+1}\in\mathcal{A}_{1,\nu}^{0,\lambda}$. Then
\begin{equation*}
E_M(u_{\nu}^{\lambda+1},\mathcal{F}_{1,\nu})=E_M(u_{\nu}^{\lambda},\mathcal{F}_{1,\nu})
\end{equation*} 
and both are infimal minimizers. Hence they must agree. This proves the first claim. 
\\
\hspace*{0,5cm}
Give an arbitrary configuration $v\in\mathcal{A}_{m,\nu}^{-n,\lambda_0+n}$ we choose a vector $k\in\mathbb{Z}^2$ such that $\langle k,\nu\rangle\geq n$ and $\langle k,\nu\rangle \in\mathbb{N}$. Then 
\begin{equation*}
\tau_kv\in \mathcal{A}_{m,\nu}^{-n+\langle k,\nu\rangle,\lambda_0+n+\langle k,\nu\rangle}\subset\mathcal{A}_{m,\nu}^{0,\lambda_0+n^{\prime}}
\end{equation*}
with $n^{\prime}\in\mathbb{N}$. Using the first claim and the Lemmata \ref{elementary} and \ref{nosymmetrybreaking} we obtain that $E_M(u_{\nu}^{\lambda_0},\mathcal{F}_{m,\nu})\leq E_M(\tau_kv,\mathcal{F}_{m,\nu})= E_M(v,\mathcal{F}_{m,\nu})$. As $u_{\nu}^{\lambda_0}\in\mathcal{A}_{m,\nu}^{-n,\lambda_0+n}$ we proved the claim.
\end{proof}
\begin{proof}[Proof of Theorem \ref{existenceplanelike}]
First assume that $\nu\in S^1\cap\mathbb{Q}^2$. We show that $u_{\nu}^{\lambda_0}$ is a ground state. To this end let $\Gamma\subset\mathbb{Z}^2$ be finite and let $v:\mathbb{Z}^2_M\to\{\pm 1\}$ be such that $v=u_{\nu}^{\lambda_0}$ on $\{z\in \mathbb{Z}^2_M:\;\exists z^{\prime}\in(\mathbb{Z}^2\backslash\Gamma)_M\text{ with }|z-z^{\prime}|\leq L\}$. Then we find $m\in\mathbb{N}$ such that, for a suitable fundamental domain, $\Gamma\subset \mathcal{F}_{m,\nu}$. By Lemma \ref{unconstrainedII} we have that $E_M(u_{\nu}^{\lambda_0},\mathcal{F}_{m,\nu})\leq E_M(v,\mathcal{F}_{m,\nu})$ and the claim then follows by Remark \ref{monotonicity}.
\\
\hspace*{0,5cm}
For general directions $\nu\in S^1$ we argue by approximation. Take a sequence $\nu_j\to\nu$ of rational directions and consider the sequence $u_j:=u_{\nu_j}^{\lambda_j}$ where $\lambda_j$ is uniformly bounded in $j$. By Tychonoff's theorem we can assume that $u_j\to u$ for some $u:\mathbb{Z}^2_M\to\{\pm 1\}$. It holds that $u$ is a plane-like configuration. By definition of the topology, given any finite set $\Gamma\subset\mathbb{Z}^2$ we find an index $j_0$ such that $u_j(x)=u(x)$ for all $x\in\Gamma_M$ and all $j\geq j_0$. Since we assume a finite range of interaction, the previous convergence property implies that $u$ is also a ground state. 
\end{proof}

\section{Density results for trace-constraints on partitions}
In this second appendix we show the density result needed in the proof of Theorem \ref{constrainedproblem}.
\begin{lemma}\label{strictinterior}
Let $A\subset\subset B$ be both bounded open sets with Lipschitz boundary. Given $v,w\in BV(B,\S)$ such that $\mathcal{H}^{k-1}(S_w\cap\partial A)=0$ we set $u=\mathds{1}_A v+(1-\mathds{1}_A)w$. Then there exists a sequence $A_{n}\subset\subset A$ of sets of finite perimeter such that $u_{n}:=\mathds{1}_{A_{n}}v+(1-\mathds{1}_{A_{n}})w$ converges to $u$ in $L^1(B)$ and additionally $\mathcal{H}^{k-1}(S_{u_n}\cap B)\to\mathcal{H}^{k-1}(S_u\cap B)$.
\end{lemma}

\begin{proof}
We define the mapping $T:\S\to\mathbb{R}^q$ defined by $T(s_i)=e_i$. As a special case of Proposition 4.1 in \cite{TS15}, applied to the bounded $BV$-function $\a:=T(w)-T(v)$, for every $\e>0$ we find an open set $A_{\e}$ of finite perimeter such that $A_{\e}\subset\subset A$, $|A\backslash A_{\e}|\leq \e$ and
\begin{equation}\label{stricttrace}
\int_{\partial A_{\e}}|\a_{| \partial A_{\e}}^+|\,\mathrm{d}\mathcal{H}^{k-1}\leq \int_{\partial A}|\a_{|\partial A}^+|\,\mathrm{d}\mathcal{H}^{k-1}+\e.
\end{equation}
With the same arguments as in in \cite{TS15}, the sets $A_{\e}$ can be constructed in a way that for all $\delta>0$ there exists $\e_0>0$ such that for all $\e<\e_0$
\begin{equation}\label{fillset}
\{x\in A:\;{\rm dist}(x,\partial A)>\delta\}\subset A_{\e}.
\end{equation}
We show that the sets $A_{\e}$ fulfill the required properties. As a first step we claim that $T(u_{\e})$ converges strictly to $T(u)$. We have that $T(u_{\e})$ converges to $T(u)$ in $L^1(B)$. By lower semicontinuity of the total variation it is enough to show that 
\begin{equation}\label{strictproperty}
\limsup_{\e\to 0}|DT(u_{\e})|(B)\leq |DT(u)|(B).
\end{equation}
By definition we have $|DT(u_{\e})|(B\backslash\overline{A})=|DT(u)|(B\backslash\overline{A})$, so that we can reduce the analysis to $\overline{A}$. By Theorem 3.84 in \cite{AFP} it holds that
\begin{equation*}
DT(u_{\e})=DT(v)\LL A_{\e}^{(1)}+DT(w)\LL A_{\e}^{(0)}+(T(v)_{|\partial A_{\e}}^+-T(w)_{|\partial A_{\e}}^-)\otimes\nu\,\mathcal{H}^{k-1}\LL\partial A_{\e},
\end{equation*}
where in general $A_{\e}^{(t)}$ is defined for $t\in[0,1]$ via
\begin{equation*}
A_{\e}^{(t)}=\left\{x\in\mathbb{R}^k:\;\lim_{\rho\to 0}\frac{|A_{\e}\cap B_{\rho}(x)|}{|B_{\rho}(x)|}=t\right\}.
\end{equation*}
Since $A_{\e}\subset\subset A$ and $A_{\e}$ is open we infer $A_{\e}^{(1)}\subset A$ and $A_{\e}^{(0)}\subset \mathbb{R}^k\backslash A_{\e}$, so that
\begin{align*}
|DT(u_{\e})|(\overline{A})\leq&|DT(v)|(A)+|DT(w)|(\overline{A}\backslash A_{\e})+\int_{\partial A_{\e}}|T(v)_{|\partial A_{\e}}^+-T(w)_{|\partial A_{\e}}^-|\,\mathrm{d}\mathcal{H}^{k-1}
\\
\leq &|DT(v)|(A)+|DT(w)|(\overline{A}\backslash A_{\e})+\int_{\partial A_{\e}}|T(w)_{|\partial A_{\e}}^+-T(w)_{|\partial A_{\e}}^-|\,\mathrm{d}\mathcal{H}^{k-1}
\\
&+\int_{\partial A_{\e}}|T(v)_{|\partial A_{\e}}^{+}-T(w)_{|\partial A_{\e}}^+|\,\mathrm{d}\mathcal{H}^{k-1}.
\end{align*}
By assumption on $w$ we have $|DT(w)|(\partial A)=0$, so that by (\ref{fillset}) the second and the third term vanish when $\e\to 0$. For the fourth one we use (\ref{stricttrace}) and infer
\begin{align*}
\limsup_{\e\to 0}|DT(u_{\e})|(\overline{A}) &\leq |DT(v)|(A)+\int_{\partial A}|T(v)_{|\partial A}^{+}-T(w)_{|\partial A}^+|\,\mathrm{d}\mathcal{H}^{k-1}
\\
&=|DT(v)|(A)+\int_{\partial A}|T(v)_{|\partial A}^{+}-T(w)_{|\partial A}^-|\,\mathrm{d}\mathcal{H}^{k-1}=|DT(u)|(\overline{A}),
\end{align*}
where we used that inner and outer trace of $T(w)$ agree for $\mathcal{H}^{k-1}$-almost every $x\in\partial A$. By the structure of the set $T(\S)$ strict convergence implies that 
\begin{equation*}
\mathcal{H}^{k-1}(S_{T(u_{\e})}\cap B)=\frac{1}{\sqrt{2}}|DT(u_{\e})|\to\frac{1}{\sqrt{2}}|DT(u)|=\mathcal{H}^{k-1}(S_{T(u)}\cap B).
\end{equation*}	
As for every $u\in BV(B,\S)$ it holds that $\mathcal{H}^{k-1}(S_u\cap B)=\mathcal{H}^{k-1}(S_{T(u)}\cap B)$ and also $L^1$-convergence is conserved, we conclude the proof.
\end{proof}

\section*{Acknowledgements}

The work of MC was supported by the DFG Collaborative Research Center TRR 109, ``Discretization in Geometry and Dynamics''.


\begin{thebibliography}{34}
	
	\bibitem{AkKr} 
	\newblock M.A. Akcoglu and U. Krengel:
	\newblock Ergodic theorems for superadditive processes.
	\newblock \emph{J. Reine Ang. Math.}, \textbf{323} (1981) , 53--67.
	
	
	\bibitem{ABC}
	R.~Alicandro, A.~Braides and M.~Cicalese:
	\newblock Phase and anti-phase boundaries in binary discrete systems: a variational viewpoint.
	\newblock {\em Netw. Heterog. Media}, {\bf 1} (2006), 85--107.
	
	\bibitem{ABC2}
	R.~Alicandro, A.~Braides and M.~Cicalese:
	\newblock Continuum limits of discrete thin films with superlinear growth densities.
	\newblock {\em Calc. Var. Partial Diff. Eq.}, {\bf 33} (2008), 267--297.
	
		
	\bibitem{AC} 
	\newblock R. Alicandro and M. Cicalese:
	\newblock A general integral representation result
for continuum limits of discrete energies with superlinear growth
	\newblock \emph{SIAM J. Math. Anal.}, \textbf{36} (2004), 1--37.
	
	
	\bibitem{ACG2} 
	\newblock R. Alicandro, M. Cicalese and A. Gloria:
	\newblock Integral representation results for energies defined on stochastic lattices and application to nonlinear elasticity.
	\newblock \emph{Arch. Ration. Mech. Anal.}, \textbf{200} (2011), 881--943.
	
	\bibitem{ACR}
	\newblock R. Alicandro, M. Cicalese and M. Ruf:
	\newblock Domain formation in magnetic polymer composites: an approach via stochastic homogenization.
	\newblock \emph{Arch. Ration. Mech. Anal.}, \textbf{218} (2015), 945--984.
	
	\bibitem{ACS}
	\newblock R. Alicandro, M. Cicalese and L. Sigalotti:
	\newblock Phase transition in presence of surfactants: from discrete to continuum.
	\newblock \emph{Interfaces Free Bound.}, \textbf{14} (2012), 65--103.

	\bibitem{AlGe}
	\newblock R. Alicandro and M.S. Gelli:
	\newblock Local and non local continuum limits of Ising type energies for spin systems.
	\newblock {\em SIAM J. Math. Anal.}, {\bf 48} (2016), 895--931.
	
	\bibitem{AmBrI}
	L.~Ambrosio and A.~Braides:
	\newblock Functionals defined on partitions of sets of finite perimeter I: integral representation and $\Gamma$-convergence.
	\newblock {\em J. Math. Pures. Appl.}, \textbf{69} (1990), 285--305.
	
	\bibitem{AmBrII}
	L.~Ambrosio and A.~Braides:
	\newblock Functionals defined on partitions of sets of finite perimeter II: semicontinuity, relaxation and homogenization.
	\newblock {\em J. Math. Pures. Appl.}, \textbf{69} (1990), 307--333.
	
	\bibitem{AFP} 
	\newblock L. Ambrosio, N. Fusco and D. Pallara:
	\newblock \emph{Functions of Bounded Variation and Free Discontinuity Problems}.
	\newblock Oxford Mathematical Monographs, Oxford University Press, New York, 2000.
	
	
	\bibitem{BLBL}
	\newblock X. Blanc, C. Le Bris and P.-L. Lions:
	\newblock The energy of some microscopic stochastic lattices.
	\newblock \emph{Arch. Ration. Mech. Anal.} \textbf{184} (2007), 303--339.
	
	\bibitem{BFLM}
	G. Bouchitt\'e, I. Fonseca, G. Leoni, and L. Mascarenhas:
	\newblock A global method for relaxation in ${\MakeUppercase w}^{1,p}$ and in
	${\MakeUppercase {sbv}}_p$.
	\newblock {\em Arch. Ration. Mech. Anal.}, \textbf{165} (2002), 187--242.
	
	\bibitem{GCB} 
	\newblock A. Braides:
	\newblock \emph{$\Gamma$-convergence for Beginners}.
	\newblock Oxford Lecture Series in Mathematics and its Applications 22, Oxford University Press, Oxford, 2002.
	
	\bibitem{BrCaSo} 
	A.~Braides, A.~Causin and M.~Solci:
	\newblock Interfacial energies on quasicrystals.
	\newblock \emph{IMA J. Appl. Math.}, \textbf{77} (2012), 816--836.
	
	\bibitem{BC}
	A.~Braides and M.~Cicalese:
	\newblock Interfaces, modulated phases and textures in lattice systems.
	\newblock \emph{Arch. Ration. Mech. Anal.}, to appear (2016).
	
	\bibitem{BrCoGa}
	A.~Braides, S.~Conti and A.~Garroni:
	\newblock Density of polyhedral partitions.
	\newblock {\em Calc. Var. Partial Diff. Eq.}, to appear
	
	\bibitem{BrDe} 
	A.~Braides and A.~Defranceschi:
	\newblock \emph{Homogenization of Multiple Integrals}.
	\newblock Oxford Lecture Series in Mathematics and its Applications 12, Oxford University Press, New York, 1998.
	
		\bibitem{BFF}
	A.~Braides, G.~Francfort and I.~Fonseca:
	\newblock 3D-2D asymptotic analysis for inhomogeneous thin films.
	\newblock {\em Indiana Univ. Math. J.}, {\bf 49} (2000), 1367--1404

	\bibitem{Dilute}
	A.~Braides and A.~Piatnitski:
	\newblock Variational problems with percolation: dilute spin systems at zero temperature.
	\newblock \emph{J. Stat. Phys.}, \textbf{149} (2012), 846--864.
	
	\bibitem{BP}
	A.~Braides and A.~Piatnitski:
	\newblock Homogenization of surface and length energies for spin systems.
	\newblock {\em J. Funct. Anal.}, {\bf 264} (2013), 1296--1328.
	
	
	\bibitem{CadlF}
	\newblock L.A. Caffarelli and R. de la Llave:
	\newblock Interfaces of ground states in Ising models with periodic coefficients.
	\newblock \emph{J. Stat. Phys.}, \textbf{118} (2005), 687--719. 
	
	\bibitem{CoDiVa}
	\newblock M. Cozzi, S. Dipierro and E. Valdinoci:
	\newblock Planelike in long-range Ising models and connections with nonlocal minimal surfaces.
	\newblock Preprint (2016), {\rm arXiv} 1605.06187. 
	
	\bibitem{DMM} 
	\newblock G. Dal Maso and L. Modica:
	\newblock Nonlinear stochastic homogenization and ergodic theory.
	\newblock \emph{J. Reine. Ang. Math.}, \textbf{368} (1986), 28--42.
	
	\bibitem{Kestbook}
	\newblock H. Kesten:
	\newblock \emph{Percolation Theory for Mathematicians}.
	\newblock Birkh{\"a}user, Boston, 1982.


        \bibitem{KS} 
	\newblock M.J. Klein and R.J. Smith:
	\newblock Thin ferromagnetic films.
	\newblock \emph{Physical Review}, \textbf{81.3} (1951):378.
	
		\bibitem{LPS}
	G.~Lazzaroni, M.~Palombaro and A.~Schl\"omerkemper:
	\newblock A discrete to continuum analysis of dislocations in
nanowire heterostructures.
	\newblock {\em Commun. Math. Sci}, {\bf 13} (2015), 1105--1133

		\bibitem{LDR}
	H.~Le Dret and A.~Raoult:
	\newblock The nonlinear membrane model as variational limit of nonlinear
three-dimensional elasticity.
	\newblock {\em J. Math. Pures Appl.}, {\bf 74} (1995), 549--578.
	
	\bibitem{LDR2}
	H.~Le Dret and A.~Raoult:
	\newblock Hexagonal lattices with three-point interactions.
	\newblock Preprint (2015), available at {\tt hal.upmc.fr}
	
	\bibitem{ruelle}
	D.~Ruelle:
	\newblock {\em Statistical Mechanics. Rigorous results}.
	\newblock River Edge, NJ: World Scientific, Reprint of the 1989 edition.
	
	\bibitem{R16}
	M.~Ruf:
	\newblock On the continuity of functionals defined on partitions.
	\newblock Preprint (2016), {\rm arXiv} 1612.01781.
	
	\bibitem{Sch}
	B.~Schmidt:
	\newblock On the passage from atomic to continuum theory for thin films
	\newblock {\em Arch. Ration. Mech. Anal.}, {\bf 190} (2008), 1--55.
	
	\bibitem{TS15}
	T.~Schmidt:
	\newblock Strict interior approximation of sets of finite perimeter and
	functions of bounded variation.
	\newblock {\em Proc. Am. Math. Soc.}, {\bf 143} (2015), 2069--2084.
	
	\bibitem{vollath2013} D.~Vollath:
	\newblock {\em Nanoparticles-Nanocomposites-Nanomaterials: An Introduction for Beginners.}
	\newblock John Wiley \& Sons, 2013.
	
\end{thebibliography}
\end{document}